\documentclass[a4paper,11pt,twoside]{amsart}

\usepackage[english]{babel} 
\usepackage[utf8]{inputenc} 
\usepackage[T1]{fontenc} 
\usepackage{amssymb,amsmath,amsfonts,amsxtra,amsthm} 
\usepackage{stmaryrd} 
\usepackage{mathtools} 
\usepackage{mathrsfs} 
\usepackage{cases} 
\usepackage[all]{xy} 
\usepackage{enumitem}
\usepackage{geometry} 
\usepackage{lmodern} 
\usepackage{xargs} 
\usepackage{graphicx} 
\usepackage[svgnames,table]{xcolor} 
\usepackage[colorlinks=true,linkcolor=Blue, citecolor=Magenta, urlcolor=Green, pdftoolbar=true]{hyperref}

\newcommand{\ensemblenombre}[1]{\ensuremath{\mathbb{#1}}}
\newcommand{\N}{\ensemblenombre{N}}

\newcommand{\R}{\ensemblenombre{R}}

\newcommand{\abs}[1]{\ensuremath{\left\lvert#1\right\rvert}}
\newcommand{\norme}[1]{\ensuremath{\left\lVert#1\right\rVert}}
\newcommand{\enstq}[2]{\ensuremath{\left\{#1\mathrel{}\middle|\mathrel{}#2\right\}}}

\newcommand{\transp}[1]{\prescript{t}{}{#1}}

\newcommand{\intervalle}[4]{\ensuremath{\mathopen{#1}#2
		\mathclose{}\mathpunct{};#3
		\mathclose{#4}}}
\newcommand{\intervalleff}[2]{\intervalle{[}{#1}{#2}{]}}
\newcommand{\intervalleof}[2]{\intervalle{]}{#1}{#2}{]}}

\newcommand{\restreinta}{\ensuremath{\mathclose{}|\mathopen{}}}
\newcommand{\cc}{\mathcal{C}}

\newcommand{\RP}[1]{\ensuremath{\R\mathbf{P}^{#1}}}
\newcommand{\Sn}[1]{\mathbf{S}^{#1}}
\newcommand{\Snhat}[1]{\hat{\mathbf{S}}^{#1}}
\newcommand{\Hn}[1]{\mathbf{H}^{#1}}

\newcommand{\Diff}[2]{\mathopen{}\mathrm{D}_{#1}#2}

\newcommand{\Fitan}[1]{\ensuremath{\mathrm{T}#1}}

\newcommand{\PFitan}[1]{\ensuremath{\mathbf{P}(\Fitan{#1})}}

\newcommand{\Fiunitan}[1]{\ensuremath{{\mathrm{T}}^1{#1}}}

\newcommand{\D}{\mathcal{D}}

\newcommand{\GL}[1]{\ensuremath{\mathrm{GL}_{#1}\mathopen{(}\R\mathclose{)}}}

\newcommand{\PGL}[1]{\ensuremath{\mathrm{PGL}_{#1}\mathopen{(}\R\mathclose{)}}}
\newcommand{\SL}[1]{\ensuremath{\mathrm{SL}_{#1}\mathopen{(}\R\mathclose{)}}}

\newcommand{\PSL}[1]{\ensuremath{\mathrm{PSL}_{#1}\mathopen{(}\R\mathclose{)}}}

\newcommand{\slR}[1]{\ensuremath{\mathfrak{sl}_{#1}}}

\newcommand{\SO}[1]{\ensuremath{\mathrm{SO}\mathopen{(}#1\mathclose{)}}}

\newcommand{\PO}[1]{\ensuremath{\mathrm{PO}\mathopen{(}#1\mathclose{)}}}

\newcommand{\Heis}[1]{\ensuremath{\mathrm{Heis}\mathopen{(}#1\mathclose{)}}}

\DeclareMathOperator{\Ker}{Ker}
\DeclareMathOperator{\Vect}{Vect}

\DeclareMathOperator{\Aut}{Aut}
\DeclareMathOperator{\Int}{Int}

\DeclareMathOperator{\id}{id}

\DeclareMathOperator{\Stab}{Stab}

\DeclareMathOperator{\Diag}{Diag}
\DeclareMathOperator{\Cl}{Cl}

\newcommand{\Pmin}{\mathbf{P}_{min}}

\newcommand{\X}{\mathbf{X}}
\newcommand{\Xhat}{\hat{\mathbf{X}}}

\newcommand{\Exalpha}{\mathcal{E}_\alpha}
\newcommand{\Extildealpha}{\hat{\mathcal{E}}_\alpha}

\newcommand{\Exbeta}{\mathcal{E}_\beta}
\newcommand{\Extildebeta}{\hat{\mathcal{E}}_\beta}

\newcommand{\Sm}{\mathcal{S}}

\newcommand{\Lm}{\mathcal{L}}

\newcommand{\Lx}{\mathcal{L}_\X}

\newcommand{\Calpha}{\mathcal{C}_\alpha}
\newcommand{\Cbeta}{\mathcal{C}_\beta}
\newcommand{\Chatalpha}{\hat{\mathcal{C}}_\alpha}
\newcommand{\Chatbeta}{\hat{\mathcal{C}}_\beta}
\newcommand{\Salphabeta}{\mathcal{S}_{\alpha,\beta}}
\newcommand{\Shatalphabeta}{\hat{\mathcal{T}}_{\alpha,\beta}}
\newcommand{\Sbetaalpha}{\mathcal{S}_{\beta,\alpha}}
\newcommand{\Shatbetaalpha}{\hat{\mathcal{T}}_{\beta,\alpha}}

\newcommand{\Falphamoins}{\mathcal{C}^-}
\newcommand{\Halphabetamoins}{\mathcal{T}^-}
\newcommand{\Hbetaalphaplus}{\mathcal{T}^+}
\newcommand{\Fbetaplus}{\mathcal{C}^+}

\newcommand{\e}{\mathrm{e}}

\newcommand{\Bplus}{B_{\alpha\beta}^+}
\newcommand{\Bmoins}{B_{\alpha\beta}^-}

\numberwithin{equation}{section}

\numberwithin{figure}{section}
\AtBeginDocument{}
\addto\captionsenglish{}

\theoremstyle{definition}
\newtheorem{definition}{Definition}[section]
\newtheorem{propositiondefinition}[definition]{Proposition-Definition}

\theoremstyle{plain}

\newtheorem{theoremintro}{Theorem}

\newtheorem{corollary}[definition]{Corollary}

\newtheorem{lemma}[definition]{Lemma}
\newtheorem{proposition}[definition]{Proposition}

\newtheorem{fact}[definition]{Fact}
\theoremstyle{remark}
\newtheorem{example}[definition]{Example}

\newtheorem{remark}[definition]{Remark}

\newtheorem{questionintro}{Question}

\usepackage{tikz}

\title[Geometrical compactifications of geodesic flows
and path structures]
{Geometrical compactifications of geodesic flows
and path structures}
\author{Martin Mion-Mouton}
\thanks{The author is supported by a
postdoctoral fellowship of the Technion.
This paper was partly written during a stay at the
Institut Math\'ematiques de Jussieu, that
the author would like to thank.}
\date{\today}

\begin{document}
\address{
Martin Mion-Mouton,
Department of Mathematics,
Technion,
Haifa 32000,
Israel.
}
\email{martinm@campus.technion.ac.il}
\urladdr{https://martinm.webgr.technion.ac.il/}

\maketitle

\begin{abstract}
 In this paper, we construct a geometrical compactification 
 of the geodesic flow
 of non-compact complete hyperbolic surfaces $\Sigma$
 without cusps having finitely generated fundamental group.
 We study the dynamical properties of the compactified flow,
 for which we show the existence of attractive circles at infinity.
 The geometric structure of $\Fiunitan{\Sigma}$ for which this compactification is realized is the pair of one-dimensional distributions 
 tangent to the stable and unstable horocyles of $\Fiunitan{\Sigma}$.
 This is a Kleinian path structure,
 that is a quotient of an open subset 
 of the flag space by a discrete subgroup $\Gamma$ of $\PGL{3}$.
 Our study relies on 
 a detailed description of the dynamics of $\PGL{3}$ on the flag space,
 and on the construction of an explicit
 fundamental domain for the action of $\Gamma$ 
 on its maximal open subset of discontinuity in the flag space.
\end{abstract}

\section{Introduction}
The geodesic flow $(g^t)$
of a compact hyperbolic surface $\Sigma$ has a very nice and well-studied dynamical property:
it is an \emph{Anosov flow} of its
unitary tangent bundle $\Fiunitan{\Sigma}$.
This means that $\Diff{}{g^t}$ preserves a splitting
\[
\Fitan{(\Fiunitan{\Sigma})}=
E^s\oplus \R\frac{dg^t}{dt} \oplus E^u
\]
of the tangent bundle of $\Fiunitan{\Sigma}$,
where $\R\frac{dg^t}{dt}$ is the direction of the flow,
and $E^s$, $E^u$ are two one-dimensional distributions 
of $\Fiunitan{\Sigma}$
(respectively called the \emph{stable} and \emph{unstable} distributions of $(g^t)$)
that are respectively \emph{uniformly contracted} and 
\emph{uniformly expanded} by $\Diff{}{g^t}$.
More precisely, for any Riemannian metric on $\Fiunitan{\Sigma}$,
there exists two constants $C>0$ and $\lambda<1$ such that for any $x\in\Fiunitan{\Sigma}$ and $t>0$:
\begin{equation}\label{equationdefanosov}
 \norme{\Diff{x}{g^t}\restreinta_{E^s}}\leq C\lambda^t
 \text{~and~}
 \norme{\Diff{x}{g^{-t}}\restreinta_{E^u}}\leq C\lambda^t.
\end{equation}
Geodesic flows of compact hyperbolic surfaces are 
very specific among Anosov flows,
since their stable and unstable distributions are smooth (that is, $\cc^\infty$).
The sum $E^s\oplus E^u$ moreover happens to be a \emph{contact distribution} --
we say in this case that the flow is \emph{contact-Anosov}.
We recall that a $\cc^1$ plane distribution of a three-dimensional manifold
is called contact if it is locally the kernel of a one-form $\theta$ which is \emph{contact} ($\theta\wedge d\theta$ nowhere vanishes).
A pair $\Lm=(E^\alpha,E^\beta)$
of smooth one-dimensional
distributions whose sum is a contact distribution
defines on a three-dimensional manifold
a geometric structure called a \emph{path structure}.
Path structures, whose study goes back to \'Elie Cartan in \cite{cartan_sur_1924},
are \emph{rigid} geometries whose interplay with smooth dynamics has shown
to be very rich 
(see Paragraph \ref{soussectionpathstructuresnoncompact}
for more details about path structures and their rigidity,
and Paragraph 
\ref{soussoussectionautrescompactifications} for examples
explaining the geometrical origin of the terminology).
From a geometrical point of view,
we may thus look at the geodesic flow of a compact hyperbolic surface $\Sigma$
as a flow of automorphisms of the path structure
$\Lm_{\Sigma}=(E^s,E^u)$ on $\Fiunitan{\Sigma}$.
This point of view allowed for instance Ghys to classify in \cite{ghys} the
three-dimensional contact-Anosov flows having smooth stable and unstable distributions (see Paragraph \ref{soussectionpathstructuresnoncompact}
for more details).
\par One can ask what remains of this beautiful geometrico-dynamical picture
for a \emph{non-compact} complete hyperbolic surface $\Sigma$.
The $(g^t)$-invariant path structure $\Lm_\Sigma$ persists
in this case, and one of the motivation of this paper
is to provide with a geometrico-dynamical compactification of both
the structure $\Lm_\Sigma$ and the flow $(g^t)$ on $\Fiunitan{\Sigma}$,
and to describe the dynamics of the compactified geodesic flow obtained in this way.

\subsection{Geometrical compactification of the geodesic flow}
The geometric picture that we described is actually independent of the compactness
of $\Sigma$.
Indeed for any complete hyperbolic surface $\Sigma$,
there exists on $\Fiunitan{\Sigma}$ a natural path structure 
$\Lm_\Sigma=(E^s,E^u)$ invariant by the geodesic flow $(g^t)$ --
and equal to the pair of stable and unstable distributions of $(g^t)$
if $\Sigma$ is compact
(see Paragraph \ref{soussectionstructureLCaudessussurfaceshyperboliques}
for a proper definition of $\Lm_\Sigma$).
The hyperbolic metric of $\Sigma$ induces on $\Fiunitan{\Sigma}$
a ``most natural'' Riemannian metric (invariant by the lifts of isometries of $\Sigma$)
called the \emph{Sasaki metric},
with respect to which 
$(g^t)$ indeed satisfies the Anosov conditions \eqref{equationdefanosov}
on the distributions $E^s$ and $E^u$ of the path structure $\Lm_\Sigma$
-- wether $\Sigma$ is compact or not.
In this regard,
one may thus say that $(g^t)$ is ``\emph{Anosov for the Sasaki metric}'',
and that the path structure $\Lm_\Sigma$ that we are studying
is the pair of stable and unstable distributions of $(g^t)$.
An important distinction to be made
is however that in the non-compact case,
the existence of the inequalities \eqref{equationdefanosov} 
will critically depend on the chosen Riemannian metric,
since $\Fiunitan{\Sigma}$ is non-compact
(whereas only constants will differ if $\Sigma$ is compact).
The Anosov property is thus in the non-compact case
not an intrinsic property of the \emph{flow} $(g^t)$ itself
but only of the \emph{pair} ($(g^t)$, Sasaki metric).
For this reason, we would like to study $(g^t)$ as acting on
an open subset of a closed three-manifold.
More precisely,
we would like to find a \emph{closed} three-manifold $M$ together
with a flow $(\varphi^t)$ on $M$,
containing a $(\varphi^t)$-invariant open subset $N\subset M$
such that $(\varphi^t\restreinta_N)$ 
is conjugated to $(g^t)$. 
In this case, we will say that $(M,\varphi^t)$ is a 
(dynamical) \emph{compactification} of $(\Fiunitan{\Sigma},g^t)$.
In general, it is not clear if a given flow acting on an open manifold 
can be compactified in that way;
but if it does, then there are certainly a lot of possible compactifications,
and one would like to choose one that 
has interesting properties with respect to the flow.
To begin with, we would like to preserve --
as far as possible --
any information that we already have about this flow.
\par In our case we do have an additional geometrical information,
the $(g^t)$-invariant path structure $\Lm_\Sigma$ on $\Fiunitan{\Sigma}$,
that we would like to preserve
in order to stay as close as possible to the Anosov behaviour.
In other words, what we want is a path structure $\Lm=(E^\alpha,E^\beta)$ on $M$
and an open subset $N\subset M$,
such that $(N,\Lm\restreinta_N)$ is isomorphic to $(\Fiunitan{\Sigma},\Lm_\Sigma)$
-- an isomorphism of path structures 
being simply a diffeomorphism
sending $E^\alpha$ (respectively $E^\beta$) on $E^s$ (resp. $E^u$).
In this case, we will say that $(M,\Lm)$ is a
(path structure) \emph{compactification} of $(\Fiunitan{\Sigma},\Lm_\Sigma)$.
If we moreover ask the dynamics and the geometry to be compatible,
then we look for 
a \emph{closed} three-manifold $M$ endowed with a path structure $\Lm$
and with a flow $(\varphi^t)$ of automorphisms of $\Lm$,
such that there exists a $(\varphi^t)$-invariant open subset $N\subset M$
and an isomorphism $\Phi$ from $(\Fiunitan{\Sigma},\Lm_\Sigma)$ to $(N,\Lm\restreinta_N)$
conjugating $(g^t)$ and $(\varphi^t\restreinta_N)$.
In this case, we will say that the points 
$x\in M$ and $\Phi(x)\in N$ are \emph{equivalent},
that $(N,\Lm_N,\varphi^t\restreinta_N)$ is \emph{a copy of}
$(\Fiunitan{\Sigma},\Lm_\Sigma,g^t)$,
and that $(M,\Lm,\varphi^t)$ is a \emph{geometrico-dynamical compactification} 
of $(\Fiunitan{\Sigma},\Lm_\Sigma,g^t)$.
\par The following result applies 
to any hyperbolic surface which is \emph{uniformized by a Schottky
subgroup with sufficiently large generators} in the following meaning:
for any hyperbolic elements $h_1,\dots,h_d$ of $\PSL{2}$ having pairwise distincts
fixed points on the boundary of the hyperbolic plane $\Hn{2}$,
and for any sufficiently large $r_i>0$,
the statement applies to the quotient of $\Hn{2}$
by the discrete subgroup of $\PSL{2}$ generated by $h_1^{r_1},\dots,h_d^{r_d}$.
\begin{theoremintro}\label{theoremintrosurfacehyperboliquedynamiqueflotgeod}
For any hyperbolic surface $\Sigma$ uniformized by a Schottky
subgroup of $\PSL{2}$ with sufficiently large generators, we have the following.
 \begin{enumerate}
  \item $(\Fiunitan{\Sigma},\Lm_\Sigma,g^t)$ admits a 
  geometrico-dynamical compactification
  $(M,\Lm,\varphi^t)$, containing four disjoint 
  copies $\{N_i\}_{i=1}^4$ of $(\Fiunitan{\Sigma},\Lm_\Sigma,g^t)$
  and such that $M\setminus\cup_{i=1}^4 N_i$ is a finite union of tori.
  \item The set of fixed points of $(\varphi^t)$ can be decomposed as a disjoint union
  $\Falphamoins\cup\Delta\cup\Fbetaplus$, 
  each of these subsets
  being a finite union of circles.
  The subset $\mathcal{W}^+$ (respectively $\mathcal{W}^-$)
  of points of $\Fiunitan{\Sigma}$ whose positive (resp. negative)
  $g^t$-orbit goes to infinity (that is, escapes from any compact subset of $\Fiunitan{\Sigma}$)
  is open and dense in $\Fiunitan{\Sigma}$.
  Furthermore for any $x\in\mathcal{W}^+$ (resp. $x\in\mathcal{W}^-$),
  with $x_i$ the equivalent point in any of the copies $N_i$,
   $\varphi^t(x_i)$ converges to a point of $\Fbetaplus$ 
  (resp. $\varphi^{-t}(x_i)$ converges to a point of $\Falphamoins$) when $t\to+\infty$.
  More precisely, compact subsets of $\mathcal{W}^\pm$ are attracted to
  $\mathcal{C}^\pm$ under $\varphi^{\pm t}$, and $(\varphi^t)$ is non-conservative.
\item Denoting by $E^c$ the direction of $(\varphi^t)$ in $M\setminus\cup_{i=1}^4 N_i$
and by $\Lm=(E^\alpha,E^\beta)$ the path structure,
$\varphi^t$ has exponential growth rates respectively equal to $-\frac{1}{2}$, $-\frac{1}{2}$ and $0$
in the directions $E^c$, $E^\alpha$ and $E^\beta$ along any positive orbit in $\mathcal{W}^+$
(resp. equal to $\frac{1}{2}$, $0$ and $\frac{1}{2}$ along any negative orbit in $\mathcal{W}^-$).
 \end{enumerate}
\end{theoremintro}
We refer to Paragraph \ref{soussoussectiondynamiqueflotcompactifie}
for precisions about the notion of exponential growth rates that we use here.
\begin{remark}\label{remarkquellessurfaces}
All the hyperbolic surfaces considered in
Theorem \ref{theoremintrosurfacehyperboliquedynamiqueflotgeod}
are non-compact complete hyperbolic surfaces of infinite volume
without cusps ($\Sigma$ only has funnels),
and with finitely generated fundamental group.
Moreover,
for any connected non-compact topological surface $S$ 
with finitely generated fundamental group, the set of hyperbolic
metrics $g$ on $S$ for which Theorem \ref{theoremintrosurfacehyperboliquedynamiqueflotgeod}
applies to $\Sigma=(S,g)$ is open and non-empty.
\end{remark}
Theorem \ref{theoremintrosurfacehyperboliquedynamiqueflotgeod}
will be proved in section \ref{sectionTheoremeB}.
More precisely, refined versions of the three claims are respectively proved
in Propositions \ref{propositioncompactificationetflot},
\ref{propositionproprietesprolongationflotgeod} 
(and \ref{corollarydynamique} for the non-conservative behaviour), 
\ref{propositionLyapunovcompactification}.

\subsubsection{About unicity of compactifications}
\label{subssubsectionsurpisingproperties}
It may seem surprising that the compactification $(M,\Lm,\varphi^t)$
of the geodesic flow given by 
Theorem \ref{theoremintrosurfacehyperboliquedynamiqueflotgeod}
contains \emph{four} copies of $(\Fiunitan{\Sigma},\Lm_\Sigma, g^t)$,
and one may ask if there exists a ``smaller'' compactification.
In particular, it is natural to ask:
\begin{questionintro}\label{questionpascompactificationavecimagedense}
Does there exist a geometrico-dynamical compactification $(M,\Lm,\varphi^t)$
containing a dense copy of $(\Fiunitan{\Sigma},\Lm_{\Sigma},g^t)$ ?
Or a path structure compactification $(M,\Lm)$
containing a dense copy of $(\Fiunitan{\Sigma},\Lm_{\Sigma})$ ?
\end{questionintro}

We will discuss in Paragraph \ref{soussectioncompactificationsKleinian} below
a partial answer to this question when restricted to \emph{Kleinian}
compactifications.
Another surprising property of the compactification given by
Theorem \ref{theoremintrosurfacehyperboliquedynamiqueflotgeod}
is the existence of circles of fixed points for the compactified 
geodesic flow $(\varphi^t)$.
This raises the following second question, intimately linked to the previous one.
\begin{questionintro}\label{questionpointsfixesflotcompactifie}
 Does there exist a geometrico-dynamical compactification $(M,\Lm,\varphi^t)$
of $(\Fiunitan{\Sigma},\Lm_{\Sigma},g^t)$ where $(\varphi^t)$ has no fixed point ?
Or, at least, where all its fixed points are isolated ?
\end{questionintro}

\subsubsection{Other geometrical compactifications of $\Fiunitan{\Sigma}$}
\label{soussoussectionautrescompactifications}
The unitary tangent bundle of a hyperbolic surface $\Sigma$
actually bears different geometric structures,
and it is interesting to compare the compactification obtained in Theorem
\ref{theoremintrosurfacehyperboliquedynamiqueflotgeod} for the path structure $\Lm_\Sigma$
with those obtained for other structures.
First of all, a path structure is associated to any Riemannian surface $S$ 
in the following way.
The set of geodesics of $S$
defines on $\Fiunitan{S}$ a one-dimensional distribution $E^\beta$
tangent to the lifts of the geodesics in $\Fiunitan{S}$.
Denoting by $E^\alpha$ the tangent direction 
the fibers of the canonical projection $\pi\colon\Fiunitan{S}\to S$, 
$E^\alpha\oplus E^\beta$ is then a contact distribution.
In other words the unitary tangent bundle $\Fiunitan{S}$ of any Riemannian surface
is naturally endowed with
a path structure $\Lm_S^{proj}=(E^\alpha,E^\beta)$,
wether $S$ is hyperbolic or not.
These classical examples explain the geometrical origin of the terminology
\emph{path structure}.
\par In the specific case of a complete hyperbolic surface $\Sigma$,
we thus have \emph{two different} path structures $\Lm_\Sigma$ and $\Lm_\Sigma^{proj}$
on $\Fiunitan{\Sigma}$.
The main difference between those two structures is that 
\emph{$\Lm_\Sigma^{proj}$ is not invariant by the geodesic flow $(g^t)$}.
Indeed, $E^\beta$ is $(g^t)$-invariant by definition
(as it is tangent to the orbits of $(g^t)$),
but a fiber of $\pi$ is not sent by $g^t$ to another fiber
and $E^\alpha$ is thus not $(g^t)$-invariant.
We will explain in Paragraph \ref{soussectionautrescompactifications}
how the work of Choi-Goldman in \cite{choi_topological_2017}
gives a compactification of $\Lm_\Sigma^{proj}$,
and we will describe the conformal compactification given by
\cite{frances_sur_2005} of a $g^t$-invariant Lorentzian metric
on $\Fiunitan{\Sigma}$.

\subsection{Path structures with non-compact automorphism groups
and partially hyperbolic diffeomorphisms}
\label{soussectionpathstructuresnoncompact}
The initial motivation of \'Elie Cartan for the study of path structures 
in \cite{cartan_sur_1924}
was to find a geometrical object that parametrizes the space of local solutions
of second-order scalar ordinary differential equations
(see for instance \cite[\S 8.6]{ivey_cartan_2016} for an explanation of this link),
and to describe the local invariants of such an ODE
through a notion of \emph{curvature} of path structures.
Path structures are nowadays studied
in a geometric setting called \emph{parabolic Cartan geometries}
and are sometimes
called \emph{Lagrangian-contact structures} in this context
(see for instance \cite[\S 4.2.3]{capslovak}, \cite{takeuchi}).
The author actually used the denomination 
Lagrangian-contact structure
in \cite{mionmouton}
before deciding
to stick to the name \emph{path structure}
which seems more geometrically meaningful
and closer to the initial motivation of Cartan.
We apologize in advance
for any confusion that this change 
of name 
could lead to.
\par Apart from the intrinsic interest of compactification of geodesic flows,
a second important motivation of this paper 
was to deduce from Theorem \ref{theoremintrosurfacehyperboliquedynamiqueflotgeod}
new and rich examples of path structures
having non-compact automorphism groups for the compact-open topology. 
The interest of such examples appears in contrast with former
rigidity results for path structures
that we now describe.

\subsubsection{Hierarchy of path structures}
\label{soussousectionrigidite}
Ghys used the path structures 
to prove in \cite{ghys} that 
the geodesic flows of compact hyperbolic surfaces are the only
three-dimensional contact-Anosov flows having
smooth stable and unstable distributions,
up to finite coverings and smooth orbit-equivalence
(path structures appear in \cite{ghys} through the point of view
of second order ODE mentionned previously).
In fact such a contact-Anosov flow $(\varphi^t)$ preserves more than
the path structure $(E^s,E^u)$ defined by its stable and unstable distributions.
The contact form $\theta$ defined by
$\theta(\frac{d \varphi^t}{dt})\equiv 1$ and $\theta\restreinta_{E^s\oplus E^u}\equiv 0$
is indeed $(\varphi^t)$-invariant,
and $(\varphi^t)$ preserves thus the triplet $(E^s,E^u,\theta)$.
This is a special instance of a 
\emph{strict path structure} $\mathcal{T}=(E^\alpha,E^\beta,\theta)$,
with $(E^\alpha,E^\beta)$ a path structure and $\theta$ a contact form of kernel 
$E^\alpha\oplus E^\beta$.
From a geometrical point of view, Ghys result corresponds thus to classify the compact
strict path structures whose Reeb flow is Anosov.
In \cite{falbel_cartan_2021} we generalize this result
with Elisha Falbel and Jose Miguel Veloso
by considering the three-dimensional
compact strict path structures $(M,\mathcal{T})$
having a non-compact automorphism group
and a dense $\Aut^{loc}$-orbit. 
We prove in this setting that up to finite coverings,
$(M,\mathcal{T})$ is either the structure preserved by
the geodesic flow of a compact hyperbolic surface,
or a left-invariant structure
on a compact quotient of the Heisenberg group $\Heis{3}$.
\par One can now forget about the contact form $\theta$ to keep only a 
third direction $E^c$ transverse to the contact distribution $E^\alpha\oplus E^\beta$
and consider the triplet
$\Sm=(E^\alpha,E^\beta,E^c)$ that we will call an \emph{enhanced path structure}
(this can be considered as the ``conformal version'' of a strict path structure,
the Reeb vector field of the contact form being weakened to a line field).
We can then ask the same question:
what are the three-dimensional compact enhanced path structures $(M,\Sm)$
having a non-compact automorphism group ?
In \cite{mionmouton}
we obtain a first classification result in this direction,
assuming that
an automorphism $f$ of $\Sm$ without wandering points uniformly contracts or expands both
$E^\alpha$ and $E^\beta$.
In this case we prove that $(M,\Sm)$
still belongs to one of the two families of algebraic examples previously mentionned
(geodesic flows or compact quotients of $\Heis{3}$)
which yields a rigidity result about
partially hyperbolic diffeomorphisms, see \cite[Theorem A]{mionmouton}.
In particular, there exists \emph{a posteriori}
a contact form $\theta$ such that
any automorphism of the enhanced path structure $\Sm$ is in fact an automorphism
of the \emph{strict} path structure $(E^\alpha,E^\beta,\theta)$.
\par The next step would be to
forget the transverse direction $E^c$ and to investigate the
three-dimensional compact path structures $\Lm=(E^\alpha,E^\beta)$ 
having a non-compact automorphism group.
Until now, we saw two kinds of examples of
automorphisms of path structures generating a non-compact subgroup,
which we will call \emph{non-equicontinuous automorphisms}: 
the first are time-one maps of geodesic flows of compact hyperbolic surfaces,
and the second are automorphisms of compact quotients of $\Heis{3}$.
In particular, all of these examples are partially hyperbolic diffeomorphisms
(see \cite[Theorem A]{mionmouton})
and are \emph{conservative} (they preserve a volume form).
Other examples are easily constructed in the following way.
Take $\varphi$ an automorphism of $\Heis{3}$ which is diagonal and expands
the three directions in an affine chart of $\Heis{3}$.
Then in the same way as classical Hopf tori, 
the quotient $\langle \varphi \rangle\backslash\Heis{3}$ is compact
and bears a path structure $\Lm$ having non-conservative automorphisms
(see \cite[p.24]{alexandre_nilclosed_2021} for more details, 
and for links with completeness results about flat strict path structures).
These examples have a simple geometrical and topological description,
all of them are homeomorphic to $\Sn{1}\times\Sn{2}$.
Moreover, 
all their automorphisms preserve not only the path structure $\Lm$
but also a direction $E^c$ transverse to $\Lm$.

\subsubsection{New essential path structures}
These last examples 
suggest that, in order to find ``more complicated'' examples of path structures 
$\Lm$, one should look for automorphisms of $\Lm$ that do not
preserve any smooth one-dimensional distribution transverse
to the contact distribution of $\Lm$ --
or in other words, that do not preserve any 
enhanced path structure compatible with $\Lm$.
We will 
say that an automorphism flow is
\emph{strongly essential} if it does not preserve any
continuous distribution transverse to $\Lm$.
This notion is reminiscent of the
\emph{essential Lorentzian conformal
structures}, whose conformal automorphism group
is the isometry group of no metric in the conformal class.
These structures are studied in \cite{frances_sur_2005}, where  
essential Lorentzian conformal structures distinct from
the Einstein universe are constructed.
One of the motivation of this paper
is to provide with the following large family of new examples of
compact path structures
with 
essential automorphisms. 
\begin{theoremintro}\label{corollaryessential}
Let $(M,\Lm,\varphi^t)$ be the compactification
of $(\Sigma,\Lm_\Sigma,g^t)$ described in Theorem
\ref{theoremintrosurfacehyperboliquedynamiqueflotgeod}.
Then for any $t\neq 0$, $\varphi^t$ is non-equicontinuous,
non-conservative and not partially hyperbolic.
Furthermore, $(\varphi^t)$ is a strongly essential automorphism flow.
\end{theoremintro}
These claims are proved in
Propositions \ref{corollarydynamique}
and \ref{propositionLmessentielle}.

\subsection{Compactifications of Kleinian path structures}
\label{soussectioncompactificationsKleinian}
So far, we considered a non-compact complete hyperbolic surface $\Sigma$
together with the path stucture $\Lm_\Sigma$ of $\Fiunitan{\Sigma}$
invariant by its geodesic flow $(g^t)$,
and we described the dynamical properties of 
the compactified geodesic flow with respect to $\Lm_\Sigma$.
We now explain the geometric origin of this compactification.
Recall from Theorem \ref{theoremintrosurfacehyperboliquedynamiqueflotgeod}
that we consider a surface $\Sigma$ which
is the quotient of $\Hn{2}$ by a discrete free subgroup
$\overline{\Gamma}_0$ of $\PSL{2}$ generated by $d$ hyperbolic elements $\overline{h}_i$
of $\PSL{2}$ having pairwise distinct fixed points on $\partial\Hn{2}$.
We choose for each $i$ a lift $h_i\in\SL{2}$ of $\overline{h}_i$ with positive eigenvalues,
and we consider the subgroup 
\[
 \Gamma_0=\langle h_1,\dots,h_d \rangle
\]
of $\SL{2}$ generated by the $h_i$.
It is known that $\Fiunitan{\Sigma}$ is identified with 
$\overline{\Gamma}_0\backslash\PSL{2}$, and 
we thus have a two-sheeted covering
from $\Gamma_0\backslash\SL{2}$ to $\Fiunitan{\Sigma}$.
As we will see in Lemma \ref{lemmaLSigmainvarianteflotgeodesique},
the pullback of the path structure $\Lm_\Sigma$ of $\Fiunitan{\Sigma}$
by this covering comes
from a left-invariant path structure $\Lm_{\SL{2}}$ on $\SL{2}$. 
It turns out that $(\SL{2},\Lm_{\SL{2}})$ can be embedded in a global
\emph{homogeneous model space}, central in the study of path structures.
\par This is the \emph{flag space} of dimension three that we denote by $\X$,
defined as 
\[
 \X=\enstq{(p,D)\in\RP{2}\times\RP{2}_*}{p\in D}
\]
where $\RP{2}$ (respectively $\RP{2}_*$) denotes the projective plane
(resp. the space of projective lines of $\RP{2}$).
$\X$ admits two natural projections $\pi_\alpha$ and $\pi_\beta$,
respectively on $\RP{2}$ and $\RP{2}_*$, which are the restrictions
to $\X$ of the first and second coordinate projections, and
whose fibers 
define two transverse 
foliations of $\X$ by circles that we respectively call 
$\alpha$ and \emph{$\beta$-circles}.
Denoting respectively by $\Exalpha$ and $\Exbeta$
the distributions tangent to these foliations,
$\Exalpha\oplus\Exbeta$ is a contact distribution and $\Lx=(\Exalpha,\Exbeta)$
is thus a path structure on $\X$.
Note that there is a natural identification of $\X$ with $\PFitan{\RP{2}}$,
for which
$\Lx$ is simply the structure $\Lm_{\RP{2}}^{proj}$
that we defined in Paragraph \ref{soussoussectionsecondepathstructure},
considering the projective lines as the geodesics of $\RP{2}$.
\par The 
natural action of $\PGL{3}$ on $\X$ 
does not only preserve $\Lx$, but is equal to its whole
automorphism group.
Note that the action of $\PGL{3}$ is transitive and identifies thus $\X$ with the homogeneous space $\PGL{3}/\Pmin$,
with $\Pmin=\Stab([e_1],[e_1,e_2])$ the subgroup of upper-triangular
matrices -- which is a (minimal) parabolic subgroup of $\PGL{3}$.
When embedded in $\PGL{3}$ through
\begin{equation*}\label{equationembeddingj}
 j\colon
 A\in\SL{2}\mapsto
 \begin{bmatrix}
  A & 0 \\
  0 & 1
 \end{bmatrix}\in\PGL{3},
\end{equation*}
$\SL{2}$ has an unique open orbit on $\X$ that we denote by $Y$,
for which $j(\Gamma_0)\backslash Y$ is isomorphic to $\Gamma_0\backslash\SL{2}$
(see Lemma \ref{lemmaLSigmainvarianteflotgeodesique}).
Here the second quotient is endowed with the path structure induced by $\Lm_{\SL{2}}$,
and $j(\Gamma_0)\backslash Y$ with the one induced by $\Lx$ 
-- this is a special instance of \emph{Kleinian} path structures, that are
quotients of open subsets of $\X$ by discrete subgroups of $\PGL{3}$.
\begin{theoremintro}\label{corollarycompactificationtopologie}
Up to replacing each generator $h_i$ by a large enough finite iterate
$h_i^{r_i}$,
the Kleinian structure $j(\Gamma_0)\backslash Y$ 
admits a Kleinian compactification $\Gamma\backslash\Omega$
where it embedds as an open and dense subset.
Moreover, $\Gamma\backslash\Omega$ is homeomorphic 
to a closed three-manifold obtained from the flag space $\X$ 
after performing $d$ times a topological surgery
described by the folllowing two operations.
 \begin{enumerate}[label=(\Alph*)]
  \item Remove the interior of two disjoint embedded genus two
  handlebodies $H^-$ and $H^+$.
  \item Glue the two boundary components 
  $\partial H^-$ and $\partial H^+$ of the resulting
  three-manifold with boundary,
  by some diffeomorphism between $\partial H^-$ and $\partial H^+$.
 \end{enumerate}
\end{theoremintro}
We emphasize that the gluing diffeomorphisms that appear are extremely specific
(they arise from elements of $\PGL{3}$),
and that we actually expect the topology of these surgeries to be
highly constrained (it is likely that this topology does only depend 
on the number of generators of $\Gamma_0$).
After this result and regarding the above Question 
\ref{questionpascompactificationavecimagedense}
in Paragraph \ref{subssubsectionsurpisingproperties}, 
it seems even more surprising
that the compactification of $(\Fiunitan{\Sigma},\Lm_\Sigma)$ in Theorem
\ref{theoremintrosurfacehyperboliquedynamiqueflotgeod} contains four copies
of $(\Fiunitan{\Sigma},\Lm_\Sigma)$,
while its two-sheeted covering $\Gamma_0\backslash\SL{2}$
admits a compactification with a dense copy.
One can actually prove that
the answer to Question 
\ref{questionpascompactificationavecimagedense} is negative if
the path structure compactification $(M,\Lm)$ is assumed to be Kleinian.
However there is \emph{a priori} no reason for
a path structure compactification of $(\Fiunitan{\Sigma},\Lm_{\Sigma})$
to be Kleinian, and 
obtaining a complete answer to Question \ref{questionpascompactificationavecimagedense}
is thus much more difficult than handling the specific case of Kleinian compactifications.
\par Theorem \ref{corollarycompactificationtopologie} is proved
in Proposition \ref{propositioncompactificationSL2surGamma},
and will be a direct consequence of a more general result
that we now present.

\subsubsection{Fundamental domains for Schottky subgroups}
We will call \emph{loxodromic} any diagonalizable element
of $\PGL{3}$ having three eigenvalues of distinct absolute values.
A loxodromic element $g\in\PGL{3}$ acts particularly nicely on the flag space:
there exists a repulsive bouquet of two circles $\Bmoins(g)\subset\X$ and
an attractive bouquet of two circles $\Bplus(g)\subset\X$
with respect to which the dynamics of $(g^n)$ are of ``north-south type'',
meaning that any compact subset of $\X\setminus \Bmoins(g)$ converges to $\Bplus(g)$ 
under the action of $(g^n)$
(see Example \ref{examplebouquetloxodromic} for more details).
For any $g$, $B_{\alpha\beta}^\pm(g)$ is the bouquet
of $\alpha$ and $\beta$-circles of a point $x^\pm\in\X$, 
and these $g$-invariant bouquets of circles play for $g$
the role of the attractive and repulsive fixed points in $\partial\Hn{2}$
of an hyperbolic element of $\PSL{2}$.
A natural analog to the classical definition of Schottky subgroups of $\PSL{2}$ is then the following.
\begin{definition}\label{definitionSchottky}
 The group $\Gamma$ generated by loxodromic elements 
$g_1,\dots,g_d\in\PGL{3}$ 
is a \emph{Schottky subgroup}
if there exists a \emph{set of separating handlebodies} 
$\{H^-_i,H^+_i\}_{i=1}^d$
for the $g_i$, where the
$H_i^\pm$ are pairwise disjoint compact neighbourhoods
of the $B_{\alpha\beta}^\pm(g_i)$ in $\X$ that are genus two handlebodies,
such that 
$H^+_i=\X\setminus\Int(g_i(H^-_i))$ for any $i$.
\end{definition}
In particular
Schottky subgroups of $\PGL{3}$ are free groups,
and as in $\PSL{2}$
we will see in Proposition \ref{propdefSchottkydansX} that
for any loxodromic elements $g_1,\dots,g_d\in\PGL{3}$ 
\emph{in general position}, that is 
whose bouquet of circles $B_{\alpha\beta}^\pm(g_i)$ are pairwise disjoint,
$\Gamma=\langle g_1,\dots,g_d \rangle$ is a Schottky subgroup
up to replacing each $g_i$ by $g_i^{r_i}$ for $r_i>0$ large enough.
Any sequence of $\PGL{3}$ eventually escaping from any compact
has a subsequence going \emph{simply} to infinity
in a sense defined in Definition \ref{definitionsimplyinfinity}.
Among those, the sequences $(\gamma_n)$ of \emph{balanced type} 
(see Definition \ref{definitiontypesdynamiques})
have on $\X$ the same kind of north-south dynamics
than the iterates of a loxodromic element,
with respect to a repulsive and an attractive bouquet of two circles 
$B_{\alpha\beta}^\pm(\gamma_n)$
(see Lemma \ref{lemmecasmixteX} for more details).
\begin{theoremintro}\label{theoremintropositiongenerale}
 Let $\Gamma$ be a Schottky subgroup of $\PGL{3}$ with $d$ generators.
 \begin{enumerate}
  \item Any sequence $\gamma_n\in\Gamma$ going simply to infinity is of balanced type,
  and $\Gamma$ acts freely, properly and cocompactly on the open 
  subset:
  \[
  \Omega(\Gamma)=\X\setminus\underset{
  \underset{\gamma_n\underset{simply}{\longrightarrow}\infty}{\gamma_n\in\Gamma}
  }{\bigcup}
  \Bplus(\gamma_n).
  \]
  \item With $\{H^-_i,H^+_i\}_{i=1,\dots,d}$ a set of separating handlebodies for the $g_i$,
  $\X\setminus\bigcup_i(H^-_i\cup H^+_i)$ is a fundamental domain
  for the action of $\Gamma$ on $\Omega(\Gamma)$.
  \item The topology of $\Gamma\backslash\Omega(\Gamma)$ 
  is obtained from the flag space $\X$ after performing $d$
  times the surgery described in Theorem \ref{corollarycompactificationtopologie}.
   \item In the case of $d=1$ loxodromic element with positive
   eigenvalues,
   $\Omega(\Gamma)$ is the complement in $\X$ of two disjoint
   bouquet of circles, and
   $\Gamma\backslash\Omega(\Gamma)$ 
  is homeomorphic to the product of the circle
  with the closed connected and orientable surface of genus two.
 \end{enumerate}
\end{theoremintro}
The existence of the open subset $\Omega(\Gamma)\subset\X$ 
with proper and cocompact action
of $\Gamma$ is a consequence of general theories independently
developped by \cite{guichard_anosov_2012} and \cite{kapovich_dynamics_2017}
for Anosov representations and CEA subgroups,
as we will explain in the next paragraph.
The interpretation that we give of the boundary $\partial\Omega(\Gamma)$
as the union of the attractive bouquets of two circles
of the sequences of $\Gamma$ going simply to infinity
does not appear in this form in these works,
though being closely related to the descriptions given therein
(see respectively \cite[\S 10.2.6]{guichard_anosov_2012} and
\cite[\S 6]{kapovich_dynamics_2017}).
We will give in 
Paragraph \ref{sectiondomainesdiscontinuiteespacedrapeaux}
of this paper an independent proof of Theorem \ref{theoremintropositiongenerale}.
More precisely, we prove the existence of $\Omega(\Gamma)$
and describe $\partial\Omega(\Gamma)$ 
for Schottky subgroups of $\PGL{3}$
in Propositions \ref{propositiondomainefondamentalSchottky}
and \ref{propositiondescriptionensemblelimite}.
Our proof relies on a simple ping-pong argument using
an explicit description of a fundamental domain
for the action of $\Gamma$ on $\Omega(\Gamma)$
(see Corollary \ref{corollarySchottky}),
which allows us to describe the topology
of $\Gamma\backslash\Omega(\Gamma)$ by a surgery
(see also Proposition \ref{propositionSchottkyungenerateur}
for the case of one generator).

\subsubsection{Relations with Anosov representations}
It follows from \cite[Theorem 5.9]{bochi_anosov_2019}
and independently from the serie of papers \cite{kapovich_morse_2014,kapovich_recent_2016,kapovich_morse_2018}
(see also related criteria in 
\cite{guichard})
that if 
$\Gamma=\langle g_1,\dots,g_d \rangle$ is a Schottky subgroup
of $\PGL{3}$ in the sense of Definition \ref{definitionSchottky},
then the induced representation of the
free group with $d$ generators into $\PGL{3}$
is \emph{Anosov}.
The notion of Anosov representation,
originally introduced by Labourie in \cite{labourie_anosov_2006},
has been intensively studied in the past years.
In particular, for Anosov representations of 
finitely generated word-hyperbolic groups
in a semi-simple Lie group $G$,
\cite{guichard_anosov_2012} proves the existence of a $\Gamma$-invariant 
open subset $\Omega$ of an homogeneous space $G/P$
where the action of $\Gamma$ is properly discontinuous and cocompact.
Independently, \cite{kapovich_dynamics_2017} proves analog results
for \emph{CEA} subgroups of $G$ -- a notion closely linked to the one of
Anosov representations.
Schottky subgroups of $\PGL{3}$
fall into both  settings
(see \cite[Remark 1.6]{kapovich_dynamics_2017})
and the existence of the open subset of discontinuity 
in $\X=\PGL{3}/\Pmin$
with cocompact action of $\Gamma$ appearing in Theorem
\ref{theoremintropositiongenerale}
is a particular case of the results
\cite[Theorem 1.11]{guichard_anosov_2012} and 
\cite[Theorem 1.8]{kapovich_dynamics_2017}.
We also point out 
\cite{stecker_domains_2018},
where the authors 
extend some of these results 
to the setting of 
\emph{purely hyperbolic generalized Schottky subgroups} 
of $\PSL{2n+1}$ acting on
spaces of oriented flags.
\par In \cite{barbot_flag_2001,barbot}, Barbot studies
the case of Anosov representations of fundamental groups 
of closed higher genus surfaces into $\PGL{3}$
-- actually, the work \cite{barbot_flag_2001} precedes
the definition of Anosov representations. 
In the same way,
the paper \cite{franceslorentziankleinian}
studies the action of conformal Schottky subgroups on the Einstein universe
before the general investigation of 
Anosov representations.
Both of these works were 
important inspirations for 
the point of view adopted in Paragraph 
\ref{sectiondomainesdiscontinuiteespacedrapeaux} of this paper 
on Schottky subgroups of $\PGL{3}$.

\subsection{Dynamics of $\PGL{3}$ on the flag space}
The central tool of this paper, 
for the dynamical results proved in section \ref{sectionTheoremeB}
and for the construction of fundamental domains for Schottky subgroups
in section \ref{sectiondomainesdiscontinuiteespacedrapeaux},
is a detailed description of the dynamics of $\PGL{3}$ on the flag space $\X$.
This is the content of section \ref{sectiondynamiqueespacedrapeaux},
which is of independent interest.
For any sequence $(g_n)$ going simply to infinity in $\PGL{3}$,
we describe two ``dual'' filtrations 
by natural geometric objects of $\X$,
that are pairwise repulsive and attractive objects 
for the action of $(g_n)$.
This description is achieved in Paragraph \ref{soussectionconsequencesdynamiquesX}
by a very hands-on approach
based on the description 
of the dynamics of $\PGL{3}$ on $\RP{2}$ in Paragraph 
\ref{soussectiondynamiqueRP2}.
\par These results are related to those obtained in
\cite[\S 6]{kapovich_dynamics_2017} in the general setting
of \emph{regular} discrete subgroups of semi-simple Lie groups --
the specificity of the case that we consider in this paper
allowing us to obtain a refined geometrical description.
\par We also point out \cite{falbel} where the authors construct
a $(\PGL{3},\X)$-structure (called a \emph{flag structure} in their paper)
on an open hyperbolic three-manifold.
The existence of such a structure on a \emph{closed} hyperbolic
three-manifold is an open question,
and a natural way to address it would be to compactify the flag structure
of \cite{falbel} by using the dynamics of $\PGL{3}$
on $\X$.

\subsection*{Acknowledgments}
While working on this paper,
I had the chance to have many inspiring discussions with several people,
offering me different points of view on this subject.
I want here to take the chance to thank them very gratefully.
To begin with, I would like to thank
Charles Frances and Elisha Falbel
for encouraging me to give to this paper a definite form.
I thank Olivier Guichard for taking the time to
explain to me the links of this subject 
with works on Anosov representations.
Finally, I thank Rapha\"el Alexandre,
Julien March\'e, Tali Pinsky and Nir Lazarovich
for enlightening discussions about different aspects of this paper.

\subsection*{Notations and conventions}
We denote by $\Diag(\alpha_1,\dots,\alpha_n)$
the diagonal $(n\times n)$-matrix 
whose entries are the $\alpha_i\in\R$
(or its projection in $\PGL{n}$,
the context avoiding any confusion),
and by $[g]$ the class in $\PGL{n}$ of an element $g\in\GL{n}$.
We denote $[x_1:\dots:x_n]=\R(x_1,\dots,x_n)\in\RP{n-1}$
for any $(x_1,\dots,x_n)\neq(0,\dots,0)$ and
$[P]$ denotes the projection in $\RP{n-1}$ of
$\Vect(P)$ for any any $P\subset\R^n$.
The standard basis of $\R^n$ will always be denoted by
$(e_1,\dots,e_n)$.
\par Any differential geometric object will
be assumed to be smooth (that is, $\cc^\infty$) 
if not otherwise specified, and all manifolds will
be assumed to be boundaryless.

\section{Dynamics in the flag space}\label{sectiondynamiqueespacedrapeaux}
The \emph{flag space} is defined by
\begin{equation}\label{equationdefinitionX}
 \X=\enstq{(p,D)}{p\in D}\subset\RP{2}\times\RP{2}_*,
\end{equation}
where $\RP{2}$ denotes the projective space and $\RP{2}_*$
the space of projective lines of $\RP{2}$.
The natural action of $\PGL{3}$ on $\X$ 
defined by $g\cdot(p,D)=(g(p),g(D))$ is transitive
and induces thus an identification of $\X$ with the homogeneous
space $\PGL{3}/\Pmin$, with $\Pmin$ the subgroup of upper-triangular
matrices of $\PGL{3}$ --
which is the stabilizer of $([e_1],[e_1,e_2])\in\X$.
Our goal in this first section is to describe the dynamics
of sequences of elements of $\PGL{3}$ on $\X$,
which will be achieved 
in Paragraph \ref{soussectionconsequencesdynamiquesX}.

\subsection{Dynamic sets}
Our main technical tool to precisely describe
the dynamics
of a sequence of diffeomorphisms $(g_n)$ of a compact manifold $M$, 
will be the \emph{dynamic sets}
of the points $x\in M$, defined as:
\begin{equation}\label{equationdefDx}
 \D_{(g_n)}(x)=\enstq{\text{accumulation points of~}
 (g_n(x_n))}{(x_n)\in M^\N,\lim x_n=x}.
\end{equation}
It is not difficult to check that these are closed
and thus compact subsets of $M$,
which are always non-empty since $M$ is compact.
The first important utility of dynamic sets is
to determine the properness of group actions on open sets
of $M$.
For $\Gamma$ a topological group acting on $M$ and $x\in M$, 
we denote by $\mathcal{D}_{\Gamma}(x)$ the union of the
$\mathcal{D}_{(g_n)}(x)$, $(g_n)$ being any sequence of $\Gamma$
\emph{going to infinity}
(that is escaping from any compact subset of $\Gamma$).
We will say that two points $x$ and $y$ of $M$ 
are \emph{dynamically related} for the action of $\Gamma$
if $y\in\mathcal{D}_{\Gamma}(x)$ or $x\in\mathcal{D}_{\Gamma}(y)$.
The following Lemma is then a straightforward translation
of the classical definition of proper actions
using dynamic sets.
\begin{lemma}\label{lemmepropretedynamiquementrelie}
 $\Gamma$ acts properly on an open set $\Omega$ of $M$, 
 if and only if no pair of points of $\Omega$
 is dynamically related for the action of $\Gamma$.
\end{lemma}
See \cite{franceslorentziankleinian,kapovich_dynamics_2017},
where the notions of dynamic sets and dynamical relations
are used to prove the properness of different group actions.
\par Dynamic sets are also an efficient way
to determine the limit of a sequence of compact subsets of $M$.
Endowing $M$ with a Riemannian metric and its associated distance $d$, 
we recall that the set $\mathcal{K}(M)$ of compact subsets of $M$ is endowed
with the classical \emph{Hausdorff distance}
\[
d_H(K,L)=\inf\enstq{r>0}{K\subset L_r\text{~and~}L\subset K_r}
\]
for any $K$ and $L$ in $\mathcal{K}(M)$,
where $K_r=\enstq{x\in M}{d(x,K)\leq r}$.
The topology induced by this distance on $\mathcal{K}(M)$
is actually independent of the Riemannian metric originally chosen
on $M$, since two such metrics induce bi-Lipschitz equivalent
distances by compacity of $M$.
We will always implicitly endow $\mathcal{K}(M)$
with the topology induced in this way by the Hausdorff distance,
that we call the \emph{Hausdorff topology},
and for which $\mathcal{K}(M)$ is compact.
We denote by $\Int P$ the interior of a subset $P$,
and by $\Cl P$ its closure.
\begin{lemma}\label{lemmeconsequenceconvergencecompacts}
Let $(g_n)$ be a sequence of diffeomorphisms of $M$.
\begin{enumerate}
 \item Let $K$ be a compact subset of $M$, 
 such that $g_n(K)$ converges to a compact $K_\infty$ 
 for the Hausdorff topology.
 Then $\bigcup_{x\in\Int(K)}\D_{(g_n)}(x) \subset K_\infty \subset \bigcup_{x\in K}\D_{(g_n)}(x)$.
 \item Let $K$ be a compact subset of $M$ of non-empty interior, 
 and $K_\infty\in\mathcal{K}(M)$ such that
 \[
 \bigcup_{x\in K}\D_{(g_n)}(x) \subset K_\infty
 \subset\Cl\left(\bigcup_{x\in \Int(K)}\D_{(g_n)}(x)\right).
 \]
 Then $g_n(K)$ converges to $K_\infty$ for the Hausdorff topology. 
\end{enumerate}
\end{lemma}
\begin{proof}
1. Let $y\in\D_{(g_n)}(x)$ with $x\in\Int(K)$, 
say $y=\lim g_n(x_n)$ with $x_n$ converging to $x$.
 Then for $n$ large enough $g_n(x_n)\in g_n(K)$, 
 and thus $d(g_n(x_n),K_\infty)\leq d_H(g_n(K),K_\infty)$.
 Hence $d(y,K_\infty)=\lim d(g_n(x_n),K_\infty)=0$
 since $\lim d_H(g_n(K),K_\infty)=0$ by hypothesis,
 implying $y\in K_\infty$ as the latter is closed. 
 For the reverse inclusion, let $y\in K_\infty$. 
 Then for any $r>0$, $y\in (g_n(K))_r$ for $n$ large enough. 
 There exists thus $n$ as large as we want
 and $x_n\in K$, such that $d(y,g_n(x_n))\leq r$. 
 Passing to a subsequence, there exists finally a sequence $x_n\in K$ 
 such that $g_n(x_n)$ converge vers $y$.
 Possibly taking a further subsequence, 
 we can assume that $(x_n)$ 
 converges to some point $x\in K$, and then $y\in\D_{(g_n)}(x)$,
 finishing the proof. \\
 2. According to the first claim of the Lemma, 
 if both inclusions hold
 "then $K_\infty$ is the unique accumulation point of $(g_n(K))$.
 Since $\mathcal{K}(M)$ is a compact metrizable space, 
 this forces $(g_n(K))$ to converge to $K_\infty$.
\end{proof}

\subsection{Dynamical types in $\PGL{3}$}
\label{soussectiontypesdynamiques}
We now come back to the setting that we will be interested with,
and describe the three possible asymptotical types of a sequence
of $\PGL{3}$ going to infinity.

\subsubsection{Cartan decomposition and projection}
Compact perturbations of a sequence $(g_n)$
do not change the nature of its dynamic,
but only shift its dynamic sets.
More precisely, if two sequences $(g_n)$ and $(a_n)$ of $\PGL{3}$ satisfy
$g_n=k_na_nl_n$, where $(k_n)$ and $(l_n)$ respectively converge 
to $k_\infty$ and $l_\infty$ in $\PGL{3}$, then
\begin{equation}\label{equationlienensemblesD}
\D_{(g_n)}(x)=k_\infty\D_{(a_n)}(l_\infty(x)).
\end{equation}
This relation is a good motivation to reduce the description
of the dynamics in $\PGL{3}$ to the study of a particularly simple
types of elements: diagonal matrices.
To this end, $\PGL{3}$ enjoys the useful \emph{Cartan decomposition}
$\PGL{3}=KAK$, with $K\coloneqq\PO{3}$ the orthogonal group
and
\begin{equation}\label{equationdefA}
 A\coloneqq
 \enstq{
 \Diag(\alpha,\beta,\gamma)=
 \begin{bmatrix}
  \alpha & 0 & 0 \\
  0 & \beta & 0 \\
  0 & 0 & \gamma
 \end{bmatrix}
 }{\alpha,\beta,\gamma>0} \subset \PGL{3}
\end{equation}
the subgroup of diagonal elements of $\PGL{3}$ 
\emph{having positive entries}
(note that the $KAK$ decomposition is in this setting a simple consequence
of the polar decomposition).
We emphasize that in a decomposition
$g=kal$ with $(k,l)\in K^2$ and $a\in A$,
the pair $(k,l)$ is non-unique
but $a$ is unique up to permutation of its diagonal entries
which are the singular values of $g$
(that is the squared roots of the eigenvalues of $\transp{g}g$).
Any $a\in A$ is thus conjugated to an 
unique \emph{standard element} $\Diag(\alpha,\beta,\gamma)\in A$
such that $\alpha\geq\beta\geq\gamma$.
The set of standard elements of $A$ will be denoted by  $A^+$
(this is only a semi-subgroup of $A$).
Any $g\in\PGL{3}$ enjoys thus a \emph{standard decomposition}
\begin{equation}\label{equationstandarddecomposition}
 g=kal, \text{~with~} (k,l)\in K^2 \text{~and~} a=a(g)\in A^+,
\end{equation}
in which $a(g)\in A^+$ is unique 
and called the \emph{Cartan projection}
of $g$.

\subsubsection{Asymptotic directions in $\PGL{3}$}
The standard decomposition of elements of $\PGL{3}$ allows us,
with the help of relation \eqref{equationlienensemblesD}, 
to reduce the investigation of dynamic sets
of sequences of $\PGL{3}$
to the specific case of sequences of $A^+$.
Therefore, we now focus on sequences of $A^+$
going to infinity.
\begin{definition}\label{definitionsimplyinfinity}
A sequence 
$a_n=\Diag(\alpha_n,\beta_n,\gamma_n)\in A^+$
goes \emph{simply to infinity} if it goes to infinity with 
the three sequences
$\frac{\alpha_n}{\beta_n}\geq1$, $\frac{\alpha_n}{\gamma_n}\geq1$,
$\frac{\beta_n}{\gamma_n}\geq1$ having a limit in 
$\intervalleff{1}{+\infty}$. 
A sequence $g_n\in\PGL{3}$ goes
\emph{simply to infinity} if there exists a standard decomposition
$g_n=k_na_nl_n$ whose factors $k_n$ and $l_n$ in $K$ converge
and whose Cartan projection $a_n=a(g_n)\in A^+$ goes simply to infinity.
\end{definition}

It is easy to check that for any sequence
$a_n=\Diag(\alpha_n,\beta_n,\gamma_n)\in A^+$ 
going simply to infinity, we have
$\lim \frac{\alpha_n}{\gamma_n}=+\infty$.
The dynamics of $(a_n)$ does thus only depend on the ratios
$\lim\frac{\alpha_n}{\beta_n}$ and $\lim\frac{\beta_n}{\gamma_n}$,
which ends up in three distinct \emph{asymptotic types} in $\PGL{3}$.
\begin{definition}\label{definitiontypesdynamiques}
Let $(g_n)$ be a sequence of $\PGL{3}$ 
going simply to infinity, 
and $\Diag(\alpha_n,\beta_n,\gamma_n)\in A^+$ 
be its Cartan projection.
\begin{itemize}
 \item If $\lim \frac{\alpha_n}{\beta_n}=+\infty$ 
 and $\lim \frac{\beta_n}{\gamma_n}<+\infty$,
 we will say that $(g_n)$ is of \emph{unbalanced type $\alpha$}.
 \item If $\lim \frac{\alpha_n}{\beta_n}<+\infty$ et $\lim \frac{\beta_n}{\gamma_n}=+\infty$,
 we will say that $(g_n)$ is of \emph{unbalanced type $\beta$}.
 \item If $\lim \frac{\alpha_n}{\beta_n}=\lim \frac{\beta_n}{\gamma_n}=+\infty$,
 we will say that $(g_n)$ is of \emph{balanced type}.
\end{itemize}
\end{definition}
The reader could recognize in those definitions
a very down-to-earth formulation of the notion of Weyl chambers
in the subgroup $A$.

\begin{remark}\label{remarksubsequencessimply}
 By compacity of $K$, any sequence $(g_n)$
 going to infinity in $\PGL{3}$
has a subsequence going simply to infinity.
Furthermore, the possible asymptotic types of the subsequences
of $(g_n)$ going simply to infinity
do only depend on those of the subsequences
of its Cartan projection $(a(g_n))$.
In particular, all the subsequences of a sequence $(g_n)$ going simply to infinity
have the same asymptotic type than $(g_n)$.
Furthermore, one easily checks that
if $(g_n)$ is a sequence going simply to infinity in $\PGL{3}$
and $(k_n)$, $(l_n)$ are relatively compact sequences in $\PGL{3}$,
then any subsequence of $(k_ng_nl_n)$ going simply to infinity
has the same asymptotic type than $(g_n)$.
\end{remark}

Among sequences of $\PGL{3}$,
iterates of a fixed element of $\PGL{3}$ are 
particularly important examples.
We will say (by a slight misuse of language)
that $g\in\PGL{3}$ is \emph{diagonalizable}
if it has a representative $g_0\in\GL{3}$
which is diagonalizable \emph{on $\R$}.
The following claim is then a straightforward 
application
of Definition \ref{definitiontypesdynamiques}.
\begin{lemma}\label{lemmasequencesiterates}
Let $g\in\PGL{3}$ be a diagonalizable element.
Then: 
\begin{itemize}
 \item $(g^n)$ goes to infinity in $\PGL{3}$ if, and only if
 we do not have $a=b=c$ with $a\geq b\geq c>0$
the absolute values 
of the eigenvalues of $g$ counted with multiplicity;
\item if $(g^n)$ goes to infinity,
then it goes simply to infinity if, and only if 
its three eigenvalues have the same sign.
\end{itemize}
Furthermore, if $(g^n)$ goes to infinity,
then its subsequences going simply to infinity all have the same asymptotical type:
\begin{itemize}
 \item unbalanced type $\alpha$ if $a>b=c$; 
 \item unbalanced type $\beta$ si $a=b>c$; 
 \item balanced type if $a>b>c$,
in which case we will say that $g$ is \emph{loxodromic}.
\end{itemize}
\end{lemma}

We conclude the description of asymptotic directions in $\PGL{3}$
with the following duality. 
\begin{lemma}\label{lemmesymetriestypesdynamiques}
 Let $(g_n)$ be a sequence of $\PGL{3}$ 
 going simply to infinity.
 Then $(g_n^{-1})$ goes simply to infinity, and if
 $(g_n)$ is of unbalanced type $\alpha$ 
 (respectively unbalanced type $\beta$, resp. balanced type),
 then $(g_n^{-1})$ is of unbalanced type $\beta$ 
 (resp. unbalanced type $\alpha$, resp. balanced type).
\end{lemma}
\begin{proof}
Thanks to standard decomposition \ref{equationstandarddecomposition}
and relation \ref{equationlienensemblesD},
it is sufficient to prove this for a sequence
$a_n=\Diag(\alpha_n,\beta_n,\gamma_n)\in A^+$ going simply to infinity,
and we moreover assume $(a_n)$ of unbalanced type $\alpha$,
the argument being similar in the two other cases.
Then 
$a_n^*\coloneqq \Diag(\gamma_n^{-1},\beta_n^{-1},\alpha_n^{-1})\in A^+$ 
goes simply to infinity with unbalanced type $\beta$, and with
\begin{equation}\label{equationdefinitiong13}
 I
= \begin{bmatrix}
   & & 1 \\
    & 1 & \\
    1 & & \\
 \end{bmatrix}
 =I^{-1}\in K
\end{equation}
we have $a_n^{-1}=Ia_n^*I^{-1}$, showing that 
$(a_n^{-1})$ is also of unbalanced type $\beta$.
\end{proof}

\subsection{Dynamics in the projective plane}
\label{soussectiondynamiqueRP2}
In this section
we give a systematic description of the dynamic sets of 
points of $\RP{2}$ for the action of 
sequences of $\PGL{3}$ escaping to infinity,
depending on the three possible asymptotic behaviours previously described.
Some of these results are likely to be already known
(see for instance 
\cite[\S 3.3]{goldman_projective_1987} 
for related material),
but we give here precise statements and complete proofs
for the convenience of the reader.
\par For $p\in\RP{2}$ we define the
\emph{dual projective line of $p$} as 
$p^*=\enstq{D\in\RP{2}_*}{D\ni p}$. 
We also recall that for $P\subset\R^3$,
$[P]$ denotes the projection of $\Vect P$ in $\RP{2}$,
and that $(e_1,e_2,e_3)$ denotes the standard basis of $\R^3$.
\begin{lemma}\label{lemmeattractifRP2}
 Let $g_n\in\PGL{3}$ be a sequence going simply to infinity 
 with unbalanced type $\alpha$. 
 Then there exists a projective line $D_-$ and a point $p_+$
 in $\RP{2}$, respectively called the
 \emph{repulsive line} and the \emph{attractive point} of $(g_n)$,
 as well as a diffeomorphism 
 $\hat{g}_\infty\colon D_-\to(p_+)^*$,
 satisfying the following.
\begin{enumerate} 
  \item For any $p\in\RP{2}\setminus D_-$, $\D_{(g_n)}(p)=p_+$.
  \item For any $p\in D_-$, 
  $\D_{(g_n)}(p)=\hat{g}_\infty(p)$.
 \item $\hat{g}_\infty$ is equivariant for a morphism
 $\rho_\infty\colon\Stab(D_-)\to\Stab(D_-)\cap\Stab(p_+)$.
 \end{enumerate}
 If moreover $g_n\in A^+$, then 
 $p_+=[e_1]$ and $D_-=[e_2,e_3]$.
\end{lemma}
\begin{proof}
Let us assume that $g_n=k_na_nl_n$ is a standard decomposition 
of $g_n$ as defined in \eqref{equationstandarddecomposition}, with
$\lim k_n=k_\infty$, $\lim l_n=l_\infty$
and $(a_n)$ going simply to infinity with unbalanced type $\alpha$.
Then if we assume that we already proved the Lemma for 
sequences of $A^+$,
it is easy to check, using the relation 
\eqref{equationlienensemblesD} between dynamic sets,
that $(g_n)$ verifies the claims
of the Lemma with $D_-(g_n)=l_\infty^{-1}(D_-(a_n))$,
$p_+(g_n)=k_\infty(p_+(a_n))$ and
$\hat{g}_\infty=k_\infty\circ\hat{a}_\infty\circ l_\infty$.
We thus only have to prove the Lemma for sequences
\[
 a_n=
 \begin{bmatrix}
  1 & & \\
   & \beta_n & \\
    & & \gamma_n
 \end{bmatrix} \in A^+,
 \]
going simply to infinity with unbalanced type $\alpha$.
Therefore $\lim \beta_n=\lim \gamma_n=0$
and $\lim \frac{\gamma_n}{\beta_n}=\lambda_\infty\in\intervalleof{0}{1}$.
We define $p_+=[e_1]$ and $D_-=[e_2,e_3]$.\\
1. Since $\RP{2}$ is compact,
$\D_{(a_n)}(p)$ is non-empty and
it is sufficient to show the direct inclusion.
If $p_n\in\RP{2}$ converges to $[e_1]$,
then for $n$ large enough there exists
$(x_n,y_n)\in\R^2$ converging to $(0,0)$
such that $p_n=[1:x_n:y_n]$.
Hence $a_n(p_n)=[1:\beta_n x_n:\gamma_n y_n]$ converges to $[e_1]$,
showing that $\D_{(a_n)}(p)\subset[e_1]$ as claimed. \\
2. We define $a_\infty=\Diag(1,1,\lambda_\infty)$ and
$\hat{a}_\infty(p)=[p_+,a_\infty(p)]$.
For convenience in the notations, we assume that
for some $y\in\R$,
$p=[0:1:y]\in[e_2,e_3]\setminus\{[e_3]\}$
(the proof being similar if $p\neq[e_2]$).
\par We first show that $\D_{(a_n)}(p)\subset\hat{a}_\infty(p)$.
Since $a_n\restreinta_{[e_2,e_3]}$ uniformly converges 
to $a_\infty\restreinta_{[e_2,e_3]}$, for any sequence
$p_n\in[e_2,e_3]$ converging to $p$ 
we readily have $\lim a_n(p_n)=a_\infty(p)\in\hat{a}_\infty(p)$.
We thus assume that $p_n=[1:x_n:y_n]\in\RP{2}\setminus[e_2,e_3]$ converges to $p$, implying $\lim \frac{1}{x_n}=0$ and $\lim \frac{y_n}{x_n}=y$.
Passing to a subsequence, we can assume that 
$a_n(p_n)=[1:\beta_n x_n : \gamma_n y_n]$ converges to
$q\in\RP{2}$ and we want to prove that $q\in\hat{a}_\infty(p)$.
If $q\in [e_2,e_3]$ then $\lim\norme{(\beta_n x_n,\gamma_n y_n)}=+\infty$,
and since 
$\lim\abs{\frac{\gamma_n y_n}{\beta_n x_n}}=\lambda_\infty y<+\infty$
this prevents $(\beta_n x_n)$ to be bounded.
Passing to a subsequence we thus have $\lim \abs{\beta_n x_n}=+\infty$, 
and 
$q=\lim [\frac{1}{\beta_n x_n}:1:\frac{\gamma_n y_n}{\beta_n x_n}]
=[0:1:\lambda_\infty y]=a_\infty(p)\in\hat{a}_\infty(p)$.
If $q=[1:x_\infty:y_\infty]\in\RP{2}\setminus [e_2,e_3]$
then $\frac{y_\infty}{x_\infty}=\lim \frac{\gamma_n y_n}{\beta_n x_n}=\lambda_\infty y$, which exactly means that
 $q\in[e_1,(0,1,\lambda_\infty y)]=\hat{a}_\infty(p)$.
\par We now show the reverse inclusion 
$\hat{a}_\infty(p)\subset\D_{(a_n)}(p)$.
For $t\in\R^*$, the sequence
 $p_n\coloneqq[1:\frac{t}{\beta_n}:\frac{yt}{\beta_n}]
 =[\frac{\beta_n}{t}:1:y]$ 
 converges to $p$ while $a_n(p_n)$ converges to 
 $q=[1:t:\lambda_\infty yt]\in[e_1,a_\infty(p)]\cap(\RP{2}\setminus [e_2,e_3])$.
 This shows that 
 $\hat{a}_\infty(p)\setminus\{e_1,a_\infty(p)\}\subset\D_{(a_n)}(p)$
 which implies $\hat{a}_\infty(p)\subset\D_{(a_n)}(p)$ 
 since the latter is closed. \\
 3. Let us denote by $b_\infty=\Diag(1,\lambda_\infty)\in\PGL{2}$ the restriction of $a_\infty$ to $D^-$.
Then $\hat{a}_\infty$ is equivariant for the following morphism:
 \[
 \rho_\infty\colon
 \begin{pmatrix}
 1 & 0 \\
 * & g
 \end{pmatrix}
 \mapsto 
 \begin{pmatrix}
 1 & 0 \\
 0 & b_\infty g b_\infty^{-1}
 \end{pmatrix}.
 \qedhere
 \]
\end{proof}
\begin{lemma}\label{lemmerepulsifRP2}
 Let $g_n\in\PGL{3}$ be a sequence going simply to infinity 
 with unbalanced type $\beta$. 
 Then there exists a point $p_-$ and a projective line $D_+$
 in $\RP{2}$, respectively called the
 \emph{repulsive point} and the \emph{attractive line} 
 of $(g_n)$,
 as well as a $\R$-bundle
 $\bar{g}_\infty\colon\RP{2}\setminus\{p_-\}\to D_+$,
 satisfying the following.
 \begin{enumerate}
  \item  For any $p\in\RP{2}\setminus\{p_-\}$, 
$\D_{(g_n)}(p)=\bar{g}_\infty(p)$.
  \item  $\D_{(g_n)}(p_-)=\RP{2}$.
 \item $\bar{g}_\infty$ is equivariant
 for a morphism
  $\rho_\infty\colon\Stab(p_-)\to\Stab(D_+)\cap\Stab(p_-)$.
 \end{enumerate}
 If moreover $g_n\in A^+$, then 
 $p_-=[e_3]$ and $D_+=[e_1,e_2]$.
\end{lemma}
\begin{proof}
As we saw in the proof of Lemma \ref{lemmeattractifRP2},
we only have to prove the claims for a sequence
 \[
 a_n=
 \begin{bmatrix}
  \alpha_n & & \\
   & \beta_n & \\
    & & 1
 \end{bmatrix}\in A^+
 \]
going simply to infinity with unbalanced type $\beta$,
therefore 
$\lim \alpha_n=\lim \beta_n=+\infty$ 
and $\lim \frac{\beta_n}{\alpha_n}=
\lambda_\infty\in\intervalleof{0}{1}$. 
We define $p_-=e_3$ and $D_+=[e_1,e_2]$. \\
1. With $a_\infty=\Diag(1,\lambda_\infty,1)$, the fibration is
$\bar{a}_\infty(p)=a_\infty([p_-,p]\cap D_+)$.
 According to Lemma \ref{lemmesymetriestypesdynamiques}
 $(a_n^{-1})$ goes simply to infinity with unbalanced type $\alpha$,
 and since $a_n^{-1}=g_0b_ng_0^{-1}$ with 
 $b_n=\Diag(1,\beta_n^{-1},\alpha_n^{-1})\in A^+$ and
 $g_0=\left[\begin{smallmatrix}
   & & 1 \\
    & 1 & \\
    1 & & \\
 \end{smallmatrix}\right]$,
 the attractive point and repulsive line of $(a_n^{-1})$
 are respectively $[e_3]$ and $[e_1,e_2]$
 and $\widehat{a^{-1}}_\infty(p)=[e_3,a_\infty^{-1}(p)]$
 for any $p\in[e_1,e_2]$.
 With $p\neq p_-$,
we only have to prove that 
$\mathcal{D}_{(a_n)}(p)\subset \{\bar{a}_\infty(p)\}$
since $\D_{(a_n)}(p)$ is non-empty.
We take $q\in\mathcal{D}_{(a_n)}$ and passing to a subsequence of $(a_n)$
we can assume that
$q=\lim a_n(p_n)$ with $p_n\in\RP{2}$ such that $\lim p_n=p$.
Denoting $q_n=a_n(p_n)$, since $\lim a_n^{-1}(q_n)=p\neq[e_3]$ 
 is not the attractive point of $(a_n^{-1})$,
 $q=\lim q_n$ belong to its repulsive line $[e_1,e_2]$
 according to Lemma \ref{lemmeattractifRP2}.
 We thus have $p\in [e_3,a_\infty^{-1}(q)]$
 and thus $a_\infty^{-1}(q)\in[e_3,p]$.
 Therefore $q=[e_3,a_\infty(p)]\cap[e_1,e_2]=\bar{a}_\infty(p)$
 as claimed. \\
2.  For any $q=[x:y:1]\in\RP{2}\setminus D_+$,
$p_n\coloneqq[\frac{x}{\alpha_n}:\frac{y}{\beta_n}:1]$ converges 
to $p_-$ and $a_n(p_n)=q$ converges to $q$.
Hence $\RP{2}\setminus D_+\subset\D_{(a_n)}(p_-)$, 
showing the claim since $\D_{(a_n)}(p_-)$ is closed. \\
3. Denoting $b_\infty=\Diag(1,\lambda_\infty)\in\PGL{2}$,
$\bar{a}_\infty$ is equivariant for
the following morphism:
 \begin{equation}\label{equationrelationdefinitionrhoinfinirepulsif}
 \rho_\infty\colon
 \begin{pmatrix}
 g & 0 \\
 * & 1
 \end{pmatrix}
 \mapsto 
 \begin{pmatrix}
 b_\infty g b_\infty^{-1} & 0 \\
 0 & 1
 \end{pmatrix}.
 \end{equation}
\end{proof}
We will say that two flags $(p,D)$ and $(p',D')$ in $\X$ are
\emph{in general position} if $p\notin D'$ and $p'\notin D$.
\begin{lemma}\label{lemmemixteRP2}
Let $g_n\in\PGL{3}$ be a sequence going simply to infinity 
 with balanced type. 
 Then there exists two projective lines $D_-$, $D_+$ 
 and two points $p_-\in D_-$, $p_+\in D_+$ of $\RP{2}$,
 respectively called the \emph{repulsive and attractive lines}
 and \emph{points} of $(g_n)$,
 satisfying the following.
\begin{enumerate}
 \item For any $p\in\RP{2}\setminus D_-$, $\D_{(g_n)}(p)=\{p_+\}$.
 \item For any $p\in D_-\setminus\{p_-\}$, $\D_{(g_n)}(p)=D_+$.
 \item $\D_{(g_n)}(p_-)=\RP{2}$.
\end{enumerate}
If $x_-=(p_-,D_-)$ and $x_+=(p_+,D_+)$ are in general position
then $p_\pm\coloneqq D_-\cap D_+\in\RP{2}$
is called the \emph{saddle-point} of $(g_n)$.
Moreover if $g_n\in A^+$, then 
 $p_+=[e_1]$, $p_{\pm}=[e_2]$ and $p_-=[e_3]$.
\end{lemma}
\begin{proof}
As we saw in the proof of Lemma \ref{lemmeattractifRP2},
we only have to prove the claims for a sequence
\[
 a_n=
 \begin{bmatrix}
  1 & & \\
   & \beta_n & \\
    & & \gamma_n
 \end{bmatrix}\in A^+
 \]
going simply to infinity with balanced type,
therefore 
$\lim \beta_n=\lim \gamma_n=\lim \frac{\gamma_n}{\beta_n}=0$.
We define $p_-=[e_3]\in D_-=[e_2,e_3]$,
$p_+=[e_1]\in D_+=[e_1,e_2]$ and $p_\pm=[e_2]$. \\
1. The proof of this first claim is very similar to
the one of the first claim of Lemma \ref{lemmeattractifRP2}. \\
2. Let $p=[0:1:y]\in D_-\setminus\{p_-\}$ with $y\in\R$,
and let $p_n\in\RP{2}$ be a sequence converging to $p$. 
Passing to a subsequence we can assume that $a_n(p_n)$ converges
to a point $p\in\RP{2}$ and we now show that $q\in D_+$.
If some subsequence of $(p_n)$ is contained in $D_-$,
then for $n$ large enough
$p_n=[0:1:y_n]$ for some $y_n$ converging to $y$
and $a_n(p_n)=[0:1:\frac{\gamma_n}{\beta_n}y_n]$ converges thus
to $e_2=p_\pm$.
If not, there exists a sequence $(x_n,y_n)\in\R^2$
such that $p_n=[1:x_n:y_n]$ for $n$ large enough, and
we thus have
$\lim\frac{1}{x_n}=0$ and $\lim\frac{y_n}{x_n}=y$
since $(p_n)$ converges to $p$.
We first assume that $a_n(p_n)=[1:\beta_n x_n:\gamma_n y_n]$ converges
to $q\in [e_2,e_3]$.
If $(\beta_n x_n)$ was bounded then $\lim\gamma_n y_n=0$
since $\lim\frac{\gamma_n y_n}{\beta_n x_n}=0$,
and thus $(\beta_n x_n,\gamma_n y_n)$ would be bounded
which contradicts $\lim [1:\beta_n x_n:\gamma_n y_n]\in[e_2,e_3]$.
Passing to a subsequence, we thus have
$\lim\abs{\beta_n x_n}=+\infty$ 
and $q=[e_2]\in D_+$.
We now assume that $q=[1:x_\infty:y_\infty]\notin [e_2,e_3]$. 
Then $y_\infty=0$ since $\lim\frac{\gamma_n y_n}{\beta_n x_n}=0$,
hence $q\in D_+$ again which concludes the proof of the inclusion
$\D_{(a_n)}(p)\subset D_+$.
\par Conversely for any $t\in\R^*$, 
$p_n=[1:\frac{t}{\beta_n}:\frac{ty}{\beta_n}]$ converges to
$p=[0:1:y]$ while $a_n(p_n)=[1:t:ty\frac{\gamma_n}{\beta_n}]$ 
converges to $[1:t:0]$.
This shows that $D_+\setminus\{p_+,p_-\}\subset \D_{(a_n)}(p)$, 
hence $D_+\subset\D_{(a_n)}(p)$
since $\D_{(a_n)}(p)$ is closed. \\
3. For any $(x,y)\in\R^2$ with $y\neq0$,
$p_n=[1:\frac{x}{\beta_n}:\frac{y}{\gamma_n}]$ converges to 
$p_-$ while $a_n(p_n)$ converges to $[1:x:y]$. 
This shows that 
$\RP{2}\setminus ([e_2,e_3]\cup[e_1,e_2])\subset\D_{(a_n)}(p_-)$,
hence $\RP{2}=\D_{(a_n)}(p_-)$ since the latter is closed.
\end{proof}
 
\begin{example}\label{exemplemixteRP2dz}
Let $g\in\PGL{3}$ be a loxodromic element.
We already know from Lemma \ref{lemmasequencesiterates}
that $(g^n)$ goes to infinity, and that any subsequence
of $(g^n)$ going simply to infinity has balanced type.
The proof of Lemma \ref{lemmemixteRP2} furthermore learns us that
the repulsive, saddle and attractive points
of any such subsequence are the three eigenlines
$p_-$, $p_\pm$ and $p_+$
of a representative of $g$,
arranged in the ascending order of 
the absolute value of their associated eigenvalues.
\end{example}

\begin{remark}
We saw in the proof of the three previous Lemmas
that for sequences of $A^+$,
the dynamical objects are disjoints
($p_+\notin D_-$ for the unbalanced type $\alpha$,
$p_-\notin D_+$ for the unbalanced type $\beta$,
$(p_-,D_-)$ and $(p_+,D_+)$ are in general position for the balanced type).
Note that for generic sequences of $\PGL{3}$, 
there is \emph{a priori} 
no longer any reason for this to be true.
For instance, one can check that with
\[
 g=
 \begin{bmatrix}
  1 & 1 & 0 \\
  0 & 1 & 0 \\
  0 & 0 & 1
 \end{bmatrix},
\]
any subsequence of $(g^n)$ going simply to infinity
has balanced type,
with dynamical objects
$p_-=p_+=[e_1]$ and $D_-=D_+=[e_1,e_3]$.
\end{remark}

\begin{remark}\label{remarqueliengngninverse}
We saw in Lemma \ref{lemmesymetriestypesdynamiques} 
the following duality of asymptotic directions in $\PGL{3}$:
if $g_n\in\PGL{3}$ goes simply to infinity with unbalanced type 
$\alpha$ (respectively $\beta$, respectively balanced type),
then $(g_n^{-1})$ goes simply to infinity with unbalanced type
$\beta$ (resp. $\alpha$, resp. balanced type).
As one could expect,
this duality also applies in the naive way
to dynamical objects:
repulsive objects of $(g_n^{-1})$ are the corresponding attractive
objects of $(g_n)$. 
For instance if $(g_n)$ has balanced type,
then $p_-(g_n^{-1})=p_+(g_n)$ and $D_-(g_n^{-1})=D_+(g_n)$,
and conversely.
These relations are easily verified for a sequence $a_n\in A^+$
which readily implies the general case by
using standard decomposition \eqref{equationstandarddecomposition}
and relation \eqref{equationlienensemblesD} between
dynamic sets.
\end{remark} 

\subsection{Dynamics in the dual projective plane}
\label{soussectiondynamiqueRP2star}
Denoting by $P^\bot$ the orthogonal subspace of $P\subset\R^3$
for the standard euclidean scalar product,
the diffeomorphism 
 \begin{equation}\label{equationdefinitiontau}
  \tau\colon m\in\RP{2}\mapsto [m^\bot]\in\RP{2}_*
 \end{equation}
between the projective space and its dual
is equivariant with respect to the involutive morphism
\begin{equation}\label{equationdefinvolutionCartanG}
 \Theta\colon g\in\PGL{3}\mapsto \transp{g}^{-1}\in\PGL{3}.
\end{equation}
The dynamics on $\RP{2}_*$
of a sequence $(g_n)$ of unbalanced type $\alpha$
(respectively unbalanced type $\beta$, resp. balanced type)
will thus be conjugated by $\tau$ to the dynamics on $\RP{2}$
of the sequence $(\transp{g_n}^{-1})$ of unbalanced type $\beta$
(resp. unbalanced type $\alpha$, resp. balanced type).
Furthermore, dynamical objects of $(g_n)$ acting on $\RP{2}_*$ 
are directly deduced from its dynamical objects in $\RP{2}$
as described in the following Lemmas.
\begin{lemma}\label{lemmeattractifRP2star}
Let $g_n\in\PGL{3}$ going simply to infinity
of unbalanced type $\alpha$,
  with repulsive projective line $D_-$
  and attractive point $p_+$ in $\RP{2}$.
  Denoting by $\hat{g}_\infty\colon D_-\to p_+^*$
  the diffeomorphism of Lemma \ref{lemmeattractifRP2},
  $(g_n)$ has on $\RP{2}_*$ the following dynamics:
  \begin{enumerate}
   \item for $D\in \RP{2}_*\setminus\{D_-\}$, 
   $\D_{(g_n)}(D)=\hat{g}_\infty(D\cap D_-)\in (p_+)^*$;
   \item $\D_{(g_n)}(D_-)=\RP{2}_*$.
  \end{enumerate}
\end{lemma}
There exists of course a corresponding result
for sequences of unbalanced type $\beta$,
although we will not state it here
as it will not be used in this paper.
\begin{lemma}\label{lemmemixteRP2star}
Let $g_n\in\PGL{3}$ going simply to infinity of balanced type, 
with repulsive and attractive points and projective lines
$p_-\in D_-$ and $p_+\in D_+$
in $\RP{2}$.
Then $(g_n)$ has on $\RP{2}_*$ the following dynamics:
 \begin{enumerate}
  \item for $D\in\RP{2}_*\setminus (p_-)^*$, $\D_{(g_n)}(D)=D_+$;
  \item for $D\in (p_-)^*\setminus\{D_-\}$, $\D_{(g_n)}(D)=(p_+)^*$;  
  \item $\D_{(g_n)}(D_-)=\RP{2}_*$.
 \end{enumerate}
\end{lemma}
As before, these results are easily proved in the case of sequences
of $A^+$, which yields the general case by 
using standard decomposition \eqref{equationstandarddecomposition}.

\subsection{Dynamics in the flag space}\label{soussectionconsequencesdynamiquesX}
From our description of the dynamics of $\PGL{3}$ on $\RP{2}$
and $\RP{2}_*$, 
we now deduce a precise description of its dynamics 
on the flag space $\X$
(see \eqref{equationdefinitionX}).
We refer to \cite[\S 6]{kapovich_dynamics_2017} 
for results related to those of this paragraph, in the general setting
of \emph{regular} discrete subgroups of semi-simple Lie groups.
\subsubsection{Geometry of the flag space}
\label{soussoussectionobjetsgeomX}
The dynamical repulsive and attractive objects
of a sequence of $\PGL{3}$ acting on $\X$
are natural geometric objects of $\X$ that we now define.
First of all, $\X$ bears a path structure
$\Lx=(\Exalpha,\Exbeta)$ whose associated one-dimensional 
$\alpha$ and $\beta$-leaves are
the respective fibers of the two following 
projections:
\begin{equation}\label{equationdefinitionpialphaetbeta}
 \pi_\alpha\colon(p,D)\in\X\mapsto p\in\RP{2}
 \text{~and~}
 \pi_\beta\colon(p,D)\in\X\mapsto D\in\RP{2}_*.
\end{equation}
In other words, $\Exalpha$ and $\Exbeta$ are respectively tangent
at $x=(p,D)\in\X$ to the 
circles 
\begin{equation}\label{equationcerclesalphabeta}
 \Calpha(x)=\enstq{(p,D')}{D'\ni p}=\pi_{\alpha}^{-1}(p)
 \text{~and~}
 \Cbeta(x)=\enstq{(p',D)}{p'\in D}=\pi_{\beta}^{-1}(D)
\end{equation}
that we respectively call the \emph{$\alpha$-circle}
and the \emph{$\beta$-circle} of $x$.
We will also denote $\Calpha(x)=\Calpha(p)$
and $\Cbeta(x)=\Cbeta(D)$. 
The fact that $\Exalpha\oplus\Exbeta$ is a contact distribution
and that $\Lx$ is thus a path structure
on $\X$
is classical and
can be verified by a calculation in a chart of $\X$.
Moreover
$\pi_\alpha$ and $\pi_\beta$ are $\PGL{3}$-equivariant
for the natural action of $\PGL{3}$ on $\X$
and $\Lx$ is thus $\PGL{3}$-invariant.
Actually, $\PGL{3}$ is the whole automorphism group
of $(\X,\Lx)$ 
(see for instance \cite[Lemma 2.2]{mionmouton}), 
where an \emph{automorphism} of a path structure 
$\Lm=(E^\alpha,E^\beta)$ 
on a manifold $M$ is a diffeomorphism $f$ of $M$
such that $f^*E^\alpha=E^\alpha$ and $f^*E^\beta=E^\beta$.
\par One obtains natural surfaces of $\X$ by defining
\begin{equation}\label{equationdefinitionSalphabetaalphabeta}
 \Salphabeta(x)=\bigcup_{y\in\Calpha(x)}\Cbeta(y)
 \text{~and~}
 \Sbetaalpha(x)=\bigcup_{y\in\Cbeta(x)}\Calpha(y)
\end{equation}
that we respectively call the \emph{$\alpha$-$\beta$} and the 
\emph{$\beta$-$\alpha$ surfaces} of $x=(p,D)$.
As before, we will also denote 
$\Salphabeta(x)=\Salphabeta(p)=\pi_\beta^{-1}(p^*)$
and $\Sbetaalpha(x)=\Sbetaalpha(D)=\pi_\alpha^{-1}(D)$
if $x=(p,D)$.
Note that $\alpha$-$\beta$ and $\beta$-$\alpha$
surfaces are compact and connected surfaces
of Euler characteristic zero according to Poincar\'e-Hopf Theorem
(because each of these surfaces bears a one-dimensional distribution
which is the restriction of $\Exbeta$ or $\Exalpha$).
One moreover checks that these surfaces are non-orientable,
and that $\alpha$-$\beta$ and $\beta$-$\alpha$ surfaces 
are thus Klein bottles embedded in $\X$.
\par To finish this short geometric introduction to the flag space,
we introduce the following
natural involutive diffeomorphism of $\X$:
\begin{equation}\label{equationdefinitionkappa}
 \kappa\colon (m,D)\in\X\mapsto (D^\bot,m^\bot)\in\X,
\end{equation}
called the \emph{dual involution of $\X$},
which is equivariant for the automorphism
$\Theta\colon g\mapsto \transp{g}^{-1}$ of $\PGL{3}$.
Note that $\kappa$ does not preserve $\Lx$
but switches its $\alpha$ and $\beta$-distributions:
$\kappa^*\Exalpha=\Exbeta$.
From this, one easily verifies the relations
\begin{equation*}\label{equationrelationskappa}
 \kappa(\Calpha(x))=\Cbeta(\kappa(x)),
\kappa(\Cbeta(x))=\Calpha(\kappa(x)),
\kappa(\Salphabeta(x))=\Sbetaalpha(\kappa(x))
\text{~and~} 
\kappa(\Sbetaalpha(x))=\Salphabeta(\kappa(x)).
\end{equation*}

\subsubsection{Unbalanced type $\alpha$}
For $g_n\in\PGL{3}$ a sequence going simply to infinity 
of unbalanced type $\alpha$,
with repulsive projective line $D_-$ 
and attractive point $p_+$ in $\RP{2}$,
we define $\Cbeta^-=\Cbeta(D_-)$, 
$\Sbetaalpha^-=\Sbetaalpha(D_-)$,
$\Calpha^+=\Calpha(p_+)$ and
$\Salphabeta^+=\Salphabeta(p_+)$.
\begin{lemma}\label{lemmeattractifX}
There exists a surjective application $\phi\colon\X\to\Calpha^+$
such that:
 \begin{enumerate}
\item $\phi\restreinta_{\X\setminus\Cbeta^-}$ is a (smooth)
$(\Sn{1}\times\R)$-fiber bundle
whose fibers are the $\Salphabeta(x)\setminus\Cbeta^-$ for $x\in\Cbeta^-$,
but $\phi$ is not continuous on $\Cbeta^-$;
 \item for $x\in\X\setminus\Sbetaalpha^-$, $\D_{(g_n)}(x)=\phi(x)$;
 \item for 
 $x\in\Sbetaalpha^-\setminus\Cbeta^-$, $\D_{(g_n)}(x)=\Cbeta(\phi(x))$;
 \item for $x\in\Cbeta^-$, $\D_{(g_n)}(x)=\Sbetaalpha(\phi(x))$;
 \item $\phi$ is equivariant for the morphism
 $\rho_\infty\colon\Stab(D_-)\to \Stab(p_+)\cap\Stab(D_-)$
 of Lemma \ref{lemmeattractifRP2}.
 \end{enumerate}
\end{lemma}
\begin{proof}
Denoting $x=(p,D)$,
we define 
$\phi(x)=(p_+,\hat{g}_\infty(D\cap D_-))$
if $x\notin\Cbeta^-$
and $\phi(x)=(p_+,\hat{g}_\infty(p))$ 
if $x\in\Cbeta^-$,
with $\hat{g}_\infty$ the application introduced in Lemma \ref{lemmeattractifRP2}. \\
1. These claims are immediate consequences of the definition of $\phi$. \\ 
2. If $x\notin\Sbetaalpha^-$ then
$p\notin D_-$ hence $\mathcal{D}_{(g_n)}(p)=\{p_+\}$ 
according to Lemma \ref{lemmeattractifRP2}.
Moreover $D\neq D_-$, hence
$\mathcal{D}_{(g_n)}(D)=\{\hat{g}_\infty(D\cap D_-)\}$
according to Lemma \ref{lemmeattractifRP2star}.
This proves the claim. \\
3. If $x\notin\Cbeta^-$ then $D\neq D_-$
and thus $\mathcal{D}_{(g_n)}(D)=\{\hat{g}_\infty(D\cap D_-)\}$,
proving the direct inclusion.
For the reverse inclusion, let $p_\infty\in\hat{g}_\infty(D\cap D_-)$.
Since $D\cap D_-=p\in D_-$, 
according to Lemma \ref{lemmeattractifRP2} 
there exists $p_n\in\RP{2}$ converging to $p$
such that 
$g_n(p_n)$ converges to $p_\infty$.
For $n$ large enough $p_n\notin p^\bot$ and thus
$D_n=[p_n,p^\bot\cap D]$ is a projective line
converging to $D$.
Hence $x_n=(p_n,D_n)$ converges to $x$, with $g_n(x_n)$ converging
to $(p_\infty,\hat{g}_\infty(D\cap D_-))$.
This proves the equality. \\
4. If $x=(p,D_-)\in\Cbeta^-$ then $\mathcal{D}_{(g_n)}(p)=\hat{g}_\infty(p)$
according to Lemma \ref{lemmeattractifRP2},
proving the direct inclusion.
For the reverse one, let $x_\infty=(p_\infty,D_\infty)\in
\Sbetaalpha(\hat{g}_\infty(p))\setminus\Cbeta(\hat{g}_\infty(p))$,
that is $p_\infty\in\hat{g}_\infty(p)$ and
$D_\infty\neq\hat{g}_\infty(p)$, 
hence $D_\infty\cap\hat{g}_\infty(p)=\{p_\infty\}$.
Lemma \ref{lemmeattractifRP2} gives a sequence $p_n\in\RP{2}$ 
converging to $p$ such that
$g_n(p_n)$ converges to $p_\infty$.
With $q\in D_\infty\setminus\{p_+,p_\infty\}$, for $n$ large enough
$[g_n(p_n),q]$ is a projective live converging to $D_\infty$.
According to 
Remark \ref{remarqueliengngninverse}, 
$(g_n^{-1})$ is a sequence 
of unbalanced type $\beta$, repulsive point $p_+$ and atttractive circle $D_-$.
Passing to a subsequence,
$g_n^{-1}(q)$ converges thus to a point $q_\infty\in D_-$.
Furthermore $q_\infty=p$ is impossible
because $g_n(g_n^{-1}(q))=q$ would converge to a point
of $\hat{g}_\infty(p)$,
which would imply $q\in D_\infty\cap\hat{g}_\infty(p)=\{p_\infty\}$ since $q\in D_\infty$,
contradicting our hypothesis on $q$.
Hence $D_n=[p_n,g_n^{-1}(q)]$ is for $n$ large enough a projective line
converging to $D_-$.
Finally $x_n=(p_n,D_n)\in\X$ converges to $x$ with 
$g_n(x_n)$ converging to $x_\infty$, so $x_\infty\in\D_{(g_n)}(x)$.
This shows
$\Sbetaalpha(\hat{g}_\infty(p))\setminus\Cbeta(\hat{g}_\infty(p))
\subset\D_{(g_n)}(x)$
which concludes the proof of the equality
since $\D_{(g_n)}(x)$ is closed. \\
5. This is a direct consequence of the 
$\rho_\infty$-equivariance of $\hat{g}_\infty$ proved in Lemma \ref{lemmeattractifRP2}.
\end{proof}

\subsubsection{Unbalanced type $\beta$}
For $g_n\in\PGL{3}$ a sequence going simply to infinity 
of unbalanced type $\beta$,
with repulsive point $p_-$ 
and attractive projective line $D_+$ in $\RP{2}$,
we define $\Calpha^-=\Cbeta(p_-)$, 
$\Salphabeta^-=\Salphabeta(p_-)$,
$\Cbeta^+=\Cbeta(D_+)$ and
$\Sbetaalpha^+=\Sbetaalpha(D_+)$.
\begin{lemma}\label{lemmerepulsifX}
There exists a surjective application $\phi\colon\X\to\Cbeta^+$
such that:
 \begin{enumerate}
\item $\phi\restreinta_{\X\setminus\Calpha^-}$ is a (smooth)
$(\Sn{1}\times\R)$-fiber bundle
whose fibers are the $\Sbetaalpha(x)\setminus\Calpha^-$ for $x\in\Calpha^-$,
but $\phi$ is not continuous on $\Calpha^-$;
 \item for $x\in\X\setminus\Salphabeta^-$, $\D_{(g_n)}(x)=\phi(x)$;
 \item for $x\in\Salphabeta^-\setminus\Calpha^-$, 
 $\D_{(g_n)}(x)=\Calpha(\phi(x))$;
 \item for $x\in\Calpha^-$, $\D_{(g_n)}(x)=\Salphabeta(\phi(x))$;
 \item $\phi$ is equivariant for the morphism
 $\rho_\infty\colon\Stab(p_-)\to \Stab(p_-)\cap\Stab(D_+)$
 of Lemma \ref{lemmerepulsifRP2}.
 \end{enumerate}
\end{lemma}
\begin{proof}
The standard decomposition \eqref{equationstandarddecomposition}
and the relation \eqref{equationlienensemblesD}
allow us to assume that $g_n\in A^+$
to prove these assumptions.
Thus $D_+=[e_1,e_2]$ and $p_-=[e_3]$
according to Lemma \ref{lemmeattractifRP2}.
The dual application 
$\kappa$ of $\X$ 
being equivariant for the morphism
$g\mapsto\transp{g}^{-1}$ (see \eqref{equationdefinitionkappa}), 
we have
$g_n=\kappa\circ g_n^{-1}\circ\kappa^{-1}$ and thus
\begin{equation}\label{equationrelationkappaD}
\D_{(g_n)}(x)=\kappa(\D_{(g_n^{-1})}(\kappa^{-1}(x)))
\end{equation}
for any $x\in\X$.
Now $(g_n^{-1})$ goes simply to infinity with unbalanced type $\alpha$,
and $p_+(g_n^{-1})=[e_3]$, 
$D_-(g_n^{-1})=[e_1,e_2]$.
Denoting by
$\psi\colon\X\to\Calpha[e_3]$ the application 
associated to $(g_n^{-1})$ in Lemma
\ref{lemmeattractifX} we define 
$\phi=\kappa\circ\psi\circ\kappa^{-1}$,
and all the claims are now a direct consequence of 
the corresponding statements in Lemma \ref{lemmeattractifX},
thanks to relation \eqref{equationrelationkappaD}.
With $\bar{g}_\infty\colon\RP{2}\setminus\{p_-\}\to D_+$
the application introduced in Lemma \ref{lemmerepulsifRP2}
and denoting $x=(p,D)$,
a straightforward calculation
in the case of $g_n\in A^+$ furthermore shows that:
\begin{itemize}
 \item $\phi(x)=(\bar{g}_\infty(p),D_+)$
if $x\notin\Calpha^-$,
\item and $\phi(x)=(\bar{g}_\infty(D\cap p_-^\bot),D_+)$ if $x\in\Calpha^-$. \qedhere
\end{itemize}
\end{proof}

\begin{remark}\label{remarkgeometricinterpretation}
 There is a simple geometric interpretation of the fibration $\phi$ associated to a sequence $(g_n)$ of unbalanced type $\beta$.
For any $p\in\RP{2}$, $\RP{2}\setminus\{p\}$ is foliated by the intervals 
$D\setminus\{p\}$ for $D\in(p)^*$,
and $\X\setminus\Calpha(p)$ is thus foliated by
the $\Sbetaalpha(D)\setminus\Calpha(p)$ for $D\in(p)^*$.
In other words,
$\X\setminus\Calpha(p)$ is foliated 
by the $\Sbetaalpha(x)\setminus\Calpha(p)$ for
$x=(p,D)\in\Calpha(p)$ (these are cylinders of $\X$).
These leaves are precisely the fibers of the fibration defined in 
Lemma \ref{lemmerepulsifX} with $p$ the repulsive point $p_-$
of the sequence $(g_n)$.
Moreover for any $D\in\RP{2}_*$ that does not contain $p$,
each of these leaves 
intersects $\Cbeta(D)$
in one point.
If the attractive line $D_+$ of $(g_n)$ does not contain $p_-$
(which is the case if $g_n\in A^+$)
then $\phi(x)=\phi^{-1}(x)\cap\Cbeta^+$ for any $x\in\X\setminus\Calpha^-$.
It is easy to deduce from this the corresponding description
for the fibration associated to a sequence of unbalanced type $\alpha$.
\end{remark}

\subsubsection{Balanced type}
\label{soussoussectiondynamiquebalancedtype}
For $g_n\in\PGL{3}$ a sequence going simply to infinity 
of balanced type,
with repulsive and attractive
points and projective lines
$p_-\in D_-$ and $p_+\in D_+$ in $\RP{2}$,
we define $x^-=(p_-,D_-)$, $x^+=(p_+,D_+)$,
$\Calpha^-=\Calpha(p_-)$, $\Cbeta^-=\Cbeta(D_-)$,
$\Salphabeta^-=\Salphabeta(p_-)$, $\Sbetaalpha^-=\Sbetaalpha(D_-)$,
$\Calpha^+=\Calpha(p_+)$, $\Cbeta^+=\Cbeta(D_+)$,
$\Salphabeta^+=\Salphabeta(p_+)$, $\Sbetaalpha^+=\Sbetaalpha(D_+)$,
$\Bmoins=\Calpha^-\cup\Cbeta^-$ and
$\Bplus=\Calpha^+\cup\Cbeta^+$.
\begin{remark}
 Note that $\Salphabeta^-\cap\Sbetaalpha^-=\Bmoins$
 and $\Salphabeta^+\cap\Sbetaalpha^+=\Bplus$.
\end{remark}
\begin{lemma}\label{lemmecasmixteX}
For $x\in\X\setminus \Bmoins$, $\mathcal{D}_{(g_n)}(x)\subset \Bplus$.
More precisely:
 \begin{enumerate}
  \item For $x\in\X\setminus(\mathcal{S}_{\beta,\alpha}^-\cup\mathcal{S}_{\alpha,\beta}^-)$, $\D_{(g_n)}(x)=x^+$.
  \item For 
  $x\in\mathcal{S}_{\alpha,\beta}^-\setminus\mathcal{S}_{\beta,\alpha}^-
  =\Salphabeta^-\setminus(\Calpha^-\cup\Cbeta^-)$, 
  $\D_{(g_n)}(x)=\mathcal{C}_\alpha^+$.
  \item For 
  $x\in\mathcal{S}_{\beta,\alpha}^-\setminus\mathcal{S}_{\alpha,\beta}^-
  =\Sbetaalpha^-\setminus(\Calpha^-\cup\Cbeta^-)$, 
  $\D_{(g_n)}(x)=\mathcal{C}_\beta^+$.
  \item For $x\in\mathcal{C}_\alpha^-\setminus\{x^-\}$, $\D_{(g_n)}(x)=\mathcal{S}_{\alpha,\beta}^+$.
  \item For $x\in\mathcal{C}_\beta^-\setminus\{x^-\}$, $\D_{(g_n)}(x)=\mathcal{S}_{\beta,\alpha}^+$.
  \item $\D_{(g_n)}(x^-)=\X$. 
 \end{enumerate}
\end{lemma}
\begin{proof}
The direct inclusions of these claims are direct consequences
of Lemmas \ref{lemmemixteRP2} and \ref{lemmemixteRP2star}.
We thus only prove the reverse inclusions, denoting $x=(p,D)$. \\
1. Since $\D_{(g_n)}(x)$ is non-empty, nothing remains to be proved. \\
2. If $p_-\in D$, Lemma \ref{lemmemixteRP2star}
gives for any $D_\infty\in(p_+)^*$
a sequence $D_n$ converging to $D$ such that 
$\lim g_n(D_n)=D_\infty$.
For $n$ large enough, $p_n=D_n\cap[p,D^\bot]$ is a sequence
of $\RP{2}$ converging to $p$,
hence $x_n=(p_n,D_n)\in\X$ converges to $x$
and verifies $\lim g_n(x_n)=(p_+,D_\infty)$.
This shows $(p_+,D_\infty)\in\mathcal{D}_{(g_n)}(x)$
and thus $\Calpha^+\subset\mathcal{D}_{(g_n)}(x)$,
finishing the proof of the equality. \\
3. If $p\in D_-$, 
Lemma \ref{lemmemixteRP2} gives for any $p_\infty\in D_+$
a sequence 
$p_n$ converging to $p$ such that $\lim g_n(p_n)=p_\infty$.
Then with $D_n=[p_n,p^\bot\cap D]$, 
$x_n=(p_n,D_n)$ converges to $x$ 
and verifies $\lim g_n(x_n)=(p_\infty,D_+)$.
Hence $\Cbeta^+\subset\mathcal{D}_{(g_n)}(x)$,
which concludes the proof. \\
4. We have $x=(p_-,D)$.
We choose 
$x_\infty=(p_\infty,D_\infty)\in
\Salphabeta^+\setminus\Sbetaalpha^+=
\Salphabeta^+\setminus(\Calpha^+\cup\Cbeta^+)$.
According to Lemma \ref{lemmemixteRP2star},
there exists $D_n\in\RP{2}_*$ converging to $D$
and such that $\lim g_n(D_n)=D_\infty$.
Then for $n$ large enough, $q_n=g_n(D_n)\cap [p_\infty,D_\infty^\bot]$
is a sequence of $\RP{2}$ converging to $p_\infty$.
Since $D_+$ is the repulsive line of $(g_n^{-1})$
according to Remark \ref{remarqueliengngninverse},
and $p_\infty\notin D_+$, $p_n=g_n^{-1}(q_n)$ converges
to the attractive point of $(g_n^{-1})$, that is $p_-$.
Finally $x_n=(p_n,D_n))$ converges to $x$
and $\lim g_n(x_n)=x_\infty\in\mathcal{D}_{(g_n)}(x)$.
This shows that $\mathcal{D}_{(g_n)}(x)$ contains
$\Salphabeta^+\setminus\Sbetaalpha^+$ and thus $\Salphabeta^+$,
since $\mathcal{D}_{(g_n)}(x)$ is closed
and $\Calpha^+\cup\Cbeta^+$ has empty interior. \\
5. In this case $x=(p,D_-)$.
As before we choose $x_\infty=(p_\infty,D_\infty)\in
\Sbetaalpha^+\setminus\Salphabeta^+=
\Sbetaalpha^+\setminus(\Cbeta^+\cup\Calpha^+)$.
According to Lemma \ref{lemmemixteRP2}
there exists a sequence $p_n$ converging to $p$ such that 
$\lim g_n(p_n)=p_\infty$,
and we define $L_n=[g_n(p_n),p_\infty^\bot\cap D_\infty]$, converging to $D_\infty$.
According to Lemma \ref{lemmemixteRP2star},
$(p_+)^*$ is the repulsive dual projective line
of $(g_n^{-1})$ acting on $\RP{2}_*$. 
Since $p_+\notin D_\infty$,
$D_n\coloneqq g_n^{-1}(L_n)$ converges thus to the attractive line 
of $(g_n^{-1})$, equal to $D_-$ 
as we saw in Remark \ref{remarqueliengngninverse}.
Hence $x_n=(p_n,D_n)$ converges to $x$ and 
$\lim g_n(x_n)=x_\infty\in\mathcal{D}_{(g_n)}(x)$.
As before, this concludes the proof of the claim since
$\mathcal{D}_{(g_n)}(x)$ is closed. \\
6. Let $D_\infty\notin(p_+)^*$ 
and $p_\infty\in D_\infty$.
According to Lemma \ref{lemmemixteRP2} there exists $(p_n)$ converging 
to $p_-$ such that $\lim g_n(p_n)=p_\infty$.
For $n$ large enough, 
$L_n=[g_n(p_n),p_\infty^\bot\cap D_\infty]$ is a projective line
converging to $D_\infty$. 
Since $D_\infty\notin(p_+)^*$, 
$D_n=g_n^{-1}(L_n)$ converges to $D_-$
and $x_n=(p_n,D_n)$ converges vers $x_-$ with 
$\lim g_n(x_n)=(p_\infty,D_\infty)$.
This proves that $X\setminus\Salphabeta^+\subset\D_{(g_n)}(x_-)$, 
proving our claim since $\D_{(g_n)}(x^-)$ is closed
and $\Salphabeta^+$ has empty interior.
\end{proof}

\begin{example}\label{examplebouquetloxodromic}
 Let $g\in\PGL{3}$ be a loxodromic element
 and $p_-$, $p_\pm$, $p_+$ be its repulsive, saddle
 and attractive points
 (see Example \ref{exemplemixteRP2dz}).
 Then $x^-=(p_-,[p_-,p_\pm])$ and $x^+=(p_+,[p_\pm,p_+])$ will respectively be called the \emph{repulsive}
 and \emph{attractive flags} of $g$, and
 \[
  \Bmoins(g)=\Calpha(p_-)\cup\Cbeta[p_-,p_\pm],
 \Bplus(g)=\Calpha(p_+)\cup\Cbeta[p_\pm,p_+]
 \]
 its \emph{repulsive} and \emph{attractive bouquets of circles}.
 Those are indeed the repulsive and attractive bouquets of circles
 of any subsequence of $(g^n)$ going simply to infinity.
\end{example}

\begin{remark}\label{remarqueobjetsdynamiquessuitinverse}
If $(g_n)$ is of unbalanced type $\alpha$ (respectively $\beta$), 
then we saw in
Lemma \ref{lemmesymetriestypesdynamiques} that $(g_n^{-1})$ is
of unbalanced type $\beta$ (resp. $\alpha$).
Actually, dynamics of $(g_n^{-1})$ and $(g_n)$
are directly related through the following relations between
their dynamical objects:
$\Salphabeta^-(g_n^{-1})=\Salphabeta^+(g_n)$,
$\Calpha^-(g_n^{-1})=\Calpha^+(g_n)$,
$\Sbetaalpha^+(g_n^{-1})=\Sbetaalpha^-(g_n)$,
$\Cbeta^+(g_n^{-1})=\Cbeta^-(g_n)$.
If $(g_n)$ is of balanced type, then $(g_n^{-1})$
is also of balanced type
according to
Lemma \ref{lemmesymetriestypesdynamiques},
and in this case any attractive (respectively repulsive) 
object of $(g_n)$ is the corresponding repulsive (resp. attractive)
object of $(g_n^{-1})$.
For instance
$\Calpha^-(g_n^{-1})=\Calpha^+(g_n)$, 
$\Salphabeta^-(g_n^{-1})=\Salphabeta^+(g_n)$.
\end{remark}

\section{Fundamental domains in the flag space}
\label{sectiondomainesdiscontinuiteespacedrapeaux}
In this section we introduce a natural
notion of Schottky subgroups of $\PGL{3}$,
for which we describe fundamental domains and limit sets
in the flag space.

\subsection{Fundamental domain for a loxodromic element}
Let $g$ be a loxodromic element of $\PGL{3}$
having positive eigenvalues,
whose attractive (respectively repulsive) bouquet of circles
is denoted by $\Bplus$ (resp. $\Bmoins$),
and let $(g^t)$ be the one-parameter loxodromic subgroup of $\PGL{3}$
for which $g=g^1$.
We denote by $\Gamma$ the subgroup generated by $g$
and we introduce the open set 
\[
\Omega\coloneqq \X\setminus(\Bmoins\cup \Bplus)
\]
of $\X$.
We will say that an open set $U\subset\Omega$ is a \emph{fundamental
domain for the action of $\Gamma$ on $\Omega$}, if
\begin{enumerate}[label=(\alph*)]
 \item for any $x\neq y\in U$, $y\notin\Gamma\cdot x$;
 \item $\underset{\gamma\in\Gamma}{\bigcup}\gamma(\bar{U})=\Omega$.
\end{enumerate}
\begin{lemma}\label{lemmevoisinagetubulairesurfacegenredeux}
There exists a compact neighbourhood $H^-$ of $\Bmoins$, 
as close to $\Bmoins$ as we want, 
disjoint from $\Bplus$,
and satisfying the following properties.
\begin{enumerate}
 \item $H^-$ is a genus two handlebody, whose boundary is
 transverse to the orbits of $(g^t)$.
 \item Denoting $H^+\coloneqq\X\setminus\Int(g(H^-))$,
 $(g^{-n}(H^-))$ and $(g^n(H^+))$ respectively 
 converge to $\Bmoins$ and $\Bplus$
 for the Hausdorff topology.
 \item $\Phi\colon(x,t)\in\partial H^-\times\R\mapsto g^t(x)\in\Omega$ 
 is a diffeomorphism. 
 \item $U\coloneqq\X\setminus(H^-\cup H^+)$ 
 is a fundamental domain for the action of $\Gamma$ on $\Omega$.
\end{enumerate}
\end{lemma}
We recall that a \emph{genus two handlebody} is a
(unique up to homeomorphism)
connected compact and orientable three-manifold with boundary,
obtained from the three-ball after adding two $1$-handles.
The boundary of a genus two handlebody is homeomorphic to the 
closed connected and orientable surface of genus two.
\begin{proof}[Proof of Lemma \ref{lemmevoisinagetubulairesurfacegenredeux}]
1. Up to conjugation in $\PGL{3}$, we can assume that 
\[
 g^t=
 \begin{bmatrix}
  \e^{\alpha t} & 0 & 0 \\
   0 & \e^{\beta t} & 0 \\
   0 & 0 & 1
 \end{bmatrix}
\]
with $\alpha>\beta>0$,
so that the repulsive circles of $g$ are
$\Calpha^-=\Calpha[e_3]$ and $\Cbeta^-=\Cbeta[e_2,e_3]$.
We introduce $x^-=([e_3],[e_2,e_3])=\Calpha^-\cap\Cbeta^-$
and the open $g^t$-invariant set
$\mathcal{U}=\X\setminus(\Sbetaalpha[e_1,e_2]\cup\Salphabeta[e_1])$.
In the chart 
$\psi\colon([x,y,1],[(x,y,1),(z,1,0)])\in\mathcal{U}\mapsto(x,y,z)\in\R^3$
of $\mathcal{U}$, 
a straightforward calculation shows that
$(g^t)$ is conjugated to the diagonal flow
\[
 a^t\coloneqq\Diag(\e^{\alpha t},\e^{\beta t},\e^{(\alpha-\beta)t})
 =\psi\circ g^t \circ\psi^{-1}.
\]
\par We first build a neighbourhood of $\Calpha^-$ transverse to $(g^t)$.
Let $D$ be a closed disk of $\R^2$ centered at the origin
whose boundary is transverse 
to the diagonal flow $\Diag(\e^{\alpha t},\e^{\beta t})$.
Then $A_0=\enstq{\psi^{-1}(x,y,z)}{(x,y)\in D,z\in\R}$ is transverse to
$(g^t)$ and is a neighbourhood of 
$\Calpha^-\cap \mathcal{U}=\psi^{-1}(\{0\}^2\times\R)$.
The solid torus
$A=\enstq{([x:y:1],[(x,y,1),q])}{(x,y)\in D,q\in[e_1,e_2]}$
is a neighbourhood of $\Calpha^-$,
it is the closure of $A_0$.
Since
$\partial(A\setminus A_0)=
\enstq{([x:y:1],[(x,y,1),e_1])}{(x,y)\in\partial D}$
is transverse to the orbits of $(g^t)$
and $(g^t)$ preserves $\mathcal{U}$,
$A$ is transverse to $(g^t)$.
Choosing at the beginning the disk $D$ as little as we want, 
$A$ is as close to $\Calpha^-$ as we want.
Since $\Calpha[e_1]$ is the repulsive circle of $(g^{-t})$,
this discussion applied to $(g^{-t})$ 
provides us with a solid torus $C$ transverse to $(g^{-t})$ which is a
neighbourhood of $\Calpha[e_1]$.
Its image by the dual application of $\X$
introduced in \eqref{equationdefinitionkappa}
is thus a solid torus
$C'=\kappa(C)$
which is a neighbourhood of $\Cbeta[e_2,e_3]$,
transverse to $(g^t)$ by equivariance of $\kappa$ 
(see \eqref{equationdefinvolutionCartanG}).
We can moreover choose $C'$ as close to $\Cbeta[e_2,e_3]$ as we want.
We now only need to choose a little closed ball $B$ centered at $x^-$ 
and transverse to $(g^t)$,
and to glue to $B$ the handles $A$ and $C'$, 
to obtain a neighbourhood $H^-$ of $\Bmoins$ transverse to $(g^t)$.
By construction this neighbourhood is homeomorphic to
a genus two handlebody, which concludes the proof of the claim. \\
2. Let $p_-,p_{\pm},p_+$ be the attractive, saddle and repulsive points of $g$ in $\RP{2}$.
Since 
$x_1=(p_-,[p_-,p_+])\in\Salphabeta^+(g)\setminus\Sbetaalpha^+(g)$, 
$\D_{(g^{-n})}(x_1)=\Calpha^-$ 
according to Lemma \ref{lemmecasmixteX}.
Since 
$x_2=(p_{\pm},[p_{\pm},p_-])\in\Sbetaalpha^+(g)\setminus\Salphabeta^+(g)$,
$\D_{(g^{-n})}(x_2)=\Cbeta^-$. 
Since $x_1\in\Calpha^-\subset\Int H^-$ and $x_2\in\Cbeta^-\subset\Int H^-$,
we thus obtain 
$\Bmoins\subset\cup_{x\in\Int H^-}\mathcal{D}_{(g^{-n})}(x)$. 
Applying Lemma \ref{lemmecasmixteX} and Remark \ref{remarqueobjetsdynamiquessuitinverse}
to $(g^{-n})$, we also have $\cup_{x\in H^-}\mathcal{D}_{(g^{-n})}(x)\subset \Bmoins$
since $H^-\subset \X\setminus \Bplus$. 
According to Lemma \ref{lemmeconsequenceconvergencecompacts}, this implies
that $(g^{-n}(H^-))$ converges to $\Bmoins$.
We show on the same way that $(g^n(H^+))$ converges to $\Bplus$. \\ 
3. Since the orbits of $(g^t)$ are transverse to $\partial H^-$ 
and escape out from $H^-$, 
$\Phi$ is a local diffeomorphism.
Moreover a $g^t$-orbit cannot cross $\partial H^-$ more than once, 
hence $\Phi$ is injective.
The description of the dynamics of $(g^n)$ in the previous claim
shows its surjectivity, finishing the proof of the claim. \\
4. This is a direct consequence of the previous claim.
\end{proof}

\begin{proposition}\label{propositionSchottkyungenerateur}
1. $\Gamma$ acts freely, properly and cocompactly
on $\Omega$. \\
2. Furthermore
$\Gamma\backslash\Omega$ 
is diffeomorphic to the product of the circle 
with the closed connected and orientable surface of genus two.
\end{proposition}
\begin{proof}
1. No pair of points of $\Omega$ being dynamically related according
to Lemma \ref{lemmecasmixteX}, Lemma \ref{lemmepropretedynamiquementrelie}
implies that the action of $\Gamma$ on $\Omega$ is proper.
This action is free since any non-trivial element 
of $\Gamma$ has all 
its fixed points 
on $\Bmoins\cup \Bplus$.
Finally, this action is cocompact since we found a relatively compact 
fundamental domain $U\subset\Omega$ for the action of $\Gamma$. \\
2. According to the third claim of Lemma 
\ref{lemmevoisinagetubulairesurfacegenredeux},
$\Gamma\backslash\Omega$ is indeed diffeomorphic to
the quotient of
$\partial H^-\times\intervalleff{0}{1}$
by the equivalence relation $(x,0)\sim(x,1)$,
and thus to $\partial H^-\times\Sn{1}$.
\end{proof}

\subsection{Fundamental domains for Schottky subgroups}
We now introduce a notion of Schottky subgroups in $\PGL{3}$. 
Let us first recall that two flags $(p,D)$ and $(p',D')$ in $\X$
are \emph{in general position} if $p\notin D'$ and $p'\notin D$.
\begin{definition}
 We will say that $d\geq1$ loxodromic elements 
 $g_1,\dots,g_d$ of $\PGL{3}$
 are \emph{in general position} 
 if their attractive and repulsive flags $\{x_i^{\pm}\}$ are pairwise in general position.
\end{definition}
Note that the bouquets 
of circles of loxodromic elements in general position 
are pairwise disjoint.

\begin{propositiondefinition}\label{propdefSchottkydansX}
Let $g_1,\dots,g_d$ be loxodromic elements of $\PGL{3}$ in general position,
whose repulsive and attractive 
bouquet of circles are denoted by $B_1^\pm,\dots,B_d^\pm$.
Then up to replacing each $g_i$ by $g_i^{r_i}$
for $r_i>0$ large enough, 
$g_1,\dots,g_d$ satisfy the following:
there exists $2d$ pairwise disjoint compact genus two handelbodies
$\{H_1^-,H_1^+,\dots,H_d^-,H_d^+\}$ in $\X$,
such that each $H_i^\pm$ is a neighbourhood of $B_i^\pm$
and $H_i^+=\X\setminus\Int g_i(H_i^-)$.
We will say in this case that $\Gamma=\langle g_1,\dots,g_d \rangle$
is a \emph{Schottky subgroup} of $\PGL{3}$,
and that $\{H_i^\pm\}_{i=1}^d$ 
is a \emph{set of separating handlebodies} for the $g_i$.
\end{propositiondefinition}
\begin{proof}
Since the statement is claimed modulo finite iterates of the $g_i$,
we can assume that each of them has positive eigenvalues.
For any $i$, the compact genus two handlebody
neighbourhood $H^-_i$ of $B_i^-$ built in Lemma 
\ref{lemmevoisinagetubulairesurfacegenredeux}
can be chosen as close to $B_i^-$ 
as we want, possibly replacing
$g_i$ by an iterate $g_i^{r_i}$.
According to Lemma \ref{lemmevoisinagetubulairesurfacegenredeux},
$H^+_i\coloneqq\X\setminus\Int(g_i(H^-_i))$ is a compact
genus two handlebody,
neighbourhood of
the attractive bouquet of circles $B_i^+$ of $g_i$, such that
$g^n(H^+_i)$ converges to $B_i^+$.
We can thus choose $H^+_i$ as close to $B_i^+$ as we want,
possibly replacing again $g_i$ by an iterate.
Since the $B_i^{\pm}$ are disjoint,
the $H_i^\pm$ are also pairwise disjoint if they are
sufficiently close to the $B_i^\pm$.
\end{proof}
\begin{remark}\label{remarkrelationsSchottky}
Note that Lemma \ref{lemmecasmixteX} shows that
the repulsive and attractive bouquet of circles of a loxodromic element $g$
are the only geometric objects with respect to which $(g^n)$
has a North-South dynamics 
with repulsive and attractive sets \emph{of equal dimensions}.
The notion of Schottky subgroups defined previously
is in this sense imposed by 
the dynamics in $\X$ of loxodromic elements of $\PGL{3}$.
Indeed, Definition \ref{propdefSchottkydansX}
is the natural translation of
the classical definition of a Schottky subgroup $\Gamma_0$ of $\PSL{2}$,
where the half-planes of $\overline{\Hn{2}}$ containing the
repulsive and attractive points in $\partial\Hn{2}$ of the loxodromic generators of $\Gamma_0$
are here replaced by the handlebodies
containing the repulsive and attractive bouquet of circles of the generators of $\Gamma$.
\end{remark}
By construction,
these Schottky subgroups satisfy the classical ``ping-pong'' Lemma.
\begin{proposition}\label{propositiondomainefondamentalSchottky}
Let $\Gamma=\langle g_1,\dots,g_d \rangle$ be a Schottky subgroup of $\PGL{3}$
and ${H_i^{\pm}}$ be a set of separating handlebodies for the $g_i$.
 \begin{enumerate}
  \item $\Gamma=\langle g_1,\dots,g_d \rangle$ 
  is a discrete subgroup of $\PGL{3}$ freely generated by
  $g_1,\dots,g_d$.
  \item With $U=\bigcap_{i=1}^d (\X\setminus(H^-_i\cup H^+_i))$,
  $\Omega=\bigcup_{\gamma\in\Gamma}\gamma(\bar{U})$ 
  is an open set of $\X$ where $\Gamma$ acts freely, properly and cocompactly.
 \end{enumerate}
\end{proposition}
\begin{proof}
According to Proposition \ref{propositionSchottkyungenerateur},
the subgroups $\langle g_i \rangle$ acts freely, properly and
cocompactly on $\X\setminus(B_i^-\cup B_i^+)$, 
with $U_i=\X\setminus(H_i^-\cup H_i^+)$ as a fundamental domain.
These properties allow us to apply the 
classical ping-pong Lemma to the action of
$\Gamma$ on $\Omega$. 
More precisely, the version of \emph{Klein’s combination Theorem}
proved by Frances in \cite[Theorem 5]{franceslorentziankleinian}
(see also \cite{maskit_kleinian_1988})
allows us to conclude:
$\Omega$ is open and $\Gamma$ is a discrete free subgroup 
acting freely, properly and cocompactly on $\Omega$.
\end{proof}

\subsection{Limit sets of Schottky subgroups}
In this paragraph, we would like to give the following intrisinc dynamical
meaning to the open set $\Omega$ 
found in Proposition \ref{propositiondomainefondamentalSchottky}:
$\Lambda(\Omega)=\partial\Omega$ is the limit set in $\X$ of
the Schottky subgroup $\Gamma=\langle g_1,\dots,g_d \rangle$.
\par To this end, we will use 
the classical notion of \emph{boundary of
$\Gamma$}, denoted by $\partial_{\infty}\Gamma$,
that we will see as the set of non-empty right-infinite words
$g_{i_1}^{\varepsilon_1}\dots g_{i_n}^{\varepsilon_n}\dots$
on the alphabet $\mathcal{A}=\{g_1,g_1^{-1},\dots,g_d,g_d^{-d}\}$
(with $i_k\in\{1,\dots,d\}$ and $\varepsilon_i\in\{\pm1\}$)
that are reduced, that is
$\varepsilon_k i_k\neq -\varepsilon_{k+1} i_{k+1}$.
We will say that a reduced word 
$g_{i_1}^{\varepsilon_1}\dots g_{i_n}^{\varepsilon_n}\in\Gamma$
has \emph{length $n=\abs{\gamma}$}.
The length of words $\abs{\gamma}$ defines the \emph{word metric} 
$d(\gamma,\delta)=\abs{\gamma\delta^{-1}}$ on $\Gamma$.
The natural embedding of $\partial_{\infty}\Gamma$ in $\mathcal{A}^{\N}$ 
defined by 
identifying $g_{i_1}^{\varepsilon_1}\dots g_{i_n}^{\varepsilon_n}\dots$
to the sequence $(g_{i_n}^{\varepsilon_n})$
endows $\partial_{\infty}\Gamma$ with the restriction of the product topology 
of $\mathcal{A}^{\N}$,
for which $\partial_{\infty}\Gamma$ is a Cantor space 
(in particular, $\partial_\infty\Gamma$ is compact).
The disjoint union $\Gamma\cup\partial_{\infty}\Gamma$
is endowed with a (metrizable) topology extending
those of $\Gamma$ and $\partial_{\infty}\Gamma$
and for which $\Gamma\cup\partial_{\infty}\Gamma$ is a compact space.
For this topology,
a sequence $\gamma_k\in\Gamma$ 
such that $\abs{\gamma_k}\to+\infty$
converges to a point $\delta_\infty=
g_1^{\varepsilon_1}\dots g_n^{\varepsilon_n}\dots
\in\partial_{\infty}\Gamma$
if, and only if 
there exists a non-decreasing sequence of non-negative integers 
$n_k\to+\infty$
and a sequence $\mu_k\in\Gamma$
such that $\gamma_k=\delta_{n_k}\mu_k$,
where $\delta_n=g_1^{\varepsilon_1}\dots g_n^{\varepsilon_n}$ denotes 
the sequence of finite subwords of $\delta_\infty$.
Our 
use of $\partial_{\infty}\Gamma$ being very elementary,
we introduce these notions
in a simple and naive way to avoid a technicality which 
would be useless here.
We refer to \cite{ghysdlH} for 
more details about the notion of boundary
of a Gromov-hyperbolic group.
\par Let $\{H_i^\pm\}_{i=1}^d$ 
be a set of separating handlebodies for the $g_i$ as defined in
Definition \ref{propdefSchottkydansX}.
We associate to any 
$\gamma=g_{i_1}^{\varepsilon_1}\dots g_{i_n}^{\varepsilon_n}\in\Gamma$
the compact $K(\gamma)\coloneqq 
g_{i_1}^{\varepsilon_1}\dots g_{i_{n-1}}^{\varepsilon_{n-1}}(H^{\varepsilon_n}_{i_n})$
and to any 
$\gamma_\infty\in\partial_{\infty}\Gamma$
the sequence of compacts
$K(\gamma_n)$ with
$\gamma_n$ the finite subwords of $\gamma_\infty$.
Since $g_i^{\varepsilon}(H_j^{\delta})\subset \Int(H_i^\varepsilon)$
for any $(j,\delta)\neq(i,-\varepsilon)$,
we have $g_{i_n}^{\varepsilon_n}(H_{i_{n+1}}^{\varepsilon_{n+1}})\subset 
H_{i_n}^{\varepsilon_n}$ for any $n$, 
and the sequence of compacts $K(\gamma_n)$ associated to $\gamma_\infty$
is thus decreasing.
Therefore, the intersection
\begin{equation}\label{equationdefinitionK}
\Bplus(\gamma_\infty)\coloneqq\bigcap_{n\in\N}K(\gamma_n)
\end{equation}
is a non-empty connected compact subset, which is the limit of
$(K(\gamma_n))$ for the Hausdorff topology.
\par The following statement and some parts of its proof
were inspired by analog results proved in 
\cite[Lemma 7 and 8]{franceslorentziankleinian}
for Schottky conformal subgroups acting on the Einstein universe.
We also refer to \cite[\S 6]{kapovich_dynamics_2017} for related
results in the general setting of regular discrete subgroups of semi-simple Lie groups.
\begin{proposition}\label{propositiondescriptionensemblelimite}
Let $\Gamma=\langle g_1,\dots,g_d \rangle$ be a Schottky subgroup of $\PGL{3}$,
$\{H_i^\pm\}_{i=1}^d$ 
be a set of separating handlebodies for the $g_i$
and 
$\Omega$ be 
the open set of discontinuity of Proposition 
\ref{propositiondomainefondamentalSchottky}
defined by the $H_i^\pm$.
 \begin{enumerate}
  \item The application $\Bplus$ defined in \eqref{equationdefinitionK}
  is an homeomorphism from $\partial_{\infty}\Gamma$ to the set 
  of connected components
  of $\Lambda(\Gamma)\coloneqq\partial\Omega$
  endowed with the Hausdorff topology.
  \item Let $\gamma_\infty\in\partial_{\infty}\Gamma$ and $(\gamma_n)$ denote its sequence of finite subwords.
  Then any subsequence of $(\gamma_n)$ going simply to infinity
  is of balanced type with attractive bouquet of circles 
  $\Bplus(\gamma_\infty)$.
  \item Let 
$(\gamma_n)$ be a sequence of $\Gamma$ going simply to infinity
in $\PGL{3}$.
Then $(\gamma_n)$ converges in $\Gamma\cup\partial_{\infty}\Gamma$
to a point $\gamma_\infty\in\partial_{\infty}\Gamma$,
$(\gamma_n)$ is of balanced type
and 
$\Bplus(\gamma_n)=\Bplus(\gamma_\infty)$.
  \item $\Omega$ is equal to 
  the \emph{maximal open subset of discontinuity} of $\Gamma$ defined as:
  \[
  \Omega(\Gamma)\coloneqq\X\setminus\bigcup_{
\underset{\gamma_n\underset{simply}{\longrightarrow}\infty}{\gamma_n\in\Gamma}}\Bplus(\gamma_n).
  \]
  \end{enumerate}
\end{proposition}
\begin{proof}
1. The proof of this first claim is formally
the same than \cite[Lemma 7]{franceslorentziankleinian}
but we repeat it here for the convenience of the reader.
We define $U=\X\setminus\bigcup_i(H^-_i\cup H^+_i)$ as in Proposition 
\ref{propositiondomainefondamentalSchottky},
and we introduce
\[
\Omega_n=\cup_{\abs{\gamma}\leq n}\gamma(\bar{U}) \text{~and~} \Lambda_n=\X\setminus\Omega_n,
\]
so that $\Lambda(\Gamma)=\cap_n\Lambda_n$.
Note that for any $\gamma_\infty\in\partial_{\infty}\Gamma$,
$K(\gamma_n)\subset\Lambda_n$ for any $n$, 
with $\gamma_n$ the finite subwords of $\gamma_\infty$,
showing that $\Bplus(\gamma_\infty)\subset\Lambda(\Gamma)$.
Let $x\in\Lambda(\Gamma)$, 
and $C_n$, $C$ respectively
denote the connected components of $x$ in $\Lambda_n$ and $\Lambda(\Gamma)$.
Any connected component of $\Lambda_n$ is of the form $K(\gamma)$ for 
some $\gamma\in\Gamma$ of length $n$. Moreover $C_n$ is decreasing,
hence $C_n=K(\gamma_n)$ with $\gamma_n$
the finite subwords of some $\gamma_\infty\in\partial_{\infty}\Gamma$.
Since $C\subset C_n$ for any $n$, we have 
$C\subset \Bplus(\gamma_\infty)=\cap_n C_n$ and by maximality of the connected
component this inclusion is an equality.
This shows that $\Bplus(\gamma_\infty)$ is always a connected component
of $\Lambda(\Gamma)$ and that $\Bplus$ is surjective onto the
set of connected components of $\Lambda(\Gamma)$.
To prove the injectivity, 
take $\gamma_\infty\neq\gamma_\infty'$ in $\partial_{\infty}\Gamma$
of finite subwords $(\gamma_n)$ and $(\gamma'_n)$.
There exists $k$
such that $g_{i_k}^{\varepsilon_k}\neq g_{i'_k}^{\varepsilon'_k}$
and for any $n\geq k$
$K(g_{i_k}^{\varepsilon_k}\dots g_{i_n}^{\varepsilon_n})
\subset H_{i_k}^{\varepsilon_k}$
and $K(g_{i'_k}^{\varepsilon'_k}\dots g_{i_n}^{\varepsilon_n})
\subset H_{i'_k}^{\varepsilon'_k}$
are thus disjoints.
This implies that $K(\gamma_n)$ and $K(\gamma'_n)$ are disjoints for $n$
large enough and thus that 
$\Bplus(\gamma_\infty)\neq \Bplus(\gamma'_\infty)$.
\par It only remains 
to show that $\Bplus$ is continuous and the conclusion will follow
by compacity of $\partial_{\infty}\Gamma$.
Let $\gamma_\infty^{(n)}\in\partial_{\infty}\Gamma$ converge to 
$\gamma_\infty$
with $(\gamma^{(n)}_k)$ and $(\gamma_k)$
the respective sequences of finite subwords of
$\gamma_\infty^{(n)}$ and $\gamma_\infty$.
There exists then a strictly increasing sequence
$k_n\to+\infty$ such that $\gamma_{k_n}^{(n)}=\gamma_{k_n}$
for any $n$, and thus
$K(\gamma^{(n)}_{k_n})=K(\gamma_{k_n})$
with $K(\gamma_{k_n})$ converging to 
$\Bplus(\gamma_\infty)$.
Since $K(\gamma^{(n)}_\infty)\subset K(\gamma^{(n)}_{k_n})$,
$\Bplus(\gamma_\infty)$ is thus the only
accumulation point of $(K(\gamma_\infty^{(n)}))$.
By compacity of the space of compacts for 
the Hausdorff topology, this shows that
$K(\gamma_\infty^{(n)})$ converges to $\Bplus(\gamma_\infty)$
and concludes the proof of the claim. \\
2. The proof of this claim was inspired by the one of
\cite[Lemma 8]{franceslorentziankleinian}.
We will make a repeated use of the following argument to prove our second claim. 
\begin{fact}\label{faitchangementsdomainefondamentaux}
 For $1\leq i \leq g$, 
 let ${H'}_i^-$ be a compact neighbourhood of $B_i^-$
 close enough to $H^-_i$ for the Hausdorff topology,
 and let define ${H'}_i^+\coloneqq\X\setminus g_i(\Int({H'}_i^-))$.
 Then $U'_i=\X\setminus({H'}_i^-\cup {H'}_i^+)$ remains a fundamental
 domain for the action of 
 $\langle g_i \rangle$ on $\X\setminus(B_i^-\cup B_i^+)$,
 and $U'\coloneqq U'_i\cap_{j\neq i}U_j$ remains a fundamental domain
 for the action of $\Gamma$ on $\Omega$:
 $\cup_{\gamma\in\Gamma}\gamma(\overline{U'})=\cup_{\gamma\in\Gamma}\gamma(\overline{U})=\Omega$.
\end{fact}
\begin{proof}
The same proof as in Lemma \ref{lemmevoisinagetubulairesurfacegenredeux} shows that 
$U'_i$ is a fundamental domain for the action of
$\langle g_i \rangle$ on $\X\setminus(B_i^-\cup B_i^+)$.
If ${H'}_i^-$ is close enough to $H^-_i$, then ${H'}_i^-$ et ${H'}_i^+$
remain disjoints from the other $H_j^{\pm}$ and the proof of Proposition
\ref{propositiondomainefondamentalSchottky} applies thus to the neighbourhoods
$\{{H'}_i^{\pm};H_j^{\pm},j\neq i\}$,
showing that 
$\Omega'\coloneqq\cup_{\gamma\in\Gamma}\gamma(\overline{U'})$ is open.
Moreover, if ${H'}_i^-$ is close enough to $H^-_i$ then
$\overline{U'}$ is close enough to $\overline{U}$ to be contained
in its neighbourhood $\Omega$,
and then $\Omega'\subset\Omega$
since $\Omega$ is $\Gamma$-invariant.
Likewise, $\overline{U}$ is in this case contained in the neighbourhood
$\Omega'$ of $\overline{U'}$, showing that $\Omega\subset\Omega'$
and finishing the proof.
\end{proof}
In other words, slight modifications of the neighbourhoods $H_i^{\pm}$ 
are authorized.
Let $(\gamma_n)$ be a subsequence of the finite subwords of a point 
$\gamma_\infty\in\partial_{\infty}\Gamma$,
going simply to infinity in $\PGL{3}$.
Passing to a subsequence of $(\gamma_n)$, 
we can assume that there exists two letters $g_a^{\varepsilon_a}$
and $g_b^{\varepsilon_b}\neq g_a^{-\varepsilon_a}$
such that the reduced word $\gamma_n$
ends by $g_a^{\varepsilon_a}$ for any $n$
and that $\gamma_n g_b^{\varepsilon_b}$
is a subsequence of the finite subwords of $\gamma_\infty$.
Therefore,
$\gamma_n(H_b^{\varepsilon_b})=K(\gamma_n g_b^{\varepsilon_b})$ 
converges to $\Bplus(\gamma_\infty)$.
\begin{fact}\label{faitsuitealinfinimixte}
$(\gamma_n)$ is of balanced type.
\end{fact}
\begin{proof}
We assume by contradiction that this is not the case.
Possibly replacing $(\gamma_n)$ by $(\gamma_n^{-1})$,
we can thus assume that $(\gamma_n)$ is of unbalanced type $\alpha$
according to Lemma \ref{lemmesymetriestypesdynamiques}.
We denote by $\mathcal{C}^-\subset \mathcal{S}^-$ and
$\mathcal{C}^+\subset\mathcal{S}^+$ 
its repulsive and attractive circles and surfaces in $\X$
and by $\phi\colon\X\to\mathcal{C}^+$ the fibration introduced in 
Lemma \ref{lemmeattractifX}.
\par We first assume by contradiction that 
$\mathcal{C}^-\subset\Int(H_b^{\varepsilon_b})$.
According to Lemma \ref{lemmeconsequenceconvergencecompacts},
$\Bplus(\gamma_\infty)=\lim \gamma_n(H_b^{\varepsilon_b})$ would then 
contain $\bigcup_{x\in\mathcal{C}^-}\mathcal{D}_{(\gamma_n)}(x)$
which is equal to $\bigcup_{y\in\mathcal{C}^+}\Sbetaalpha(y)$
according to Lemma \ref{lemmeattractifX}
and thus to $\X$.
This is impossible since $\Bplus(\gamma_\infty)\subset\Lambda(\Gamma)$
is a strict subset of $\X$.
Therefore $\mathcal{C}^-\cap\Int(H_b^{\varepsilon_b})$ is a strict subset
of $\mathcal{C}^-$ that we assume to be non-empty by contradiction.
We can then slightly modify $H_b^{\varepsilon_b}$ to a neighbourhood
${H'}_b^{\varepsilon_b}$ of $B_b^{\varepsilon_b}$,
in such a way that
$\mathcal{C}^-\cap H_b^{\varepsilon_b}
\subsetneq\mathcal{C}^-\cap\Int({H'}_b^{\varepsilon_b})$,
and the new open set $U'$ defined by 
${H'}_b^{\varepsilon_b}$ as in Fact \ref{faitchangementsdomainefondamentaux}
remains a fundamental domain of $\Omega$.
The first claim of the proposition
applies then to ${H'}_b^{\varepsilon_b}$
and $(\gamma_n({H'}_b^{\varepsilon_b}))$ 
converges thus to a connected component
$K'$ of $\Lambda(\Gamma)$.
According to Lemmas \ref{lemmeconsequenceconvergencecompacts}
and \ref{lemmeattractifX},
$K'\supset\bigcup_{x\in\mathcal{C}^-\cap\Int({H'}_b^{\varepsilon_b})}
\Sbetaalpha(\phi(x))$ and
$\Bplus(\gamma_\infty)\subset\mathcal{S}^+\cup
\bigcup_{x\in\mathcal{C}^-\cap H_b^{\varepsilon_b}}\Sbetaalpha(\phi(x))$.
By hypothesis on ${H'}_b^{\varepsilon_b}$, 
$K'$ and $\Bplus(\gamma_\infty)$
are thus distincts and therefore disjoints as they
both are connected components 
of $\Lambda(\Gamma)$.
This contradicts the fact that they both contain
$\bigcup_{x\in B_b^{\varepsilon_b}}\mathcal{D}_{(\gamma_n)}(x)$.
\par Finally $\mathcal{C}^-\cap\Int(H_b^{\varepsilon_b})=\varnothing$,
and possibly shrinking $H_b^{\varepsilon_b}$ thanks to 
Fact \ref{faitchangementsdomainefondamentaux} we can assume that
$\mathcal{C}^-\cap H_b^{\varepsilon_b}=\varnothing$.
Since any $\alpha$-$\beta$ surface intersects any $\alpha$-circle,
there exists 
$y\in(\mathcal{S}^+\setminus\mathcal{C}^+)\cap\Int(H_{i_1}^{\varepsilon_1})$.
According to Remark \ref{remarqueobjetsdynamiquessuitinverse},
$\mathcal{C}^+$, $\mathcal{S}^+$, $\mathcal{C}^-$ and $\mathcal{S}^-$
are respectively the repulsive and attractive 
circles and surfaces of the sequence $(\gamma_n^{-1})$
of unbalanced type $\beta$,
and 
we thus have $\mathcal{D}_{(\gamma_n^{-1})}(y)=\Calpha(y')$
with $y'\in\mathcal{C}^-$
according to Lemma \ref{lemmerepulsifX}.
Moreover 
$\Calpha(y')=\mathcal{D}_{(\gamma_n^{-1})}(y)\subset H_a^{\varepsilon_a}$
because $\gamma_n^{-1}=g_a^{-\varepsilon_a}\dots g_{i_1}^{-\varepsilon_1}$
and $y\in\Int(H_{i_1}^{\varepsilon_1})$.
Now $\mathcal{S}^-\cap H_b^{\varepsilon_b}\neq\varnothing$ since
any $\beta$-$\alpha$ surface meets any $\beta$-circle
but $\mathcal{S}^-\cap H_b^{\varepsilon_b}$ is disjoint 
from $\Calpha(y')$,
because $\Calpha(y')\subset H_a^{\varepsilon_a}$
and $H_a^{\varepsilon_a}$ is disjoint from $H_b^{\varepsilon_b}$.
In other words, 
denoting $p=\pi_\alpha(y')\in\RP{2}$,
the compact set $\pi_\beta(\mathcal{S}^-\cap H_b^{\varepsilon_b})$
of $\RP{2}_*$ is disjoint from
$p^*=\enstq{D\in\RP{2}_*}{D\not\ni p}$,
where $\pi_\alpha$ and $\pi_\beta$ denote the two coordinate
projections of $\X$ on $\RP{2}$ and $\RP{2}_*$.
As before, this allows us
to slightly modify $H_b^{\varepsilon_b}$
into a neighbourhood ${H'}_b^{\varepsilon_b}$ of $B_b^{\varepsilon_b}$
satisfying the assumptions of Fact \ref{faitchangementsdomainefondamentaux} 
and such that $\pi_\beta(\mathcal{S}^-\cap H_b^{\varepsilon_b})\subsetneq
\pi_\beta(\mathcal{S}^-\cap \Int({H'}_b^{\varepsilon_b}))$.
The sequence $(\gamma_n({H'}_b^{\varepsilon_b}))$ 
converges thus to a connected
component $K'$ of $\Lambda(\Gamma)$ containing
$\bigcup_{x\in\mathcal{S}^-\cap\Int({H'}_b^{\varepsilon_b})}\Cbeta(\phi(x))$
whereas $\Bplus(\gamma_\infty)\subset \mathcal{C}^+\cup
\bigcup_{x\in\mathcal{S}^-\cap H_b^{\varepsilon_b}}\Cbeta(\phi(x))$.
This shows that $K'\neq \Bplus(\gamma_\infty)$
since $\phi(x)$ does only depend on $\pi_\beta(x)$.
As before, $K'$ and $\Bplus(\gamma_\infty)$ should then be disjoints
which contradicts the fact that they both contain 
$\bigcup_{x\in B_b^{\varepsilon_b}}\mathcal{D}_{(\gamma_n)}(x)$.
This final contradiction concludes the proof 
that $(\gamma_n)$ is of balanced type.
\end{proof}
We use the notations of Lemma \ref{lemmecasmixteX} for the dynamical
objects of the sequence $(\gamma_n)$ of balanced type.
In particular, $\Bmoins$ and $\Bplus$ denote its attractive and repulsive
bouquet of circles.
Any $\alpha$-$\beta$ surface meeting any $\alpha$-circle of $\X$
there exists a point 
$y\in(\Salphabeta^+\setminus\Sbetaalpha^+)\cap
\Int(H_{i_1}^{-\varepsilon_1})$,
and since $\gamma_n^{-1}=g_a^{-\varepsilon_a}\dots g_1^{-\varepsilon_1}$,
$\mathcal{D}_{(\gamma_n^{-1})}(y)\subset H_a^{-\varepsilon_a}$.
Since $\Salphabeta^+$ and $\Sbetaalpha^+$ (respectively $\Calpha^-$) 
are the repulsive surfaces (resp. attractive $\alpha$-circle) of
$(\gamma_n^{-1})$, $\mathcal{D}_{(\gamma_n^{-1})}(x)=\Calpha^-$
and thus $\Calpha^-\subset H_a^{-\varepsilon_a}$.
Analog arguments show that $\Cbeta^-\subset H_a^{-\varepsilon_a}$ and
$\Bmoins$ is thus disjoint from $H_b^{\varepsilon_b}$.
Therefore $\mathcal{D}_{(\gamma_n)}(x)\subset \Bplus$
for any $x\in H_b^{\varepsilon_b}$
and thus 
$\lim \gamma_n(H_b^{\varepsilon_b})=\Bplus(\gamma_\infty)\subset \Bplus$ 
according to Lemma \ref{lemmeconsequenceconvergencecompacts}.
Conversely there exists
$y\in(\Salphabeta^-\setminus\Sbetaalpha^-)\cap\Int(H_b^{\varepsilon_b})$
and thus 
$\mathcal{D}_{(\gamma_n)}(y)=\Calpha^+\subset \Bplus(\gamma_\infty)$,
and there exists 
$z\in(\Sbetaalpha^-\setminus\Salphabeta^-)\cap\Int(H_b^{\varepsilon_b})$
and thus 
$\mathcal{D}_{(\gamma_n)}(z)=\Cbeta^+\subset \Bplus(\gamma_\infty)$.
This concludes the proof that $\Bplus(\gamma_\infty)$
is the attractive bouquet of circles of any subsequence going simply
to infinity
of the finite subwords of $\gamma_\infty$. \\
3. We now 
consider a sequence $\gamma_k\in\Gamma$ going simply to infinity
in $\PGL{3}$.
In particular $(\gamma_k)$ goes to infinity in $\Gamma$
for the word metric
and we consider an accumulation point $\delta_\infty\in\partial_{\infty}\Gamma$
of $(\gamma_k)$.
Passing to a subsequence we can assume
that $\gamma_k=\delta_{n_k}\mu_k$
with $(\delta_{n_k})$
a subsequence 
of the finite subwords of $\delta_\infty$
and $\mu_k\in\Gamma$ always finishing with the same letter 
$g_a^{\varepsilon_a}$.
We choose a letter $g_b^{\varepsilon_b}\neq g_a^{-\varepsilon_a}$
and passing to a subsequence again we can moreover assume that
$\gamma_k(H_b^{\varepsilon_b})$ 
converges for the Hausdorff topology to a compact subset
$K_\infty$ of $\X$.
We proved previously that $\Bplus(\delta_\infty)=\lim K(\delta_{n_k})$
is a bouquet of two circles.
Since $\gamma_k(H_b^{\varepsilon_b})\subset K(\delta_{n_k})$ for any $k$
we have $K_\infty\subset \Bplus(\delta_\infty)$.
Let us assume by contradiction that $(\gamma_k)$ 
is of unbalanced type $\alpha$ and denote by $\mathcal{C}^-\subset \mathcal{S}^-$ and
$\mathcal{C}^+\subset\mathcal{S}^+$ 
its repulsive and attractive circles and surfaces in $\X$
and by $\phi\colon\X\to\mathcal{C}^+$ the fibration introduced in 
Lemma \ref{lemmeattractifX}.
According to Lemma \ref{lemmeconsequenceconvergencecompacts},
$K_\infty$ contains
$\bigcup_{x\in(\mathcal{S}^-\setminus\mathcal{C}^-)\cap \Int(H_b^{\varepsilon_b})}\Cbeta(\phi(x))$,
but $(\mathcal{S}^-\setminus\mathcal{C}^-)\cap \Int(H_a^{\varepsilon_a})$
is a non-empty open subset of $\mathcal{S}^-\setminus\mathcal{C}^-$
since any $\beta$-$\alpha$ surface intersects any $\beta$-circle,
and $K_\infty$
contains thus a topological disc which contradicts
$K_\infty\subset \Bplus(\gamma_\infty)$.
We obtain a contradiction in the same way if we assume
$(\gamma_k)$ to be of unbalanced type $\beta$.
\par Hence $(\gamma_k)$ is of balanced type
and we denote by $\Salphabeta^-$, $\Sbetaalpha^-$
$\Calpha^+$ and $\Cbeta^+$ its repulsive surfaces and attractive
circles.
As before $K_\infty$ contains $\Calpha^+$ since
$(\Salphabeta^-\setminus\Sbetaalpha^-)\cap \Int(H_a^{\varepsilon_a})
\neq\varnothing$
and contains $\Cbeta^+$
since $(\Sbetaalpha^-\setminus\Salphabeta^-)\cap \Int(H_a^{\varepsilon_a})
\neq\varnothing$.
As $K_\infty\subset \Bplus(\delta_\infty)$
this proves that $\Bplus(\gamma_k)=\Calpha^+\cup\Cbeta^+=K_\infty=\Bplus(\delta_\infty)$.
By injectivity of $x\in\partial_{\infty}\Gamma\mapsto \Bplus(x)$,
this shows in particular that $\delta_\infty$ is the only
accumulation point of $(\gamma_k)$, 
which converges thus to $\delta_\infty$
by compacity of $\Gamma\cup\partial_{\infty}\Gamma$.
This concludes the proof of our claim. \\
4. This follows readily from the previous claims.
\end{proof}

We can now summarize our results as follows.
\begin{corollary}\label{corollarySchottky}
Let $\Gamma=\langle g_1,\dots,g_d \rangle$ be a Schottky subgroup of $\PGL{3}$.
\begin{enumerate}
\item With $\{H_i^{\pm}\}_{i=1}^d$ any set of separating handlebodies for the $g_i$,
$\X\setminus\bigcup_{i=1}^d(H^-_i\cup H^+_i)$ 
is a fundamental domain for the action of $\Gamma$ on 
its maximal open subset of discontinuity $\Omega(\Gamma)$.
\item Moreover, $\Gamma\backslash\Omega(\Gamma)$ is homeomorphic to a 
  closed three-manifold
 obtained from the flag space $\X$ after succesively performing $d$ times
 the following two operations:
 \begin{enumerate}[label=(\Alph*)]
  \item Remove the interior of two disjoint embedded genus two
  handlebodies $H^-$ and $H^+$.
  \item Glue the two boundary components 
  $\partial H^-$ and $\partial H^+$ of the resulting
  three-manifold with boundary
  by a diffeomorphism $f\colon \partial H^-\to \partial H^+$.
 \end{enumerate}
\end{enumerate}
\end{corollary}
\begin{proof}
1. This is a straightforward reformulation of Propositions 
\ref{propositiondomainefondamentalSchottky}
and \ref{propositiondescriptionensemblelimite}. \\
2. Let us assume that $d=2$, that is 
$\Gamma=\langle g_1,g_2 \rangle$.
Then $\Gamma\backslash\Omega$ is homeomorphic
 to $\sim\backslash \bar{U}$, 
 with $U=\X\setminus\cup_{i=1,2}(H^-_i\cup H^+_i)$
 and $\sim$ the equivalence relation generated by the relations
 $x\sim g_i(x)$ for $i=1,2$ and any $x\in \partial H^-_i$.
 Denoting $\Gamma_1=\langle g_1 \rangle$ and
 $\Omega_1=\X\setminus(B_1^-\cup B_1^+)$, 
 the topology of the quotient
 $M_1=\Gamma_1\backslash\Omega_1$ is obtained from the one of $\X$
 by performing the operations (A) and (B) described in the statement.
 Indeed $H_1^-$ and $H_1^+$ are disjoint embedded genus two handlebodies
 in $\X$ and 
 $M_1=\sim_1\backslash(\X\setminus(\Int H_1^-\cup \Int H_1^+))$,
 where $x\sim_1 g_1(x)$ for any $x\in\partial H_1^-$.
 If $\pi_1\colon\Omega_1\to M_1$ denotes the canonical projection,
 $H^-=\pi_1(H_2^-)$ and $H^+=\pi_1(H_2^+)$
 are two disjoint embedded genus two handlebodies in $M_1$.
 Moreover for any $\bar{x}\in\partial H^-$ there exists
 an unique $x\in\partial H_2^-$ such that $\bar{x}=\pi_1(x)$,
 since $\partial H_2^-\subset U_1$ and 
 $\pi_1\restreinta_{U_1}$ is injective.
 This allows to defines a diffeomorphism 
 $f\colon \partial H^-\to \partial H^+$ by
 $f(\bar{x})=\pi_1\circ g_2(x)$,
 for any $\bar{x}\in\partial H^-$ and 
 $x\in\partial H_2^-$ such that $\bar{x}=\pi_1(x)$.
 Now $\Gamma\backslash\Omega=\sim\backslash\bar{U}$
 is homeomorphic to 
 $\sim_2\backslash(M_1\setminus(\Int H^-\cup \Int H^+))$,
 where $x\sim_2 f(x)$ for any $x\in\partial H^-$,
 hence the topology of $\Gamma\backslash\Omega$ is obtained
 from the one of $\X$ after succesively performing two times 
 the operations (A) and (B).
 The claim immediately follows by a finite recurrence argument.
\end{proof}

\section{Path structures compactifications of geodesic flows}
\label{sectionTheoremeB}
Let $\overline{h}_1,\dots,\overline{h}_d$
be hyperbolic elements of $\PSL{2}$
having pairwise distincts repulsive and attractive fixed points
in the boundary $\partial\Hn{2}$ of the hyperbolic plane $\Hn{2}$,
and for which we choose representatives $h_i\in\SL{2}$ 
with positive eigenvalues.
We introduce the embedding
\begin{equation}\label{equationdefinitionj}
 j\colon h\in\GL{2}\mapsto
 \begin{bmatrix}
  h & 0 \\
  0 & 1
 \end{bmatrix}\in\PGL{3}
\end{equation}
and we define $g_i\coloneqq j(h_i)\in\PGL{3}$.
Each $g_i$ is then a loxodromic element of $\PGL{3}$ 
(having positive eigenvalues)
with repulsive and attractive fixed points $p_i^\pm$
parwise distincts on $[e_1,e_2]$
and with $[e_3]$ as a common saddle point
(see Example \ref{exemplemixteRP2dz} for these definitions).
In particular, the $g_i$ are in general position
and according to Proposition \ref{propdefSchottkydansX}
there exists for each $i$ some $r_i>0$,
such that $g_1^{r_1},\dots,g_d^{r_d}$ 
generates a Schottky subgroup of $\PGL{3}$.
We replace each $h_i$ by $h_i^{r_i}$, 
and we denote
$\overline{\Gamma}_0=\langle \overline{h}_1,\dots,\overline{h}_d \rangle$,
$\Gamma_0=\langle h_1,\dots,h_d \rangle$ 
(note that $h_i\mapsto\overline{h}_i$ defines an isomorphism
from $\Gamma_0$ to $\overline{\Gamma}_0$)
and
\[
\Gamma=j(\Gamma_0)=\langle g_1,\dots, g_d \rangle\subset j(\SL{2}).
\]
We work from now on with the hyperbolic surface 
$\Sigma=\overline{\Gamma}_0\backslash\Hn{2}$
whose geodesic flow on $\Fiunitan{\Sigma}$
is denoted by $(g^t)$.

\subsection{Hyperbolic surfaces and path structures}
\label{soussectionstructureLCaudessussurfaceshyperboliques}
We first define the path structure
$\Lm_\Sigma$ that we will study on $\Fiunitan{\Sigma}$.
Let us recall that a \emph{path structure} on a three-dimensional manifold
is a couple $\Lm=(E^\alpha,E^\beta)$
of one-dimensional distributions whose sum is a contact
distribution,
and that an \emph{isomorphism} between path structures is a diffeomorphism
sending $\alpha$-distribution on $\alpha$-distribution,
and $\beta$-distribution on $\beta$-distribution.

\subsubsection{An invariant path structure 
on $\Fiunitan{\Sigma}$}\label{soussoussectiondefLSigma}
Let us consider on $\SL{2}$ the left-invariant one-dimensional
distributions $E^\alpha$ and $E^\beta$ respectively generated by the
elements
\begin{equation*}\label{equationdefinitionLSL2PSL2}
 E=\begin{pmatrix}
  0 & 1 \\
  0 & 0
 \end{pmatrix}
 \text{~and~}
 F=\begin{pmatrix}
  0 & 0 \\
  1 & 0
 \end{pmatrix}
\end{equation*}
of its Lie algebra $\slR{2}$.
Then $\Lm_{\SL{2}}=(E^\alpha,E^\beta)$ 
is a left-invariant path structure on $\SL{2}$, and
the same construction defines on $\PSL{2}$ a left-invariant path 
structure $\Lm_{\PSL{2}}$ for which the 
two-sheeted covering $\SL{2}\to\PSL{2}$
is a local isomorphism.
We recall that, $\PSL{2}$ acting simply transitively 
on $\Fiunitan{\Hn{2}}$, 
we can identify $\Fiunitan{\Sigma}$ with 
$\overline{\Gamma}_0\backslash\PSL{2}$.
This quotient inherits from $\PSL{2}$
a natural path structure
and we denote by $\Lm_{\Sigma}$ the corresponding structure 
on $\Fiunitan{\Sigma}$.
If $\Sigma$ is compact, 
the geodesic flow is Anosov
and the same construction defines
a path structure $\Lm_\Sigma$
whose $\alpha$ (respectively $\beta$) direction is the stable 
(resp. unstable) distribution of the geodesic flow.
\begin{lemma}\label{lemmaLSigmainvarianteflotgeodesique}
 The path structure $(\Fiunitan{\Sigma},\Lm_{\Sigma})$ 
 is invariant by the geodesic flow of $\Sigma$.
\end{lemma}
\begin{proof}
We recall first that the geodesic flow
of $\Hn{2}$ is conjugated in $\PSL{2}$
to the right translations $R_{a^{t/2}}$,
where
\begin{equation*}\label{equationdefinitionht}
 a^t=
 \begin{bmatrix}
  \e^t & 0 \\
  0 & \e^{-t}
 \end{bmatrix},
\end{equation*}
and that the geodesic flow $(g^t)$ of $\Sigma$ is thus conjugated
to $R_{a^{t/2}}$ on $\overline{\Gamma}_0\backslash\PSL{2}$.
But the adjoint action of $a^t$ 
preserves for any $t$ the lines $\R E$ and $\R F$ in $\slR{2}$,
and $R_{a^t}$ preserves thus $\Lm_{\SL{2}}$, 
showing that $\Lm_\Sigma$ is invariant by the geodesic flow.
\end{proof}

\begin{remark}\label{remarkuniciteLSL2}
 It is in fact not difficult to show that, 
 modulo inversion of its distributions $E^\alpha$ and $E^\beta$,
 $\Lm_{\PSL{2}}$ is 
 the only $\PSL{2}$-invariant path structure of $\PSL{2}$
 which is also $(R_{a^t})$-invariant (see for instance \cite[Lemme 1.1.14]{mmthese}).
 In other words, $\Lm_{\Sigma}$ is the only path structure of $\Fiunitan{\Sigma}$
 invariant by the geodesic flow
 that comes from an invariant path structure on $\PSL{2}$,
 and is in this sense the most natural path structure that we could look at
 on $\Fiunitan{\Sigma}$.
\end{remark}

\subsubsection{A Kleinian path structure}
\label{soussoussectionlsigmaestkleinienne}
The algebraic construction that we made has in fact a natural geometrical counterpart.
$\SL{2}$ can indeed be identified with its only open orbit
\begin{equation*}
Y\coloneqq
  \X\setminus(\Sbetaalpha[e_1,e_2]\cup\Salphabeta[e_3])
\end{equation*}
in $\X$, which makes of $(\Fiunitan{\Sigma},\Lm_{\Sigma})$
the quotient of $Y$ by a discrete subgroup of $\PGL{3}$.
This is a particular instance of what is called
a \emph{Kleinian} path structure,
that is the quotient of an open subset of $\X$
by a discrete subgroup of $\PGL{3}$.
We introduce the following notations:
\begin{equation*}\label{equationdefinitiongzerohatGamma}
o'\coloneqq([1,0,1],[(1,0,1),e_2])\in\X,
 g_0=
 \begin{bmatrix}
  -1 & 0 & 0 \\
  0 & -1 & 0 \\
  0 & 0 & 1
 \end{bmatrix}=j(-\id),
  \hat{\Gamma}=
  \Gamma\cup g_0\Gamma.
\end{equation*}  
\begin{lemma}\label{lemmaproprietesYSL2}
 \begin{enumerate}
  \item $Y$ is the $j(\SL{2})$-orbit of $o'$,
  $j(\SL{2})$ acts simply transitively on $Y$ and 
  $\theta_{o'}\colon h\mapsto h\cdot o'$ is an
  isomorphism of path
structures from $(\SL{2},\Lm_{\SL{2}})$ to $(Y,\Lx\restreinta_Y)$.
\item $\hat{\Gamma}$ is a discrete subgroup of $j(\SL{2})$.
\item $\Lx$ induces a flat path structure on
$\hat{\Gamma}\backslash Y$ which is isomorphic to 
$(\Fiunitan{\Sigma},\Lm_\Sigma)$.
 \end{enumerate}
\end{lemma}
\begin{proof}
1. This follows from straightforward calculations
 (detailed for instance in \cite[\S 4.2.2]{mionmouton}). \\
2. Let us assume by contradiction that $\gamma_n\in\hat{\Gamma}$
is a non-stationnary sequence converging to $\id$.
Then either some subsequence of $(\gamma_n)$ is contained in $\Gamma$
or a subsequence of $(g_0^{-1}\gamma_n)$ does.
In both cases this contradicts the discreteness of $\Gamma$ or
the non-stationnary nature of $(\gamma_n)$, 
proving that $\hat{\Gamma}$ is indeed discrete. \\
3. Since $\hat{\Gamma}$ is discrete, it is closed in $j(\SL{2})$,
and the action of $\hat{\Gamma}$ is thus free and proper on $Y$.
Now $\theta_{o'}$ defines an isomorphism from
$\overline{\Gamma}_0\backslash\PSL{2}\simeq\Fiunitan{\Sigma}$ to
$\hat{\Gamma}\backslash Y$, proving our claim.
\end{proof}

\subsection{A first compactification}
\label{soussoussectioncompctificationrevetementindicedeux}
We saw in Proposition \ref{propositiondescriptionensemblelimite}
that for $\gamma_\infty\in\partial_{\infty}\Gamma$,
any subsequence going simply to infinity
of the sequence of finite subwords 
of $\gamma_\infty$
has balanced dynamics with
$\Bplus(\gamma_\infty)$
as attractive bouquet of circles.
Since $\Gamma\subset j(\SL{2})$,
the description of attractive circles in
Lemma \ref{lemmecasmixteX} moreover shows that,
with $p_+(\gamma_\infty)\in[e_1,e_2]$
the attractive point in $\RP{2}$
of the sequence of finite subwords of 
$\gamma_\infty$ 
(see Lemma \ref{lemmemixteRP2}),
we have
\begin{equation}\label{equationdescriptionBplusgammainfty}
 \Calpha^+(\gamma_\infty)=\Calpha(p_+(\gamma_\infty))
\text{~and~} 
\Cbeta^+(\gamma_\infty)=\Cbeta[e_3,p_+(\gamma_\infty)].
\end{equation}
Note that $p_+\colon\partial_\infty\Gamma\to[e_1,e_2]$
is an homeomorphism onto its image.
According to Proposition \ref{propositiondomainefondamentalSchottky},
$\Gamma$ acts freely, properly and cocompactly
on 
\[
\Omega=\X\setminus\Lambda
\text{~with~}
\Lambda=\bigcup_{\gamma_\infty\in\partial_{\infty}\Gamma}\Bplus(\gamma_\infty).
\]
In particular, since
$\Calpha(p_+(\gamma_\infty))\subset\Sbetaalpha[e_1,e_2]$ and
$\Cbeta[p_+(\gamma_\infty),e_3]\subset\Salphabeta[e_3]$
we obtain $Y\subset\Omega$,
wich directly provides us with a first compactification result.
\begin{proposition}\label{propositioncompactificationSL2surGamma}
 $\Gamma\backslash\Omega$ is a path structure compactification
 of the Kleinian structure $\Gamma\backslash Y$, 
 where $\Gamma\backslash Y$ embedds as an open and dense subset.
\end{proposition}
\begin{proof}
 The inclusion $Y\subset \Omega$ induces an embedding $j$
 of path structures 
 of $\Gamma\backslash Y$ 
 in the closed three-manifold $\Gamma\backslash\Omega$.
 Moreover, $Y$ being dense in $\Omega$, $j(\Gamma\backslash Y)$
 is dense in $\Gamma\backslash\Omega$.
\end{proof}

According to Lemma \ref{lemmaproprietesYSL2},
$\Gamma\backslash\Omega$ is two-sheeted covering of 
$(\Fiunitan{\Sigma},\Lm_{\Sigma})$,
the non-trivial automorphism of this covering being induced
by $g_0$.
A naive way to obtain a compactification of 
$(\Fiunitan{\Sigma},\Lm_{\Sigma})$ should be to
take the quotient of $\Gamma\backslash\Omega$ by $g_0$.
But $g_0$ has a lot of fixed points on $\Omega$: 
it acts trivially on $\Calpha[e_1]\cup\Cbeta[e_1,e_2]$
and on  
$\mathcal{C}=\enstq{(p,D)}{p\in[e_1,e_2],D\ni[e_3]}$,
whose intersections with $\Omega$ are non-empty.
This prevents us from obtaining a smooth
quotient of $\Omega$ by $\Gamma$,
and leads us to consider a covering of $\X$
where $g_0$ will have no fixed points in the preimage of $\Omega$.

\subsection{A journey in a covering of $\X$}
\label{soussectionjourneyXhat}
Natural two-sheeted coverings of $\X$ are given by
the space $\mathbf{P}(\Fitan{\Sn{2}})$ of tangent lines of $\Sn{2}$
and the space $\mathbf{P}^+(\Fitan{\RP{2}})$ of tangent half-lines 
of $\RP{2}$, 
both endowed with natural actions of $\PGL{3}$
and natural path structures given by the pullbacks of 
$\Lx$.
But $g_0$ acts trivially on the $\alpha$-circle
defined by $e_3$ in $\mathbf{P}(\Fitan{\Sn{2}})$,
and on the $\beta$-circle defined by $(e_1,e_2)$ in 
$\mathbf{P}^+(\Fitan{\RP{2}})$,
whose intersections with the preimage of $\Omega$ are non-empty.
Hence these coverings are not enough
and we have to consider the next one, that is
the space 
$\Xhat=\mathbf{P}^+(\Fitan{\Sn{2}})$
of tangent half-lines of $\Sn{2}$.
We can also think to $\Xhat$ as the set of \emph{oriented flags}
$(d,P)$ of $\R^3$, $d$ being an oriented line of $\R^3$ 
contained in an oriented plane $P$.
For 
$(u,v)$ two non-conlinear vectors of $\R^3$, 
we will denote by
$(u,v)$ the plane $\Vect(u,v)$ oriented by its basis $(u,v)$,
and by
$(u,(u,v))$ the corresponding point $(\R^+u,(u,v))$ of $\Xhat$.
Note that $\Xhat$ is diffeomorphic to
$\Fiunitan{\Sn{2}}$.
In particular, $\Xhat$ is orientable and has $\Sn{3}$
as a double-cover.
$\Xhat$ is a four-sheeted covering of $\X$ through the projection
\[
\pi\colon (d,P)\in\Xhat \mapsto ([d],[P])\in\X,
\]
and we endow $\Xhat$ with the path structure
$\Lm_{\Xhat}=\pi^*\Lx$.
The $\alpha$ and $\beta$-leaves of $\Lm_{\Xhat}$
are circles that we denote 
by $\Chatalpha$ and $\Chatbeta$.
For $x\in\X$, 
$\pi^{-1}(\Calpha(x))$ (respectively $\pi^{-1}(\Cbeta(x))$)
is the disjoint union of two $\alpha$-circles in $\Xhat$
(resp. of two $\beta$-circles),
the restriction of $\pi$ to an $\alpha$ or $\beta$-circle 
of $\Xhat$
being a double covering $\Sn{1}\to\RP{1}$.
There is a natural action of $\GL{3}$ on $\Xhat$
and since the projection $g\mapsto[g]$
of $\SL{3}$ in $\PGL{3}$ is an isomorphism
we can define an action of $\PGL{3}$ on $\Xhat$ by the formula
\[
[g]\cdot x\coloneqq g\cdot x \text{~for any~} g\in\SL{3} 
\text{~and~} x\in\Xhat.
\]
This action preserve the orientation of $\Xhat$ and
makes $\pi\colon\Xhat\to\X$ 
equivariant for the respective actions of $\PGL{3}$. 
In particular, $\PGL{3}$ preserves $\Lm_{\Xhat}$.
Observe that the action of the special orthogonal group 
$\SO{3}$ is simply transitive on $\Xhat$.
\par To give a better picture of the covering $\Xhat$,
let us look more closely at the 
surfaces $\Shatalphabeta(x)=\cup_{y\in\Chatalpha(x)}\Chatbeta(y)$
and $\Shatbetaalpha(x)=\cup_{y\in\Chatbeta(x)}\Chatalpha(y)$
for $x\in\Xhat$.
\begin{lemma}\label{lemmesurfaceshatalphabeta}
 For any $x\in\Xhat$, $\Shatalphabeta(x)$ and $\Shatbetaalpha(x)$ are tori. Furthermore, 
 $\Xhat\setminus\Shatalphabeta(x)$ and $\Xhat\setminus\Shatbetaalpha(x)$
 have two connected components.
\end{lemma}
\begin{proof}
Since the involution $\kappa$ defined in \eqref{equationdefinitionkappa}
 switches $\alpha$-$\beta$ and $\beta$-$\alpha$ surfaces,
 it is sufficient to prove it for $\Shatbetaalpha(x)$,
 and by transitivity of $\PGL{3}$, it is sufficient to prove it
 for $\Shatbetaalpha(e_1,e_2)=\enstq{(d,P)\in\Xhat}{d\in(e_1,e_2)}$.
 The equality
 \[
 \Shatbetaalpha(e_1,e_2)=\bigcup_{(A,B)\in\SO{2}^2}
 \begin{pmatrix}
  A & 0 \\
  0 & 1
 \end{pmatrix}
 \begin{pmatrix}
  1 & 0 \\
  0 & B
 \end{pmatrix}\cdot (e_1,(e_1,e_2))
 \]
 proves that this surface is a torus.
 Furthermore, 
 $\Xhat\setminus\Shatbetaalpha(e_1,e_2)=
 \enstq{(d,P)\in\Xhat}{d\notin(e_1,e_2)}$
 is disconnected since its projection $\Snhat{2}\setminus(e_1,e_2)$ 
 on $\Snhat{2}$ has two connected components $C_1$ and $C_2$.
 Since $\pi^{-1}(C_1)$ and $\pi^{-1}(C_2)$ are both connected
 (they are in fact solid tori), 
 $\Xhat\setminus\Shatbetaalpha(e_1,e_2)=\pi^{-1}(C_1)\cup\pi^{-1}(C_2)$
 has two connected components.
\end{proof}

The following lemma shows that the subgroup $\hat{\Gamma}$
acts as we wish on $\hat{\Omega}\coloneqq\pi^{-1}(\Omega)$.
We point out related results in \cite[\S 7.2]{stecker_domains_2018}, 
where the authors describe cocompact domains of discontinuity 
for \emph{purely hyperbolic generalized Schottky subgroups} of
$\PSL{2n+1}$ acting on oriented flag spaces.
\begin{lemma}\label{lemmaactionhatGamma}
\begin{enumerate}
 \item $\Gamma$ acts freely, properly and cocompactly on $\hat{\Omega}$.
 \item $g_0$ has no fixed points on $\Xhat$.
 \item $\hat{\Gamma}$ preserves $\Omega$ and $\hat{\Omega}$.
 \item $\hat{\Gamma}$ acts freely and properly on $\hat{\Omega}$.
 \item $M\coloneqq\hat{\Gamma}\backslash\hat{\Omega}$ is an
orientable and compact three-dimensional manifold.
\end{enumerate}
\end{lemma}
\begin{proof}
1. Since $\pi$ is a $\Gamma$-equivariant covering,
$\Gamma$ acts as freely and properly on $\hat{\Omega}$ as it does on $\Omega$.
The covering
$\bar{\pi}\colon\Gamma\backslash\hat{\Omega}\to\Gamma\backslash\Omega$ 
induced by $\pi$ having finite fibers and $\Gamma\backslash\Omega$ being 
compact,
$\Gamma\backslash\hat{\Omega}$ is compact as well. \\
2. The only fixed points of the action of $g_0$ on $\Sn{2}$
are $e_3$ and $-e_3$,
so that fixed points of $g_0$ on $\Xhat$ are in 
$\Chatalpha(e_3)\cup\Chatalpha(-e_3)$.
But for $p=e_3$ or $-e_3$, the action of $g_0$ on $\Chatalpha(p)$
is conjugated to the action of $-\id$ on $\mathbf{P}^+(\R^2)$
and has thus no fixed point. \\
3. We saw in Paragraph 
\ref{soussoussectioncompctificationrevetementindicedeux} 
that the attractive bouquet of 
any $\gamma_\infty\in\partial_{\infty}\Gamma$ is of the form 
$\Bplus(\gamma_\infty)=\Calpha(p_+)\cup\Cbeta(D_+)$ 
with $p_+\in[e_1,e_2]$ and $[e_3]\in D_+$.
Since $g_0$ fixes $[e_3]$ and acts trivially on $[e_1,e_2]$
it stabilizes $\Bplus(\gamma_\infty)$
and stabilizes thus
$\Omega=\X\setminus\bigcup_{\delta_\infty\in\partial_{\infty}\Gamma}\Bplus(\delta_\infty)$
according to Proposition \ref{propositiondescriptionensemblelimite}.
Therefore, $\hat{\Gamma}$ stabilizes $\Omega$
and thus $\hat{\Omega}$ since $\pi$ is $\PGL{3}$-equivariant. \\
4. Since $\Gamma$ acts freely on $\hat{\Omega}$,
we only need to show that for any $\gamma\in\Gamma$, 
$g_0\gamma\neq\id$ has no fixed point on $\hat{\Omega}$
to prove that the action is free.
Let us assume by contradiction that $g_0\gamma\cdot x=x$
with $x\in\hat{\Omega}$.
Then for any $x\in\N$, $(g_0\gamma)^{2n}\cdot x=x$.
But $(g_0\gamma)^{2n}=\gamma^{2n}$
because $\gamma$ commutes with $g_0$ which is of order two,
hence $\gamma^{2n}\cdot x=x$.
Since $\pi(x)\in\Omega$ is outside the repulsive bouquet of circles of 
$(\gamma^{2n})$
(equal to $\Bplus(\gamma_\infty)$ with 
$\gamma_\infty=\gamma^{-2}\gamma^{-2}\gamma^{-2}\dots\in\partial_{\infty}\Gamma$),
some subsequence of 
$\gamma^{2n}\cdot\pi(x)$ converges
to a point of the attractive bouquet of circles of $(\gamma^{2n})$,
hence to $\X\setminus\Omega$.
This contradicts
$\gamma^{2n}\cdot\pi(x)=\pi(\gamma^{2n}\cdot x)=\pi(x)\in\Omega$
and shows that the action is free.
Passing to a subsequence and precomposing by $g_0^{-1}$, 
for any $\gamma_n\in\hat{\Gamma}$ going to infinity 
we can assume that $(\gamma_n)$ is contained in $\Gamma$.
Now for $(x_n)$ a converging sequence of $\hat{\Omega}$,
$(\gamma_n\cdot x_n)$ is not relatively compact
by property of the action of $\Gamma$
on $\hat{\Omega}$,
showing that the action of $\hat{\Gamma}$ is proper on 
$\hat{\Omega}$. \\
5. Since $\Xhat$ is compact and $\hat{\Gamma}\subset\PGL{3}$
preserves its orientation, $M$ is orientable
and compact as the image of the compact space 
$\Gamma\backslash\hat{\Omega}$
by the continuous projection
$\Gamma\cdot x\mapsto \hat{\Gamma}\cdot x$.
\end{proof}

\subsection{Compactification of the geodesic flow and 
proof of Theorem \ref{theoremintrosurfacehyperboliquedynamiqueflotgeod}}
We now come back to the path structure 
$(\Fiunitan{\Sigma},\Lm_\Sigma)$,
and denoting by $(g^t)$ its geodesic flow,
we describe the compactification of
$(\Fiunitan{\Sigma},\Lm_{\Sigma},g^t)$.

\subsubsection{Geometry of the compactification}
We denote by 
\[
 \Pi\colon\hat{\Omega}\to M=\hat{\Gamma}\backslash\hat{\Omega}
\]
the canonical projection,
and by $\Lm$ the path structure
of $M$ induced by $\Lm_{\Xhat}$.
We introduce
\[
\psi^t\coloneqq j(\e^t\id)=
\begin{bmatrix}
 \e^t & 0 & 0 \\
 0 & \e^t & 0 \\
 0 & 0 & 1
\end{bmatrix}.
\]
This is a flow of unbalanced type $\beta$ whose 
repulsive and attractive objects in $\X$
as described in Lemma \ref{lemmerepulsifX}
are denoted by
$\Calpha^-=\Calpha[e_3]$, $\Salphabeta^-=\Salphabeta[e_3]$,
$\Cbeta^+=\Cbeta[e_1,e_2]$ and $\Sbetaalpha^+=\Sbetaalpha[e_1,e_2]$.
We also introduce the circle
$\mathcal{C}=\enstq{(p,D)}{p\in[e_1,e_2],D\ni[e_3]}$ of $\X$.
We now define:
\begin{gather*}
 \Falphamoins=\Pi(\pi^{-1}(\Calpha^-\cap\Omega)),
 \Halphabetamoins=\Pi(\pi^{-1}(\Salphabeta^-\cap\Omega)),
 \Fbetaplus=\Pi(\pi^{-1}(\Cbeta^+\cap\Omega)),
 \Hbetaalphaplus=\Pi(\pi^{-1}(\Sbetaalpha^+\cap\Omega)), \\
 \Delta=\Pi(\pi^{-1}(\mathcal{C}\cap\Omega)).
\end{gather*}
$\Gamma$ acts freely, properly and cocompactly
on $[e_1,e_2]\setminus p_+(\partial_\infty\Gamma)$, and
$\Gamma\backslash([e_1,e_2]\setminus p_+(\partial_\infty\Gamma))$
is thus a finite union of circles.
We denote by $b\in\N^*$ the number of connected components
of this quotient,
which is the number of boundary components
of the topological compact surface with boundary whose interior is 
homeomorphic to $\Sigma$,
that is the number of funnels of the hyperbolic surface $\Sigma$
(considering the case of $d=1$ 
generators, that is the case where $\Sigma$ is a hyperbolic cylinder,
can be useful to understand these equalities).
\begin{proposition}\label{propositioncompactificationetflot}
\begin{enumerate}
 \item $\Falphamoins$ (respectively $\Fbetaplus$, 
 respectively $\Delta$)
 is the disjoint union of $b$ pairs of disjoint circles 
 $\{\Falphamoins_i\}_{i=1}^b$ 
 (resp. $\Fbetaplus_i$, resp. $\Delta_i$).
 $\Halphabetamoins$ (resp. $\Hbetaalphaplus$) is
 the disjoint union of $b$ tori $\{\Halphabetamoins_i\}_{i=1}^b$
 (resp. $\Hbetaalphaplus_i$). 
 Furthermore
 $\Falphamoins$, $\Fbetaplus$ and $\Delta$ are pairwise disjoint,
 $\Falphamoins_i\subset\Halphabetamoins_i$,
 $\Fbetaplus_i\subset\Hbetaalphaplus_i$ 
 and $\Halphabetamoins_i\cap\Hbetaalphaplus_i=\Delta_i$
 for each $i$,
 and 
 $\Halphabetamoins_i\cap\Hbetaalphaplus_j=\varnothing$ for any $i\neq j$.
\item There exists in $(M,\Lm)$ 
 four disjoint open sets $\{N_j\}_{j=1}^4$
 isomorphic to $(\Fiunitan{\Sigma},\Lm_\Sigma)$.
 Furthermore, $M\setminus\sqcup_{j=1}^4 N_j=
 \Halphabetamoins\cup\Hbetaalphaplus$.
 \item The flow $(\psi^t)$ defines on $M$ 
 a flow $(\varphi^t)$ of automorphisms
 of $\Lm$, conjugated on each of the $N_j$ to $(g^{2t})$,
 where $(g^t)$ denotes the geodesic flow of $\Fiunitan{\Sigma}$.
\end{enumerate}
\end{proposition}
We emphasize that $(\varphi^t)$ is conjugated to $(g^{2t})$
and not to $(g^t)$.
We will however consider $(\varphi^t)$ rather than 
$(\varphi^{\frac{t}{2}})$ which would be quite inconvenient.
\begin{proof}[Proof or Proposition \ref{propositioncompactificationetflot}]
1. We first emphasize that $\Calpha^-$, $\Salphabeta^-$,
$\Cbeta^+$, $\Sbetaalpha^+$ and $\mathcal{C}$ are 
$\Stab[e_3]\cap\Stab[e_1,e_2]=j(\GL{2})$-invariant,
and thus $\hat{\Gamma}$-invariant.
Now, $\pi$ being $\PGL{3}$-equivariant
and $\hat{\Omega}$ being $\hat{\Gamma}$-invariant,
$\pi^{-1}(\Calpha^-\cap\Omega)$,
$\pi^{-1}(\Salphabeta^-\cap\Omega)$,
$\pi^{-1}(\Cbeta^+\cap\Omega)$,
$\pi^{-1}(\Sbetaalpha^+\cap\Omega)$
and $\pi^{-1}(\mathcal{C}\cap\Omega)$
are $\hat{\Gamma}$-invariant.
These are closed subsets of $\Omega$,
and their projections by $\Pi$ are thus closed in $M$,
hence compact,
which already proves that $\Falphamoins$, $\Fbetaplus$ and 
$\Delta$ are finite unions of circles.
For any connected component $I_k$ of 
$\Cbeta^+\cap\Omega$, 
$\pi^{-1}(I_k)$ has four connected components
and $g_0$ preserves $\pi^{-1}(I_k)$ and 
has two orbits on its space of connected components.
Since $\hat{\Gamma}=\langle\Gamma,g_0\rangle$,
this shows that the space of connected components of
$\Fbetaplus=\hat{\Gamma}\backslash\pi^{-1}(\Cbeta^+\cap\Omega)$
surjects with a fiber of cardinal two
onto the one of $\Gamma\backslash(\Cbeta^+\cap\Omega)$.
The same happens between $\Fbetaplus$ 
(respectively $\Delta$)
and $\Gamma\backslash(\Cbeta^+\cap\Omega)$
(resp. $\Gamma\backslash(\mathcal{C}\cap\Omega)$).
But $\Gamma\backslash(\Calpha^-\cap\Omega)$,
$\Gamma\backslash(\Cbeta^+\cap\Omega)$
and $\Gamma\backslash(\mathcal{C}\cap\Omega)$
have the same number of connected components
than $\Gamma\backslash([e_1,e_2]\setminus p_+(\partial_\infty\Gamma))$,
that is $b$,
which proves the claim concerning $\Falphamoins$, $\Fbetaplus$
and $\Delta$.
\par Since the compact surfaces $\Halphabetamoins$ and $\Hbetaalphaplus$ 
bear smooth one-dimensional distributions, they have Euler characteristic equal to zero
according to Poincar\'e-Hopf Theorem,
and we only have to check that they are indeed orientable to prove
that they are finite unions of tori.
We saw in Paragraph \ref{soussoussectioncompctificationrevetementindicedeux}
that
$\Lambda=\cup_{\gamma_\infty\in\partial_{\infty}\Gamma}
\Calpha(p_+(\gamma_\infty))\cup\Cbeta[p_+(\gamma_\infty),e_3]$
where $p_+(\gamma_\infty)\in[e_1,e_2]$,
and thus
$\Calpha^-\cap\Omega=\cup_k I_k$ with $\{I_k\}$ a collection
of disjoint intervals in the circle $\Calpha^-$.
Therefore $\Salphabeta^-\cap\Omega=\cup_k\cup_{x\in I_k}\Cbeta[x]$ 
is a union of cylinders, is thus orientable,
and $\pi^{-1}(\Salphabeta^-\cap\Omega)$ is orientable as well.
Since the action of $\PGL{3}$ 
preserves the orientation of these cylinders, their projections
in $M$ are orientable compact surfaces of Euler characteristic zero,
that is tori.
Finally $\Halphabetamoins$ is a finite union of tori,
and the same holds for $\Hbetaalphaplus$ for the same reasons.
The fact that the
number of connected components of $\Halphabetamoins$
(respectively $\Hbetaalphaplus$)
is half of the one of $\Falphamoins$ (resp. $\Fbetaplus$)
is deduced from the fact that for
any $x\in\X$, the preimage of $\Salphabeta(x)$
(resp. $\Sbetaalpha(x)$) in $\Xhat$ is connected,
whereas the preimage of $\Calpha(x)$
(resp. $\Cbeta(x)$) has two connected components.
\par The last claim directly follows from 
the fact that $\Calpha^-$, $\Cbeta^+$ and $\mathcal{C}$
are pairwise disjoint, 
and that $\Salphabeta^-\cap\Sbetaalpha^+=\mathcal{C}$. \\
2. We recall that $o'=([1,0,1],[(1,0,1),e_2])\in Y$
and we denote $\pi^{-1}(o')=\{\hat{o}_i\}_{i=1,\dots,4}$.
For any $i\neq k$, $\hat{o}_i$ and $\hat{o}_k$ are not in the same
$j(\SL{2})$-orbit.
Since $j(\SL{2})$ acts freely at $\hat{o}_i$,
the formula $\iota_i(g\cdot o')=g\cdot\hat{o}_i$
for any $g\in j(\SL{2})$ defines a
map $\iota_i\colon Y\to\hat{\Omega}$
descending for each $i=1,\dots,4$ to an embedding
$\bar{\iota}_i$ of 
$\hat{\Gamma}\backslash Y\simeq\Fiunitan{\Sigma}$
in the compact path structure $M$.
The images $N_i=\bar{\iota}_i(\hat{\Gamma}\backslash Y)$
of these embeddings
are disjoint as projections in $M$ of distinct
orbits of $j(\SL{2})$ and
the equality $M\setminus\cup_i N_i=\Halphabetamoins\cup\Hbetaalphaplus$ 
directly follows from 
$\X\setminus Y=\Salphabeta^-\cup\Sbetaalpha^+$. \\
3. We saw in Lemma \ref{lemmaLSigmainvarianteflotgeodesique}
that the geodesic flow of $\Sigma$ is conjugated
to $(R_{a^{t/2}})$ on $\bar{\Gamma}_0\backslash\PSL{2}$,
and the relation $j(ga^t)\cdot o'=j(\e^t\id)\cdot(j(g)\cdot o')$
for any $t\in\R$ and $g\in\SL{2}$ shows that 
$(R_{a^t})$ is itself conjugated in $Y$ to $(\psi^t)$.
Since $(\psi^t)$ acts trivially on $[e_1,e_2]$ and fixes $[e_3]$,
it preserves $\Omega$, and hence $\hat{\Omega}$.
Since $(\psi^t)$ commutes with $g_0$ and all the $g_i$, it
descends to
a flow of automorphisms of $M$ 
that we denote by $(\varphi^t)$,
conjugated to $(g^{2t})$
on each $N_i$.
\end{proof}

\subsubsection{Dynamics at infinity of the geodesic flow}\label{soussoussectiondynamiqueflotcompactifie}
We now describe the dynamics of $(\varphi^t)$ on $M$. 
The set of fixed points of $(\psi^t)$ on $\Sn{2}$
being $\{e_3\}\cup(e_1,e_2)$,
$\pi^{-1}(\Calpha^-\cup\Cbeta^+\cup\mathcal{C})$
is the set of fixed points of $(\psi^t)$ on $\Xhat$,
and each point of 
$\Falphamoins\cup\Fbetaplus\cup\Delta$
is thus a fixed point of $(\varphi^t)$.
\begin{lemma}\label{propositionpointsfixesvarphit}
 The set of fixed points of $(\varphi^t)$
 is precisely $\Falphamoins\cup\Fbetaplus\cup\Delta$.
\end{lemma}
\begin{proof}
1. Let $x\in\hat{\Omega}$ such that $\Pi(x)$ is a fixed point of $(\varphi^t)$.
For any $t\in\R$ there exists then $\gamma^t\in\hat{\Gamma}$ such that
$\psi^t(x)=\gamma^t(x)$, and such a $\gamma^t$ is unique since
$\hat{\Gamma}$ acs freely on $\hat{\Omega}$.
Moreover, $\gamma^{s+t}(x)=\psi^s(\gamma^t(x))=\gamma^t(\psi^s(x))
=\gamma^t\gamma^s(x)$ since $\hat{\Gamma}$ and $(\psi^t)$ commute,
hence $\gamma^{s+t}=\gamma^s\gamma^t$.
Finally $(\gamma^t)$ is a one-parameter subgroup of $\hat{\Gamma}$,
implying $\gamma^t=\id$ for any $t\in\R$ since $\hat{\Gamma}$ is discrete. 
Therefore $x$ is a fixed point of $(\psi^t)$, 
that is $x\in\pi^{-1}(\Calpha^-\cup\Cbeta^+\cup\mathcal{C})\cap\hat{\Omega}$,
which proves our claim.
\end{proof}

The dynamics of the flow $(\psi^t)$ 
described in Lemma \ref{lemmerepulsifX}
allow us to obtain an accurate picture of those
of $(\varphi^t)$.
We denote by $\phi_+\colon\X\to\Cbeta^+$ the application
associated in Lemma \ref{lemmerepulsifX} to the
flow $(\psi^t)$ of unbalanced type $\beta$,
and by $\phi_-\colon\X\to\Calpha^-$ the application
associated to its inverse $(\psi^{-t})$ of unbalanced type $\alpha$
in Lemma \ref{lemmeattractifX}.
\begin{proposition}\label{propositionproprietesprolongationflotgeod}
We introduce 
$\Upsilon^-=\Pi(\pi^{-1}(\phi_+^{-1}(\Lambda)\cap\Omega))$ and
$\Upsilon^+=\Pi(\pi^{-1}(\phi_-^{-1}(\Lambda)\cap\Omega))$.
\begin{enumerate}
 \item $\Upsilon^-$ and $\Upsilon^+$ are 
 contained
 in the union $\cup_{i=1}^4 N_i$ 
 of the copies of $\Fiunitan{\Sigma}$ in $M$.
 \item The closure of $\Upsilon^-$ (respectively $\Upsilon^+$)
 is equal to $\Upsilon^-\cup\Falphamoins$
 (resp. $\Upsilon^+\cup\Fbetaplus$).
 In particular,
 $M\setminus(\Halphabetamoins\cup\Upsilon^-)$
 and $M\setminus(\Hbetaalphaplus\cup\Upsilon^+)$
 are dense and open subsets of $M$.
 \item There exists two continuous applications 
 \[
 \Phi_+\colon 
 M\setminus(\Halphabetamoins\cup\Upsilon^-)\to\Fbetaplus
 \text{~and~} 
 \Phi_-\colon 
 M\setminus(\Hbetaalphaplus\cup\Upsilon^+)\to\Falphamoins
 \]
 such that $\mathcal{D}_{(\varphi^t)}(x)=\Phi_+(x)$ for any $x\in M\setminus(\Halphabetamoins\cup\Upsilon^-)$,
 and $\mathcal{D}_{(\varphi^{-t})}(x)=\Phi_-(x)$
 for any $x\in M\setminus(\Hbetaalphaplus\cup\Upsilon^+)$.
  \item For any $1\leq i\leq 4$, $\Upsilon^-\cap N_i$ (respectively $\Upsilon^+\cap N_i$)
 is (the image of)
 the subset of points of $\Fiunitan{\Sigma}$
 whose $\omega$-limit set (resp. $\alpha$-limit set)
 for the geodesic flow is non-empty.
\item Let $K\subset M\setminus(\Halphabetamoins\cup\Upsilon^-)$ 
(respectively $K\subset M\setminus(\Hbetaalphaplus\cup\Upsilon^+)$)
be a compact subset and $t_n\to+\infty$ (resp. $t_n\to-\infty$)
such that $\varphi^{t_n}(K)$ converges.
Then $\lim\varphi^{t_n}(K)\subset\Phi_+(K)$ (resp. $\lim\varphi^{t_n}(K)\subset\Phi_-(K)$).
If $K$ is the closure of its interior, then 
$\underset{+\infty}{\lim}\varphi^t(K)=\Phi_+(K)$
(resp. $\underset{-\infty}{\lim}\varphi^t(K)=\Phi_-(K)$).
\end{enumerate}
\end{proposition}
We recall the definition of the \emph{$\omega$-limit set} 
\[
\omega(x)=\enstq{\text{accumulation points of~}g^{t_n}(x)}{t_n\to+\infty}
\]
of $x$ for $g^t$, its $\alpha$-limit set being the corresponding subset for 
the sequences $t_n\to-\infty$.
\begin{proof}[Proof of Proposition \ref{propositionproprietesprolongationflotgeod}]
The arguments are formally the same in the past and in the future,
that is concerning $(\varphi^t)$ and $\Upsilon^-$,
and concerning $(\varphi^{-t})$ and $\Upsilon^+$.
We thus only write them for $(\psi^t)$. \\
1. and 2. Since $g_0$ acts trivially on $[e_1,e_2]$,
the $\Gamma$-invariant set $\Lambda$ 
is actually $\hat{\Gamma}$-invariant.
We saw in Lemma \ref{lemmerepulsifX}
that $\phi_+$ is equivariant with respect to a morphism
$\rho_\infty\colon\Stab[e_3]\to\Stab[e_3]\cap\Stab[e_1,e_2]$,
and the construction of $\rho_\infty$ in \eqref{equationrelationdefinitionrhoinfinirepulsif}
shows that $\rho_\infty$ is equal to the identity 
in restriction to $j(\GL{2})$ and thus
in restriction to $\hat{\Gamma}$.
Hence $\phi_+$ is $\hat{\Gamma}$-equivariant,
and $\phi_+^{-1}(\Lambda)$ is $\hat{\Gamma}$-invariant.
Now the description of $\phi_+$ in the proof of Lemma \ref{lemmerepulsifX}
and the description of $\Omega$ in Paragraph 
\ref{soussoussectioncompctificationrevetementindicedeux}
(see \eqref{equationdescriptionBplusgammainfty})
shows that
\[
\phi_+^{-1}(\Lambda)\cap\Omega=
\bigcup_{\gamma_\infty\in\partial_{\infty}\Gamma}
\Sbetaalpha[p_+(\gamma_\infty),e_3]\setminus
\Big(\Calpha[e_3]\cup\Cbeta[p_+(\gamma_\infty),e_3]\cup\Calpha(p_+(\gamma_\infty))\Big).
\]
In particular, $\phi_+^{-1}(\Lambda)\cap\Omega$ is disjoint
from $\Sbetaalpha^+$ and $\Salphabeta^-$
and $\Upsilon^-$ is thus disjoint from
$\Hbetaalphaplus$ and $\Halphabetamoins$,
hence contained in $\cup_{i=1}^4 N_i$.
Furthermore $\pi^{-1}(\phi_+^{-1}(\Lambda)\cap\Omega)$ 
is a $\hat{\Gamma}$-invariant subset of empty interior,
therefore $\Upsilon^-$ has empty interior.
Since $\phi_+$ is not continuous on $\Omega$,
$\phi_+^{-1}(\Lambda)\cap\Omega$ is not closed in $\Omega$.
However, $\phi_+$ being continous on $\X\setminus\Calpha^-$,
$\phi_+^{-1}(\Lambda)\cap\Omega$ is closed in $\Omega\setminus\Calpha^-$.
Hence $\Upsilon^-\setminus\Falphamoins$ is closed in $M\setminus\Falphamoins$, 
and the closure of $\Upsilon^-$ is contained 
in $\Upsilon^-\cup\Falphamoins$.
In particular, $\Cl(\Upsilon^-)$ has empty interior.
More precisely, let $\gamma_\infty\in\partial_\infty\Gamma$,
$p_n$ a sequence of $[p_+(\delta_\infty),e_3]$ 
converging to $[e_3]$,
and $p\in[e_1,e_2]\setminus p_+(\partial_\infty\Gamma)$.
Then $(p_n,[p_n,p])\in\phi_+^{-1}(\Lambda)\cap\Omega$
converges to $([e_3],[e_3,p])\in\Calpha^-\cap\Omega$.
This shows not only that the closure of $\Upsilon^-$
is equal to $\Upsilon^-\cup\Falphamoins$,
but that any connected component of $\Upsilon^-$
accumulate on one of the connected components of $\Falphamoins$.
In particular
$\Halphabetamoins\cup\Upsilon^-=
\Halphabetamoins\cup\Cl(\Upsilon^-)$
is a closed subset with empty interior,
and $M\setminus(\Halphabetamoins\cup\Upsilon^-)$
is an open and dense subset.\\
3. Let $x\in M\setminus(\Halphabetamoins\cup\Upsilon^-)$,
and let $x_n\in M$ converging to $x$ and $t_n\in\R$ to $+\infty$,
such that $\lim\varphi^{t_n}(x_n)=x_\infty\in\mathcal{D}_{(\varphi^t)}(x)$.
We choose $\hat{x}\in\Pi^{-1}(x)$, and there exists a sequence
$\hat{x}_n\in\Pi^{-1}(x_n)$ converging to $\hat{x}$.
Passing to a subsequence, we can furthermore assume that 
$\lim\psi^{t_n}(\hat{x}_n)=\hat{x}_\infty\in\Xhat$ 
by compacity of $\Xhat$.
Then $\bar{x}_n=\pi(\hat{x}_n)$ 
converges to $\bar{x}=\pi(\hat{x})\notin\Salphabeta^-$
and $\psi^{t_n}(\bar{x}_n)$ to 
$\bar{x}_\infty=\pi(\hat{x}_\infty)$.
According to Lemma \ref{lemmerepulsifX}
$\bar{x}_\infty=\phi_+(\bar{x})\in\Cbeta^+$
and since $x\notin\Upsilon^-$, 
$\bar{x}_\infty\in\Omega$ and we thus have 
$\hat{x}_\infty\in\pi^{-1}(\phi_+(\bar{x}))\subset\hat{\Omega}$.
In particular $\hat{x}_\infty\notin\Shatalphabeta(e_3)$.
Since $x\notin\Halphabetamoins$, 
$\hat{x}\in\Xhat\setminus\Shatalphabeta(e_3)$ which has two connected
components according to Lemma \ref{lemmesurfaceshatalphabeta}.
Since $((e_1),(e_1,e_2))$ and $((e_1),(e_1,-e_2))$
are not in the same connected component of 
$\Xhat\setminus\Shatalphabeta(e_3)$ and are fixed by $(\psi^t)$,
each of these components is preserved by $(\psi^t)$.
We denote by $C$ the connected component containing $\hat{x}$.
For $n$ large enough, $\hat{x}_n\in C$
and thus $\psi^{t_n}(\hat{x}_n)\in C$, 
showing that $\hat{x}_\infty\in C$.
We already saw that $\hat{x}_\infty\in\pi^{-1}(\phi_+(\bar{x}))$,
and $C\cap\pi^{-1}(\phi_+(\bar{x}))$ has cardinal two:
if $\phi_+(\bar{x})=(p,[e_1,e_2])$, 
then $C\cap\pi^{-1}(\phi_+(\bar{x}))=\{(\pm p,\varepsilon(e_1,e_2))\}$
with $\varepsilon$ the orientation corresponding to the connected 
component $C$.
Since $g_0\in\hat{\Gamma}$ identifies the two points
$(\pm p,\varepsilon(e_1,e_2))$,
$\Pi(C\cap\pi^{-1}(\phi_+(\bar{x})))$
is a point of $\Fbetaplus$ depending only on $x$,
that we denote by $\Phi_+(x)$.
We have shown that $\mathcal{D}_{(\varphi^t)}(x)\subset\{\Phi_+(x)\}$,
but $\mathcal{D}_{(\varphi^t)}(x)\neq\varnothing$ 
since $M$ is compact, and this inclusion is thus an equality.
The continuity of $\Phi_+$ on $M\setminus\Falphamoins$ follows
from the one of $\phi_+$ on $\X\setminus\Calpha^-$, proved
in Lemma \ref{lemmerepulsifX}. \\
4. Let $x\in\Fiunitan{\Sigma}$ whose $\omega$-limit set is non-empty, 
and $y$ be the corresponding point in one of the copies $N_i$,
with respect to an isomorphism
conjugating $(\varphi^t)$ with the geodesic flow.
Then if $y\notin\Upsilon^-$ by contradiction,
the $\omega$-limit set of $y$ for $(\varphi^t)$ would be
disjoint from $N_i$ according to the previous claim,
and the $\omega$-limit set of $x$ for the geodesic flow
would thus be empty.
Conversely, let $x\in\Upsilon^-$ contained in the copy $N_i$ of 
$\Fiunitan{\Sigma}$, and 
$y$ be the corresponding point of $\Fiunitan{\Sigma}$.
Let $t_n\to+\infty$ such that $\lim\varphi^{t_n}(x)=x_\infty$
in the $\omega$-limit set of $x$ for $(\varphi^t)$.
With $\hat{x}\in\Pi^{-1}(x)$ 
and $\bar{x}=\pi(\hat{x})$,
passing to a subsequence
we can assume that $\psi^{t_n}(\bar{x})$ converges in $\X$, 
and then $\lim\psi^{t_n}(\bar{x})=\phi_+(\bar{x})\in\Lambda$
by hypothesis. 
By cocompacity of the action of $\Gamma$ on $\Omega$,
there exists
a sequence $\gamma_n\in\Gamma$ such that
$\gamma_n\psi^{t_n}(\bar{x})$ is relatively compact in $\Omega$,
and we can assume that 
$\gamma_n\psi^{t_n}(\bar{x})$ converges to $\bar{x}_\infty\in\Omega$,
possibly taking a new subsequence.
Since $\phi_+^{-1}(\Lambda)\cap\Omega$ is invariant
by $\Gamma$ and by $(\varphi^t)$,
$\gamma_n\psi^{t_n}(\bar{x})\in\phi_+^{-1}(\Lambda)\cap\Omega$,
and
$\bar{x}_\infty\in\phi_+^{-1}(\Lambda)\cap\Omega\cup\Calpha^-$
since $\phi_+$ is continuous on $\X\setminus\Calpha^-$.
Let us temporarily assume 
that $\bar{x}_\infty\notin\Calpha^-$,
which will be proved thereafter.
Then $x_\infty\in\Upsilon$, which shows
that the $\omega$-limit set of $y$ for $(\varphi^t)$
is contained in $\Upsilon$, and thus in $N_i$.
Therefore the $\omega$-limit set of $x$ 
for the geodesic flow is non-empty, finishing the proof
\par It only remains to prove that $\bar{x}_\infty\notin\Calpha^-$.
Since $\psi^{t_n}(\bar{x})$ goes to infinity in $\Omega$,
$\gamma_n$ goes to infinity in $\Gamma$,
and passing to a subsequence we can assume that $\gamma_n$ goes simply to infinity.
According to Proposition \ref{propositiondescriptionensemblelimite},
$\gamma_n$ converges then 
to a point $\gamma_\infty\in\partial_\infty\Gamma$,
is of balanced type, and has $\Bplus(\gamma_\infty)$ as attractive
bouquet of circles.
Since $\bar{x}_\infty=\lim\gamma_n\psi^{t_n}(\bar{x})\notin\Lambda$,
in particular $\bar{x}_\infty\notin \Bplus(\gamma_n)$,
which implies 
$\lim\psi^{t_n}(\bar{x})=\phi_+(\bar{x})\in \Bmoins(\gamma_n)$
according to Lemma \ref{lemmecasmixteX}.
Since the repulsive point $x_-(\gamma_n)$ of $(\gamma_n)$ is of the form
$(p_+(\delta_\infty),[p_+(\delta_\infty),e_3])$
for some $\delta_\infty\in\partial_\infty\Gamma$,
we have more precisely $\lim\psi^{t_n}(\bar{x})\in
\Calpha^-(\gamma_n)\setminus\{x_-(\gamma_n)\}$
and thus $\bar{x}_\infty\in\Salphabeta^+(\gamma_n)$
according to Lemma \ref{lemmecasmixteX} again.
Since $\bar{x}_\infty\in\Omega$ and
$\Salphabeta^+(\gamma_n)\cap\Calpha^-=([e_3],[e_3,p_+(\gamma_n)])
\in\Cbeta^+(\gamma_n)\subset\Lambda$,
this shows that $\bar{x}_\infty\notin\Calpha^-$ and concludes
the proof. \\
5. Let $t_n\in\R$ be a sequence 
such that $t_n\to+\infty$.
According to the third claim of this proposition
$\cup_{x\in K}\mathcal{D}_{(\varphi^{t_n})}(x)=\Phi_+(K)$,
which proves that $\lim\varphi^{t_n}(K)\subset\Phi_+(K)$
according to Lemma \ref{lemmeconsequenceconvergencecompacts}.
If moreover $K=\Cl(\Int K)$, then
$\Phi_+(K)=
\Cl(\cup_{x\in\Int K}\mathcal{D}_{(\varphi^{t_n}}(x))$
by continuity of $\Phi_+$ and thus $\lim\varphi^{t_n}(K)=\Phi_+(K)$
according to Lemma \ref{lemmeconsequenceconvergencecompacts} again,
proving $\underset{t\to+\infty}{\lim}\varphi^t(K)=\Phi_+(K)$
since this is true for any sequence
$t_n\to+\infty$.
\end{proof}

We endow $M$ with a Riemannian metric, and
denoting $\Lm=(E^\alpha,E^\beta)$ and $E^c=\R\frac{d\varphi^t}{dt}$ 
the direction of the flow on $M\setminus(\Falphamoins\cup\Delta\cup\Fbetaplus)$,
we say that $x\in M$ is a \emph{positively regular} point
if the following \emph{exponential growth rates} exist:
\begin{equation}\label{equationLyapunov}
\lambda_c^+(x)=\underset{t\to+\infty}{\lim}
\frac{\ln\norme{\Diff{x}{\varphi^t}\restreinta_{E^c}}}{t},
 \lambda_\alpha^+(x)=
 \underset{t\to+\infty}{\lim}
\frac{\ln\norme{\Diff{x}{\varphi^t}\restreinta_{E^\alpha}}}{t},
\lambda_\beta^+(x)=
 \underset{t\to+\infty}{\lim}
\frac{\ln\norme{\Diff{x}{\varphi^t}\restreinta_{E^\beta}}}{t}.
\end{equation}
Likewise, we say that $x$ is \emph{negatively regular}
if the limits
\[
\lambda_c^-(x)=\underset{t\to+\infty}{\lim}
\frac{\ln\norme{\Diff{x}{\varphi^{-t}}\restreinta_{E^c}}}{-t},
 \lambda_\alpha^-(x)=
 \underset{t\to+\infty}{\lim}
\frac{\ln\norme{\Diff{x}{\varphi^{-t}}\restreinta_{E^\alpha}}}{-t},
\lambda_\beta^-(x)=
 \underset{t\to+\infty}{\lim}
\frac{\ln\norme{\Diff{x}{\varphi^{-t}}\restreinta_{E^\beta}}}{-t}.
\]
exist.
Note that the set of positively (respectively negatively) regular points
is preserved by $(\varphi^t)$
and that each function $\lambda_{c/\alpha/\beta}^\pm$ 
is $(\varphi^t)$-invariant.
Moreover, these real numbers are independent of the Riemannian metric
since $M$ is compact.
A linear reparametrization of the flow by dilation of the time by $\lambda>0$
multiplies these exponential rates by $\lambda$,
and we recall that $(\varphi^t)$ is conjugated to $(g^{2t})$ 
on each copy of $\Fiunitan{\Sigma}$ in $M$, with $(g^t)$ the geodesic flow of $\Sigma$
(hence the difference  of a factor $\frac{1}{2}$ with the corresponding claims of Theorem
\ref{theoremintrosurfacehyperboliquedynamiqueflotgeod}).
\begin{proposition}\label{propositionLyapunovcompactification}
Any $x\in M\setminus(\Halphabetamoins\cup\Upsilon^-)$ is positively regular, and
$\lambda_c^+(x)=-1$, $\lambda_\alpha^+(x)=-1$, $\lambda_\beta^+(x)=0$.
Any $x\in M\setminus(\Hbetaalphaplus\cup\Upsilon^+)$ is negatively regular, and
$\lambda_c^-(x)=1$, $\lambda_\alpha^-(x)=0$, $\lambda_\beta^-(x)=1$.
\end{proposition}
\begin{proof}
Let $x\in M\setminus(\Halphabetamoins\cup\Upsilon^-)$ and
$\hat{x}\in\Pi^{-1}(x)$.
According to Proposition \ref{propositionproprietesprolongationflotgeod}
there exists $\hat{x}_\infty\in\Pi^{-1}(\Phi_+(x))$
such that $\underset{t\to+\infty}{\lim}\psi^t(\hat{x})=
\hat{x}_\infty$.
Denoting by $\norme{\cdot}'$ the pullback of $\norme{\cdot}$
on $\hat{\Omega}$, for $\varepsilon=\alpha$ or $\beta$ we have
$\lim
\norme{\Diff{x}{\varphi^t}\restreinta_{E^\varepsilon}}
=\lim
\norme{\Diff{\hat{x}}{\psi^t}
\restreinta_{\hat{\mathcal{E}}_\varepsilon}}'$.
Furthermore for any $g\in j(\GL{2})$,
denoting by $\norme{\cdot}''$ the pushforward of $\norme{\cdot}'$
by $g$ on $g(\hat{\Omega})$ we have
$\lim
\norme{\Diff{\hat{x}}{\psi^t}
\restreinta_{\hat{\mathcal{E}}_\varepsilon}}'
=\lim
\norme{\Diff{g(\hat{x})}{g\psi^t g^{-1}}
\restreinta_{\hat{\mathcal{E}}_\varepsilon}}''
=\lim
\norme{\Diff{g(\hat{x})}{\psi^t}
\restreinta_{\hat{\mathcal{E}}_\varepsilon}}''$,
because $j(\GL{2})$ centralizes $(\psi^t)$.
We can thus assume that $\hat{x}_\infty=((e_1),(e_1,e_2))$
and make the calculations for a Riemannian metric defined around
$\hat{x}_\infty$.
The claim that we want to prove being $(\psi^t)$-invariant,
we can moreover assume that $\hat{x}$ is as close to $\hat{x}_\infty$ as we want,
and we will assume that $\hat{x}=\phi_1(p)$ 
with $p\in\R^3$ and $\phi_1\colon\R^3\to\Xhat$ the chart defined around 
$\hat{x}_\infty$ by
\begin{equation}\label{equationchart1}
 \phi_1(x,y,z)=\left(
 \left(\begin{smallmatrix}
  1 \\ x \\ y
 \end{smallmatrix}\right),
 \left(
 \left(\begin{smallmatrix}
  1 \\ x \\ y
 \end{smallmatrix}\right),
 \left(\begin{smallmatrix}
  0 \\ 1 \\ z
 \end{smallmatrix}
 \right)
 \right)
 \right).
\end{equation} 
In this chart $\hat{x}_\infty=\phi_1(0,0,0)$,
and our claim is then reduced to the study of 
$\Psi_1^t\coloneqq\phi_1\circ\psi^t\circ\phi_1^{-1}$
in the neighbourhood of $(0,0,0)$,
with respect to 
a Riemannian metric defined around $(0,0,0)$
that we denote by $\norme{\cdot}$ to simplify the notations.
Since $\Psi_1^t=\Diag(1,\e^{-t},\e^{-t})$ is linear,
$(\psi^t)$ is thus linearizable in the neighbourhood of $\Phi_+(x)$.
We have $E^\alpha_1\coloneqq\phi_1^*\Extildealpha=\R X^\alpha_1$
and $E^\beta_1\coloneqq\phi_1^*\Extildebeta=\R X^\beta_1$
with $X^\alpha_1$ the constant vector field $(0,0,1)$
and $X^\beta_1(x,y,z)=(1,z,0)$.
Denoting by $X^c_1=\frac{d\Psi_1^t}{dt}$ the derivative of the flow,
since $X^\beta_1$ and $X^c_1$
are preserved by $(\Psi^t)$
and $\Diff{p}{\Psi_1^t}(X^\alpha_1(p))=
\e^{-t}X^\alpha_1(\Psi_1^t(p))$, we obtain
\[
\norme{\Diff{p}{\Psi_1^t}\restreinta_{E^c_1}}
=\frac{\norme{X^c_1(\Psi_1^t(p))}}{\norme{X^c_1(p)}},
\norme{\Diff{p}{\Psi_1^t}\restreinta_{E^\alpha_1}}
=\e^{-t}\frac{\norme{X^\alpha_1(\Psi_1^t(p))}}{\norme{X^\alpha_1(p)}},
\norme{\Diff{p}{\Psi_1^t}\restreinta_{E^\beta_1}}
=\frac{\norme{X^\beta_1(\Psi_1^t(p))}}{\norme{X^\beta_1(p)}}.
\]
This concludes our claim
concerning points of $M\setminus(\Halphabetamoins\cup\Upsilon^-)$ and positive times,
since 
$\norme{X^\alpha_1}$ and $\norme{X^\beta_1}$ are non-zero and
bounded at the neighbourhood of $(0,0,0)$,
and since $\norme{X^c_1(\Psi_1^t(p))}=\e^{-t}\norme{X^c_1(p)}$.
\par For $x\in M\setminus(\Hbetaalphaplus\cup\Upsilon^+)$ 
and negative times, we elaborate on the same arguments,
assuming that 
$\underset{t\to+\infty}{\lim}\psi^{-t}(\hat{x})=((e_3),(e_3,e_2))$
with $\hat{x}\in\Pi^{-1}(x)$
and that $\hat{x}=\phi_2(p)$ with $p\in\R^3$ and $\phi_2\colon\R^3\to\Xhat$
the chart defined around $((e_3),(e_3,e_2))$
by
\begin{equation}\label{equationchart2}
 \phi_2(x,y,z)=\left(
 \left(\begin{smallmatrix}
  x \\ y \\ 1
 \end{smallmatrix}\right),
 \left(
 \left(\begin{smallmatrix}
  x \\ y \\ 1
 \end{smallmatrix}\right),
 \left(\begin{smallmatrix}
  z \\ 1 \\ 0
 \end{smallmatrix}
 \right)
 \right)
 \right).
\end{equation}
Then $\Psi_2^{-t}\coloneqq\phi_2^{-1}\circ\psi^{-t}\circ\phi_2
=\Diag(\e^{-t},\e^{-t},1)$,
$\phi_2^*\Extildealpha=\R X^\alpha_2$
and $\phi_2^*\Extildebeta=\R X^\beta_2$,
with $X^\alpha_2=(0,0,1)$
and $X^\beta_2(x,y,z)=(z,1,0)$.
Since $X^\alpha_2$ is preserved by $(\Psi_2^{-t})$
and $\Diff{p}{\Psi_2^{-t}}(X^\beta_2(p))=
\e^{-t}X^\beta_2(\Psi_2^{-t}(p))$,
this concludes our claim in the same way than before.
\end{proof}
This concludes the proof of Theorem 
\ref{theoremintrosurfacehyperboliquedynamiqueflotgeod}.

\subsubsection{New essential automorphisms of path structures}
In particular, we deduce from the previous results
the following properties of the flow $(\varphi^t)$.
We say that a diffeomorphism $f$ is \emph{non-conservative} if $f$
does not preserves any absolutely continuous measure of total support.
\begin{proposition}\label{corollarydynamique}
 For any $t\neq0$: 
  \begin{enumerate}
  \item the group generated by $\varphi^t$ is not
relatively compact for the compact-open topology;
  \item $\varphi^t$ has a dense subset of wandering points,
  and is thus non-conservative;
    \item $\varphi^{t}$ is not a partially hyperbolic
  diffeomorphism of $M$.
  \end{enumerate}
\end{proposition}
\begin{proof}
1. This is a direct consequence of
Proposition \ref{propositionproprietesprolongationflotgeod}. \\
2. This is also a direct consequence of 
Proposition \ref{propositionproprietesprolongationflotgeod},
but we give a detailed argument.
We denote $f=\varphi^t$.
Let $x\in M\setminus(\Halphabetamoins\cup\Upsilon^-\cup\Fbetaplus)$
and $K\subset M\setminus(\Halphabetamoins\cup\Upsilon^-\cup\Fbetaplus)$ 
be a closed topological ball centered at $x$.
For any sequence $(x_n)$ converging to $x$
and $k_n\to+\infty$,
$x_n\in K$ for $n$ large enough.
According to Proposition \ref{propositionproprietesprolongationflotgeod},
$f^{n_k}(K)$ converges to a compact subset of $\Fbetaplus$,
therefore $f^{k_n}(x_n)\notin K$ for $n$ large enough and
$f^{k_n}(x_n)$ does not converge to $x$.
This shows that any point of 
$M\setminus(\Halphabetamoins\cup\Upsilon^-\cup\Fbetaplus)$
is a wandering point
and proves our claim since
$\Halphabetamoins\cup\Upsilon^-$ has empty
interior according to Proposition 
\ref{propositionproprietesprolongationflotgeod},
and $\Fbetaplus$ is a finite union of circles according
to Proposition \ref{propositioncompactificationetflot}. \\
3. Up to conjugation of $\Gamma$ in $j(\SL{2})$
we can assume that $x_0\coloneqq((e_1,),(e_1,e_2))\in\hat{\Omega}$.
Then $f(x_0)=x_0$, and
we saw in the proof of Proposition \ref{propositionLyapunovcompactification} 
that $\Diff{x_0}{\varphi^t}$ is conjugated in the chart
$\phi_1$ (see \eqref{equationchart1}) to
the diagonal matrix $\Diag(1,\e^{-t},\e^{-t})$.
This matrix being not partially hyperbolic,
this shows that $f$ is not partially hyperbolic.
\end{proof}

We recall that an automorphism flow $(\varphi^t)$ of a path structure
$(E^\alpha,E^\beta)$ is said to be \emph{strongly essential}
if it does not preserve any continuous one-dimensional 
distribution transverse
to $E^\alpha\oplus E^\beta$.
\begin{proposition}\label{propositionLmessentielle}
$(\varphi^t)$ is a strongly essential flow
of the path structure $(M,\Lm)$.
\end{proposition}
\begin{proof}
We assume by contradiction that a continuous transverse 
 distribution
does exist, and we consider its pullback on $\hat{\Omega}$ by $\Pi$.
This is a $(\psi^t)$-invariant continuous
one-dimensional distribution on $\hat{\Omega}$
denoted by $E^c$,
transverse
to the contact distribution of $\Lm_{\Xhat}=(E^\alpha,E^\beta)$.
We saw in Paragraph \ref{soussoussectioncompctificationrevetementindicedeux}
that $\Omega=\X\setminus\cup_{\gamma_\infty\in\partial_{\infty}\Gamma}
(\Calpha(p_+(\gamma_\infty))\cup\Cbeta[p_+(\gamma_\infty),e_3])$.
In particular, up to conjugation of $\Gamma$
in $j(\GL{2})$ we can assume in this proof that both
$x_0\coloneqq(e_1,(e_1,e_2))$ and $y_0\coloneqq(e_3,(e_3,e_2))$
are points of $\hat{\Omega}$.
We consider the chart $\phi_1\colon\R^3\to\Xhat$ 
defined in \eqref{equationchart1} around $x_0=\phi_1(0,0,0)$.
The closed subset
$K=\enstq{x\in\R}{[1,x,0]\in\cup_{\gamma_\infty\in\partial_{\infty}\Gamma}
p_+(\gamma_\infty)}$ verifies $0\notin K$ 
(since $x_0\in\hat{\Omega}$)
and $\phi_1^{-1}(\hat{\Omega})=\R^3\setminus
(K\times\{0\}\times\R)$.
We have $\phi_1^*(E^\alpha\oplus E^\beta)=\Ker(zdx-dy)$,
and we denote $E^c_1=\phi_1^*E^c$.
For any $(\lambda,\mu)\in(\R^*)^2$:
\begin{equation}\label{equationconjugaisonflotdanscarte}
 \phi_1^{-1}\circ\Diag(\lambda,\mu,1)\circ\phi_1
 =\Diag(\lambda^{-1}\mu,\lambda^{-1},\mu^{-1}).
\end{equation}
We can conjugate $\Gamma$ in the stabilizer of 
$(x_0,y_0)$ in $j(\GL{2})$, equal to
$\{\Diag(\lambda,\mu,1)\}$.
But \eqref{equationconjugaisonflotdanscarte} shows
that this stabilizers acts transitively on the tangent directions $D$
at $(0,0,0)$ transverse to 
$\Vect(e_1,e_3)=\phi_1^*(E^\alpha\oplus E^\beta)(0,0,0)$
and distinct from $\R e_2$.
Up to conjugation in $j(\GL{2})$, 
we can thus assume that $E^c_1(0,0,0)\neq\R(0,1,1)$,
which will be important later.
\par We denote $\Psi_1^t\coloneqq\phi_1^{-1}\circ\psi^t\circ\phi_1
 =\Diag(1,\e^{-t},\e^{-t})$.
For any $(x,y,z)\in O\coloneqq\enstq{(x,y,z)\in\R^3}{x\notin K}$,
denoting 
$E^c_1(x,y,z)=\R(a,b,c)$ we have
$E^c_1(x,\e^{-t}y,\e^{-t}z)=\Diff{(x,y,z)}{\Psi_1^t}(E^c_1)
=\R(a,\e^{-t}b,\e^{-t}c)$ by $\Psi_1^t$-invariance
(note that $O$ is $(\Psi_1^t)$-invariant).
If $a\neq0$, then 
$E^c_1(x,0,0)=\R e_1$ by continuity of $E_1^c$,
which contradicts the transversality 
with $\phi_1^*(E^\alpha\oplus E^\beta)$.
Hence $a=0$, $b\neq0$ by transversality 
with $\phi_1^*(E^\alpha\oplus E^\beta)$,
$E^c_1(x,\e^{-t}y,\e^{-t}z)$ is equal to $\R(0,1,\frac{c}{b})$
for any $t\in\R$
and $E^c_1(x,0,0)=\R(0,1,\frac{c}{b})$ by continuity of $E^c_1$.
There exists thus a continuous $\R$-valued function $\lambda$
on $\R\setminus K$
such that
$E^c_1(x,y,z)=\R(0,1,\lambda(x))$ 
for any $(x,y,z)\in O$.
Furthermore, $\lambda(0)\neq1$ since
$E_1^c(0,0,0)\neq\R(0,1,1)$.
\par We now consider 
the chart $\phi_2\colon\R^3\to\Xhat$ around 
$y_0=\phi_2(0,0,0)$ defined in \eqref{equationchart2}.
With 
$U\coloneqq\phi_2^{-1}(\phi_1(\R^3))=\{
x\neq0,zyx^{-1}\neq1\}$ and
$V\coloneqq\phi_1^{-1}(\phi_2(\R^3))=\{
y\neq0,zxy^{-1}\neq1\}$,
the transition maps $\phi_2^{-1}\circ\phi_1\colon V\to U$
and $\phi_1^{-1}\circ\phi_2\colon U\to V$ are given by
\[
 \phi_2^{-1}\circ\phi_1(x,y,z)=
 \left(y^{-1},xy^{-1},\frac{zy^{-1}}{zxy^{-1}-1}\right),
 \phi_1^{-1}\circ\phi_2(x,y,z)=
 \left(yx^{-1},x^{-1},\frac{zx^{-1}}{zyx^{-1}-1}\right).
\]
Denoting $E^c_2=\phi_2^*E^c$,
since $E^c_1(x,y,z)=\R(0,1,\lambda(x))$
a straightforward calculation gives
\[
E^c_2(x,y,z)=\R(x^2,xy,x(z-\lambda(yx^{-1}))(zyx^{-1}-1)^2)
\]
for any $(x,y,z)\in \phi_2^{-1}\circ\phi_1(O)$.
Now for $0<x<1$ small enough, 
$(x,x^2,1)\in\phi_2^{-1}\circ\phi_1(O)$,
and $E^c_2(x,x^2,1)=\R(x^2,x^3,x(1-\lambda(x))(x-1)^2)
=\R(x,x^2,(1-\lambda(x))(x-1)^2)$
converges at $x=0$ to $\R(0,0,(1-\lambda(0)))=\R e_3$ 
since $\lambda(0)\neq1$.
Hence $E^c_2(0,0,1)=\R e_3$ by continuity,
but $\R e_3=(\phi_2^*E^\alpha)(0,0,1)$.
This contradicts
the transversality of $E^c$ and $E^\alpha\oplus E^\beta$
and concludes the proof.
\end{proof}

\subsection{About other compactifications of $\Fiunitan{\Sigma}$}
\label{soussectionautrescompactifications}
We conclude this paper by describing two other geometrical compactifications
of $\Fiunitan{\Sigma}$.

\subsubsection{A second path structure on $\Fiunitan{\Sigma}$}
\label{soussoussectionsecondepathstructure}
We defined in Paragraph \ref{soussoussectionautrescompactifications}
of the introduction a natural path structure $\Lm_S^{proj}$
defined on the unitary tangent bundle of any Riemannian surface $S$.
The \emph{projective class} of the metric of $S$,
that is its set of unparametrized geodesics,
is actually sufficient to define an analog path structure
on the projectivization $\PFitan{S}$ of the tangent bundle of $S$
(that we will 
still denote $\Lm_S^{proj}$ by a slight misuse of notations).
The natural two-sheeted covering $\Fiunitan{S}\to\PFitan{S}$ 
is a local isomorphism between these path structures.
\par If $\Sigma$ is a non-compact hyperbolic surface,
finding a compactification of $(\PFitan{\Sigma},\Lm_\Sigma^{proj})$
is thus equivalent to find a \emph{projective compactification} of $\Sigma$,
that is a compact projective surface $S$ with a \emph{projective copy} of $\Sigma$ --
an open subset $U\subset S$ and a diffeomorphism between $U$ and $\Sigma$
mapping unparametrized geodesics to unparametrized geodesics.
Such a projective compactification is given by \cite[Theorem 1.1]{choi_topological_2017}.
It is interesting to note that the projective compactification $S$ constructed
by Choi-Goldman contains two disjoint projective copies of 
the surface $\Sigma$
and that $(\PFitan{S},\Lm_S^{proj})$ contains thus two disjoint copies
of $(\PFitan{\Sigma},\Lm_\Sigma^{proj})$,
which is reminiscent of the four copies of 
$(\Fiunitan{\Sigma},\Lm_\Sigma)$ appearing 
in Theorem \ref{theoremintrosurfacehyperboliquedynamiqueflotgeod}
and raises the following:
\begin{questionintro}\label{questionprojective}
 Does there exist a projective compactification containing 
 a dense copy of the complete non-compact
 hyperbolic surface $\Sigma$ ?
\end{questionintro}

\subsubsection{A conformal Lorentzian compactification of $\Fiunitan{\Sigma}$} 
We recall from Paragraph \ref{soussoussectionautrescompactifications} that unlike the path structure
$\Lm_\Sigma=(E^s,E^u)$ that we have studied in the whole section \ref{sectionTheoremeB},
the previous path structure $\Lm_\Sigma^{proj}$
is \emph{not} invariant by the geodesic flow $(g^t)$ of a complete hyperbolic surface $\Sigma$.
Now consider the one-form $\theta$ defined on $\Fiunitan{\Sigma}$ by
$\theta\restreinta_{E^s\oplus E^u}\equiv 0$ and $\theta(X^c)\equiv 1$,
where $X^c=\frac{d g^t}{dt}$.
Then $h_\Sigma(v^s,v^u)=h_\Sigma(v^u,v^s)\coloneqq d\theta(v^s,v^u)$ 
for any $(v^s,v^u)\in E^s\times E^u$ and 
$h_\Sigma(E^s\oplus E^u,X^c)=0$
defines on $\Fiunitan{\Sigma}$ a \emph{Lorentzian metric} $h_\Sigma$.
This third geometric structure on $\Fiunitan{\Sigma}$
is actually closer to the focus of the present work,
as the geodesic flow $(g^t)$ acts by isometries of $h_\Sigma$.
One can weaken this structure by considering its \emph{conformal class} $[h_\Sigma]$,
that is the set of all Lorentzian metrics 
$\e^f h_\Sigma$ for $f\colon\Fiunitan{\Sigma}\to\R$ a smooth function.
If $\Sigma$ is non-compact, 
Frances describes in \cite[\S 4.6]{frances_sur_2005} 
(as a consequence of a more general work about conformal Lorentzian structures)
a \emph{conformal compactification} 
of $(\Fiunitan{\Sigma},[h_\Sigma],g^t)$, where the geodesic flow extends
to a flow of conformal automorphisms (diffeomorphisms preserving the conformal class).
We emphasize that 
unlike the compactifications 
for the path
structures $\Lm_\Sigma$ and $\Lm_\Sigma^{proj}$,
the image of $\Fiunitan{\Sigma}$ is
a dense subset of its Lorentzian conformal compactification.

\bibliographystyle{alpha}
\bibliography{bibliocompactifications}

\begin{thebibliography}{GGKW17}

\bibitem[Ale21]{alexandre_nilclosed_2021}
Raphaël Alexandre.
\newblock Closed ray nil-affine manifolds and parabolic geometries.
\newblock {\em {arXiv}:2111.09586 [math]}, November 2021.

\bibitem[Bar01]{barbot_flag_2001}
Thierry Barbot.
\newblock Flag {Structures} on {Seifert} {Manifolds}.
\newblock {\em Geometry \& Topology}, 5(1):227--266, March 2001.

\bibitem[Bar10]{barbot}
Thierry Barbot.
\newblock Three-dimensional {Anosov} flag manifolds.
\newblock {\em Geometry \& Topology}, 14(1):153--191, 2010.

\bibitem[BPS19]{bochi_anosov_2019}
Jairo Bochi, Rafael Potrie, and Andrés Sambarino.
\newblock Anosov representations and dominated splittings.
\newblock {\em Journal of the European Mathematical Society (JEMS)},
  21(11):3343--3414, 2019.

\bibitem[Car24]{cartan_sur_1924}
\'Elie Cartan.
\newblock Sur les variétés à connexion projective.
\newblock {\em Bulletin de la Société Mathématique de France}, 52:205--241,
  1924.

\bibitem[CG17]{choi_topological_2017}
Suhyoung Choi and William Goldman.
\newblock Topological tameness of {Margulis} spacetimes.
\newblock {\em American Journal of Mathematics}, 139(2):297--345, 2017.

\bibitem[{\v C}S09]{capslovak}
Andreas {\v C}ap and Jan Slov\'ak.
\newblock {\em Parabolic geometries {I} Background and general theory}, volume
  154 of {\em Mathematical Surveys and Monographs}.
\newblock American Mathematical Society, Providence, RI, 2009.

\bibitem[FMMV21]{falbel_cartan_2021}
E.~Falbel, M.~Mion-Mouton, and J.~M. Veloso.
\newblock Cartan connections and path structures with large automorphism
  groups.
\newblock {\em International Journal of Mathematics}, page 2140016, 2021.

\bibitem[Fra04]{franceslorentziankleinian}
Charles Frances.
\newblock Lorentzian {Kleinian} groups.
\newblock {\em Commentarii Mathematici Helvetici}, 80, January 2004.

\bibitem[Fra05]{frances_sur_2005}
Charles Frances.
\newblock Sur les variétés lorentziennes dont le groupe conforme est
  essentiel.
\newblock {\em Mathematische Annalen}, 332(1):103--119, May 2005.

\bibitem[FT15]{falbel}
Elisha Falbel and Rafael Thebaldi.
\newblock A flag structure on a cusped hyperbolic 3-manifold.
\newblock {\em Pacific Journal of Mathematics}, 278(1):51--78, September 2015.

\bibitem[GdlH90]{ghysdlH}
\'Etienne Ghys and Pierre de~la Harpe.
\newblock {\em Sur les {Groupes} {Hyperboliques} d’après {Mikhael}
  {Gromov}}.
\newblock Progress in {Mathematics}. Birkhäuser Basel, 1990.

\bibitem[GGKW17]{guichard}
François Guéritaud, Olivier Guichard, Fanny Kassel, and Anna Wienhard.
\newblock Anosov representations and proper actions.
\newblock {\em Geometry \& Topology}, 21(1):485--584, 2017.

\bibitem[Ghy87]{ghys}
Étienne Ghys.
\newblock Flots d'{Anosov} dont les feuilletages stables sont différentiables.
\newblock {\em Annales Scientifiques de l'École Normale Supérieure.
  Quatrième Série}, 20(2):251--270, 1987.

\bibitem[Gol87]{goldman_projective_1987}
William~M. Goldman.
\newblock Projective structures with {Fuchsian} holonomy.
\newblock {\em Journal of Differential Geometry}, 25:297--326, 1987.

\bibitem[GW12]{guichard_anosov_2012}
Olivier Guichard and Anna Wienhard.
\newblock Anosov representations: domains of discontinuity and applications.
\newblock {\em Inventiones mathematicae}, 190(2):357--438, November 2012.

\bibitem[IL]{ivey_cartan_2016}
Thomas~A. Ivey and Joseph~M. Landsberg.
\newblock {\em Cartan for beginners. Differential geometry via moving frames
  and exterior differential systems. 2nd edition}, volume 175 of {\em Graduate
  Studies in Mathematics}.
\newblock American Mathematical Society ({AMS}).

\bibitem[KLP14]{kapovich_morse_2014}
Michael Kapovich, Bernhard Leeb, and Joan Porti.
\newblock Morse actions of discrete groups on symmetric space.
\newblock {\em arXiv:1403.7671 [math]}, March 2014.

\bibitem[KLP16]{kapovich_recent_2016}
Michael Kapovich, Bernhard Leeb, and Joan Porti.
\newblock Some recent results on {Anosov} representations.
\newblock {\em Transformation Groups}, 21(4):1105--1121, 2016.

\bibitem[KLP17]{kapovich_dynamics_2017}
Michael Kapovich, Bernhard Leeb, and Joan Porti.
\newblock Dynamics on flag manifolds: domains of proper discontinuity and
  cocompactness.
\newblock {\em Geometry \& Topology}, 22(1):157--234, October 2017.

\bibitem[KLP18]{kapovich_morse_2018}
Michael Kapovich, Bernhard Leeb, and Joan Porti.
\newblock A {Morse} lemma for quasigeodesics in symmetric spaces and
  {Euclidean} buildings.
\newblock {\em Geometry \& Topology}, 22(7):3827--3923, 2018.

\bibitem[Lab06]{labourie_anosov_2006}
François Labourie.
\newblock Anosov flows, surface groups and curves in projective space.
\newblock {\em Inventiones Mathematicae}, 165(1):51--114, 2006.

\bibitem[Mas88]{maskit_kleinian_1988}
Bernard Maskit.
\newblock {\em Kleinian {Groups}}.
\newblock Grundlehren der mathematischen {Wissenschaften}. Springer-Verlag,
  Berlin Heidelberg, 1988.

\bibitem[MM20]{mmthese}
Martin Mion-Mouton.
\newblock {\em {Quelques propri{\'e}t{\'e}s g{\'e}om{\'e}triques et dynamiques
  globales des structures Lagrangiennes de contact}}.
\newblock Thesis, {Universit{\'e} de Strasbourg}, December 2020.

\bibitem[MM21]{mionmouton}
Martin Mion-Mouton.
\newblock Partially hyperbolic diffeomorphisms and lagrangian contact
  structures.
\newblock {\em Ergodic Theory and Dynamical Systems}, pages 1--47, 2021.

\bibitem[ST18]{stecker_domains_2018}
Florian Stecker and Nicolaus Treib.
\newblock Domains of discontinuity in oriented flag manifolds.
\newblock {\em {arXiv}:1806.04459 [math]}, 2018.

\bibitem[Tak94]{takeuchi}
Masaru Takeuchi.
\newblock Lagrangean contact structures on projective cotangent bundles.
\newblock {\em Osaka Journal of Mathematics}, 31(4):837--860, 1994.

\end{thebibliography}

\end{document}